\documentclass[12pt,a4paper]{amsart}
\makeatletter
\renewcommand\normalsize{%
    \@setfontsize\normalsize{11.7}{14pt plus .3pt minus .3pt}%
    \abovedisplayskip 10\p@ \@plus4\p@ \@minus4\p@
    \abovedisplayshortskip 6\p@ \@plus2\p@
    \belowdisplayshortskip 6\p@ \@plus2\p@
    \belowdisplayskip \abovedisplayskip}
\renewcommand\small{%
    \@setfontsize\small{9.5}{12\p@ plus .2\p@ minus .2\p@}%
    \abovedisplayskip 8.5\p@ \@plus4\p@ \@minus1\p@
    \belowdisplayskip \abovedisplayskip
    \abovedisplayshortskip \abovedisplayskip
    \belowdisplayshortskip \abovedisplayskip}
\renewcommand\footnotesize{%
    \@setfontsize\footnotesize{8.5}{9.25\p@ plus .1pt minus .1pt}%%
    \abovedisplayskip 6\p@ \@plus4\p@ \@minus1\p@
    \belowdisplayskip \abovedisplayskip
    \abovedisplayshortskip \abovedisplayskip
    \belowdisplayshortskip \abovedisplayskip}
\ifdefined\pdfpagewidth
\else
\fi
\calclayout
\makeatother

\usepackage{amsmath,amssymb,amsthm,mathtools,wasysym,calc,verbatim,tikz,url,hyperref,mathrsfs,cite,fullpage,bbm}
\usepackage[noabbrev,capitalize,nameinlink]{cleveref}
\usepackage[shortlabels]{enumitem}

\usepackage{comment}
\newtheorem{theorem}{Theorem}[section]
\newtheorem{lemma}[theorem]{Lemma}

\newtheorem{prop}[theorem]{Proposition}
\newtheorem{observation}[theorem]{Observation}

\newtheorem{claim}[theorem]{Claim}
\newtheorem{fact}[theorem]{Fact}
\newtheorem{definition}[theorem]{Definition}

\theoremstyle{definition}

\theoremstyle{remark}
\newtheorem{remark}[theorem]{Remark}
\newtheorem*{remark*}{Remark}

\newcommand\N{\mathbb{N}}
\newcommand\R{\mathbb{R}}
\newcommand\Z{\mathbb{Z}}

\newcommand\cB{\mathcal{B}}

\newcommand\cP{\mathcal{P}}

\newcommand\cU{\mathcal{U}}

\def\Pr{\mathbb{P}}
\def\S{\mathcal{S}}

\newcommand\eps{\varepsilon}

\renewcommand{\leq}{\leqslant}
\renewcommand{\geq}{\geqslant}
\renewcommand{\le}{\leqslant}
\renewcommand{\ge}{\geqslant}
\renewcommand{\to}{\rightarrow}

\def\eps{\varepsilon}

	\def\C{\mathbb{C}}
		
		\def\cF{\mathcal{F}}
		
	\def\R{\mathbb{R}}
		
	\def\Z{\mathbb{Z}}
	\def\N{\mathbb{N}}
	\def\PP{\mathbb{P}}
	\def\1{\mathbbm{1}}
	\def\l{\lambda}
	\def\k{\kappa}
	
	\def\s{\sigma}

	\def\g{\gamma}

	\def\la{\langle}
	\def\ra{\rangle}

	\def\Ber{\mathrm{Ber}}

	\def\tr{\mathrm{tr}}
    \def\supp{\mathrm{Supp}}

        \def\Span{\mathrm{Span\,}}

	\def\cB{\mathcal{B}}

	\def\EE{\mathbb{E}}

\newcommand{\snorm}[1]{\lVert#1\rVert}

\newcommand{\sang}[1]{\langle #1 \rangle}

\allowdisplaybreaks

\newcommand{\mb}{\mathbb}
\newcommand{\mbf}{\mathbf}
\newcommand{\mbm}{\mathbbm}
\newcommand{\mc}{\mathcal}

\newcommand{\mr}{\mathrm}

\newcommand{\ol}{\overline}
\newcommand{\on}{\operatorname}

\begin{document}
\addtolength{\footskip}{\baselineskip/2}

\title{The limiting spectral law for sparse iid matrices}

\author[A1]{Ashwin Sah}
\address{Massachusetts Institute of Technology. Department of Mathematics.}
\email{asah@mit.edu}

\author[A2]{Julian Sahasrabudhe}
\address{University of Cambridge. Department of Pure Mathematics and Mathematical Statistics.}
\email{jdrs2@cam.ac.uk}

\author[A3]{Mehtaab Sawhney}
\address{Columbia University. Department of Mathematics.}
\email{m.sawhney@columbia.edu}
\thanks{Sah was supported by the PD Soros Fellowship. Sah and Sawhney were supported by NSF Graduate Research Fellowship Program DGE-2141064. A portion of this work was conducted when Sawhney was visiting Cambridge with support from the Churchill Scholarship.}

\begin{abstract}
Let $A$ be an $n\times n$ matrix with iid entries where $A_{ij} \sim \Ber(p)$ is a Bernoulli random variable with parameter $p = d/n$ and $d$ constant.
We show that the empirical measure of the eigenvalues converges, in probability, to a deterministic distribution as $n \rightarrow \infty$. This essentially resolves a long line of work to determine the spectral laws of iid matrices and is the first known example for non-Hermitian random matrices at this level of sparsity.
\end{abstract}

\maketitle
\vspace{-2em}

\section{Introduction}\label{sec:introduction}
 For an $n\times n$ matrix $M$, define its \emph{spectral distribution} to be the probability measure on $\C$, which puts a point mass of equal weight on each eigenvalue $\l$ of $M$:
\[ \mu_M = n^{-1}\sum \delta_{\l}.\]
One of the central projects in random matrix theory, going back to the seminal 1958 work of Wigner \cite{Wig58}, is to determine the limiting spectral distribution of various random matrix models as the dimension tends to infinity.

While this area has enjoyed spectacular advances in the 80 years since its inception, several fundamental matrix models have eluded all attempts to understand their spectral law. Two major problems here concern very sparse matrices, in particular matrices with a constant number of non-zero entries in a typical row or column. The first is to show that $\mu_{M_n}$ tends to the oriented Kesten--McKay law as $n \rightarrow \infty$ when $M_n$ is an $n\times n$ matrix chosen uniformly at random from all matrices with exactly $d \in \N$ ones in each row and column (so called $d$-regular digraphs). The second is to show the existence of the limiting spectral distribution for iid Bernoulli random matrices with parameter $p = d/n$, for $d$ fixed.

In this paper we resolve this latter conjecture. As we will see,
this is the last piece in a complete understanding of the limiting spectral laws of iid random matrices and is the first time the existence of a limit law has been established for \emph{any} non-Hermitian random matrix model at this level of sparsity. In particular, this resolves a question highlighted by Tikhomirov in his 2022 ICM survey \cite[Problem~6]{Tik22}. 

\begin{theorem}\label{thm:main}
For $d > 0$, and each $n$, let $A_n$ be an $n\times n$ matrix with iid entries distributed as $\on{Ber}(d/n)$. There exists a distribution $\mu_d$ on $\C$ so that $\mu_{A_n}$ converges to $\mu_d$, in probability. 
\end{theorem}

Our proof differs significantly from previous approaches, such as \cite{RT19}; for example, we entirely avoid the use of $\eps$-nets. Rather, our approach is to ``build up'' the matrix, a row and column at a time,
and study the evolution of the point process defined by the singular values of the shifted matrices $A_n-zI$ as we add rows and columns. Our methods additionally give a considerably shorter proof of the difficult and celebrated theorem of Rudelson and Tikhomirov \cite{RT19} who proved the existence of the limiting spectral law in the case $pn \rightarrow \infty$. The details of this are contained in the sister paper \cite{SSS23}. 

We remark that the real problem here is for $d > 1$. In the ``subcritical'' and ``critical'' regimes, $d < 1$ and $d=1$, it is not hard to show that almost all of the eigenvalues of $A_n$ are $0$ and therefore $\mu_d = \delta_0$ (see Section~\ref{sec:graph-argument}). On the other hand, when $d > 1$, we expect that $\mu_d$ is a rich, non-trivial distribution. Thus $p = 1/n$ represents the threshold for the ``birth'' of the spectrum of $A_n$.

While in this paper we focused on the \emph{existence}
of the limiting distribution $\mu_d$, many properties of $\mu_d$ can be deduced, using other methods, now that it has been shown to exist. First, it is immediate from the proof that the limiting measure $\mu_d$ is the Brown's spectral measure of the adjacency operator of a directed Poisson Galton-Watson tree (see \cite{Sni02,BCC14} for further discussion regarding Brown measures). Furthermore one can deduce that $\mu_d$ is rotationally invariant, which incidentally falls out of some calculations we need in the course of proving  Theorem~\ref{thm:main} (see Lemma~\ref{lem:converge-singular}). The limiting distribution $\mu_d$ however is not expected to admit an simple ``algebraic'' form like the circular law. 

Let us also remark that our proof also can be adapted to the case where all non-zero entries are iid copies of a random variable $\xi$ with variance $1$ and with moments that decay sufficiently quickly. However, to keep this paper as streamlined as possible, we have elected not to work in this level of generality. 

In Theorem~\ref{thm:main}, and throughout this paper, we are concerned with convergence \emph{in probability}: 
a sequence of random measures $\mu_n$ converges \emph{in probability} to a probability measure $\mu$, if for all continuous bounded functions $f\colon\C \rightarrow \C$, and all $\eps >0$, we have 
\begin{equation}\label{def:measure-converge} 
\bigg|\int f~d\mu_n-\int f~d\mu \hspace{1mm} \bigg| < \eps
\end{equation}
with probability $1-o(1)$. If this holds we write $\mu_n\rightsquigarrow\mu$. It is also natural to consider the stronger notion of \emph{almost sure convergence} of $\mu_{A_n}$ to $\mu_d$, which is a problem we leave open for future work.

\subsection{The least singular value problem}

Before discussing the history of the limiting spectral laws for iid random matrices, we highlight a consequence of our results that is of independent interest and essentially resolves another question raised by Tikhomirov in his ICM survey \cite[Problem~7]{Tik22}. 

Here we are interested in proving that the spectrum of $A_n$ does not ``clump'' about a point  $z \in \C$. This ``clumping'' is captured in the extreme behaviour of the least singular value of the random \emph{shifted} matrices $A_n-zI$. Recall that for an $n\times n$ matrix $M$, its least singular value is 
\[ \sigma_n(M) =  \min_{v \in \S^{n-1} } \|M v\|_2, \]
where $\S^{n-1}$ denotes the unit sphere in $\R^d$. In this paper we prove the following ``qualitative'' estimates on $\sigma_n(A_n-zI)$, conjectured by Tikhomirov \cite{Tik22}.

\begin{theorem}\label{thm:shifted-main} 
Fix $d > 1$ and $\eps > 0$. Then for Lebesgue almost all $z$ we have the following. For each $n$, let $A_n$ be an $n\times n$ matrix with iid entries distributed as $\on{Ber}(d/n)$.  Then
\[\mb{P}\bigg(\sigma_n(A_n-zI)\le\exp(-\eps n)\bigg) = o(1).\]
\end{theorem}

In fact, a careful analysis of our proof reveals that we may take $\eps = n^{-1/2 + o(1)}$ and the $o(1)$ probability bound can be taken to be $(\log n)^{-\Omega(1)}$. However, these quantitative aspects are not the focus of this work and thus our methods are not tailored to this problem. We briefly discuss these dependencies in Remark~\ref{rem:quant}.

We also note that our proof can be modified to handle $d<1$ in the setting of Theorem~\ref{thm:shifted-main}, but we do not pursue the study of this sub-critical regime here, in the interest of brevity.

\subsection{History of the limiting spectral law for iid random matrices}

The project of determining the limit laws for random matrix models goes back to the seminal work of Wigner who proved the famous ``semi-circular'' law for random \emph{symmetric} matrices (or Wigner matrices). We let $M_n$ be an $n\times n$ random symmetric matrix with entries $(M_n)_{i\leq j}$ independent and uniform in $\{0,1\}$. Since symmetric matrices have all real eigenvalues it then makes sense to define, for $ a < b $, $N_{n}(a,b)$ to be the number of eigenvalues of $M_n$ in the interval $(a,b)$. Wigner's semi-circular law says     
\[ \lim_{n \rightarrow \infty} \frac{N_n(a\sqrt{n},b\sqrt{n})}{n} = \frac{1}{2\pi}\int_{a}^b(4 - x^2)^{1/2}_+\,dx \]
almost surely.

While these methods led to a very good understanding of symmetric and Hermitian random matrix models, determining the limiting spectral distribution for matrices with \emph{iid} (non-symmetric) entries 
proved to be substantially more difficult.  Here the first steps were
taken by Mehta \cite{Meh67}, who in the 1960s showed that when $A_n$ has iid complex Gaussian entries, the spectral distribution of $n^{-1/2}A_n$ converges to the uniform measure on the unit disc $\{ z \in \C \colon |z| \leq 1 \}$, the so-called \emph{circular law}. Mehta's proof relied deeply on the symmetries of complex Gaussian random variables and it was not until the 1990s that Edelman \cite{Ede97} managed to prove the  same result for \emph{real} Gaussian random variables. 

The case of more general coefficient distributions was studied in the 1980s by Girko \cite{Gir84}, who developed very influential ideas such as the ``hermitization'' technique, but his method relied on an unproven statement about the least singular value of iid matrices. This statement was then circumvented in the 1990s by Bai \cite{Bai97}, who extended the theorem of Edelman to matrices where the entries are iid mean $0$, variance $1$, and satisfy some smoothness and moment conditions. These results were then improved by G\"{o}tze and Tikhomirov \cite{GT10} and Tao and Vu \cite{TV08}, by using the method of Bai along with methods of Rudelson and Vershynin \cite{RV08} and Tao and Vu \cite{TV09} to control the least singular value. Finally, Tao and Vu \cite{TV10} proved the full ``universality'' theorem, showing the circular law holds for \emph{any} sequence of iid matrices with entries distributed as a mean $0$ variance $1$ random variable. Their method, relying on their breakthroughs in inverse Littlewood--Offord theory, provides a full understanding of empirical spectral distributions of ``dense'' random matrices.

\begin{theorem}[Tao and Vu] \label{thm:tao-vu-circular-law}
Let $\xi$ be a complex random variable with mean $0$ and variance $1$, let $A_n$ be a sequence of random matrices with iid entries distributed as $\xi$.
If we put $A^{\ast}_n = A_n \cdot n^{-1/2}$ then the spectral measure 
$\nu_{A_n^{\ast}}$ converges to the circular law almost surely\footnote{This notion of ``almost sure'' convergence is a stronger notion that implies convergence \emph{in probability}.}.
\end{theorem}

While this celebrated line of results gives us a very good understanding of the limiting spectral laws of dense matrices, it does not tell us anything about matrices where the non-zero entries are \emph{sparse}, as is often interesting in combinatorial settings. Of particular interest are Bernoulli random matrices: random iid matrices where all entries are Bernoulli random variables that take value $1$ with probability $p = p_n \rightarrow 0$ and $0$ otherwise. 

The spectral laws of such matrices were considered by G\"{o}tze and Tikhomirov \cite{GT10}, who proved that the limiting spectral law of $A_n$ is still the circular law (with appropriate normalization) for all $p > n^{-1/4+\eps}$. Tao and Vu \cite{TV08}  improved this range to $p > n^{-1+\eps}$, and Basak and Rudelson \cite{BR19} improved this range further to account for all $p > \omega(n^{-1}(\log n)^2)$.

Then, in an important and difficult paper, Rudelson and Tikhomirov \cite{RT19} extended these results to account for all $pn \rightarrow \infty$. This work is of particular interest since it is not hard to see that that the condition $pn\rightarrow \infty$ is \emph{necessary} for convergence to the circular law: for $pn$ bounded there is always an atom at zero.

\begin{theorem}[Rudelson and Tikhomirov]  Let $p = p_n$ be such that $pn \rightarrow \infty$ and $p \to 0$. For each $n$, let $A_n$ be an $n\times n$ matrix with iid entries distributed as $\on{Ber}(p)$. If we put $A^{\ast}_n = (pn)^{-1/2}A_n$ then $\mu_{A_n^{\ast}}$ tends to the circular law, in probability. 
\end{theorem}

In fact, they prove a more general result that allows for each non-zero entry to be a copy of an iid random variable $\xi$ with mean $0$ and variance $1$. We point the reader to \cite{RT19} or our sister paper \cite{SSS23}, where we give a simple proof of a variant of this theorem that subsumes all of these previous results. 

This leaves open what has proven to be the most difficult and subtle case, the case of $p = d/n$ for constant $d >0$. In this paper we complete this program by establishing the existence of the limiting spectral law for all $d$.

\section{Outline of Proof}\label{sec:outline}

In the companion paper \cite{SSS23} we adapt the methods of this paper to give a different and significantly shorter proof of the sparse circular law of Rudelson and Tikhomirov \cite{RT19}. In the setting of that paper we are able to avoid several significant difficulties that arise here and thus one may find it easier to absorb our method by first understanding our paper \cite{SSS23}. Of course, we will not assume any knowledge of \cite{SSS23} in our treatment here. 

In what follows we outline the proof of our main theorem. We start by discussing and motivating the general approach to this problem, via the logarithmic potential, which is now a well established ``first step'' in proving limiting laws. We then sketch how one is naturally led to consider the small singular values of the shifted matrices $A_n-zI_n$. We then go on to explain, in quite a bit of detail,  our novel ``dynamic'' method to studying these small singular values.  

For this, we first outline our proof in the case $p = d/n$, where $d$ is large but fixed. We then go on describe how we can adjust the proof so that it works for all fixed $d>1$. The subcritical case $d\leq 1$ is actually \emph{much} easier and really should be thought of as a different result. Indeed, we handle this case with a completely different and more direct proof in Section~\ref{sec:graph-argument}.  Throughout our discussion here, we give pointers to the relevant places in the text so the reader can find the precise statements and proofs.

\subsection{Convergence of the logarithmic potential}\label{sec:sketch-girko-step}
We start by discussing the standard method for proving the existence of a limit law, by way of the logarithmic potential and singular value estimates. Along the way, we will introduce a few of the central objects of study that we will use throughout the paper. Formally, much of what we discuss here, in Section~\ref{sec:sketch-girko-step}, is packaged as Proposition~\ref{prop:unicity}, which we borrow from the work of Bordenave and Chafa\"{\i} \cite{BC12}.

To establish the convergence of the spectral law, it is enough to prove the point-wise almost everywhere convergence of
the logarithmic potential of the spectral law, which is the (random) function
\begin{equation}\label{eq:sketch-Un} U_n(z) = - \frac{1}{n}\sum_{\l} \log|\l - z| \, ,  \end{equation}
where the sum is over the eigenvalues $\l$ of our random matrix $A_n$.  In fact, we already arrive at a challenge that separates this problem from case when $pn \rightarrow \infty$. When proving that the limiting law is \emph{circular}, one has the very nice target potential $U^{\circ}(z)$ (i.e. the logarithmic potential of the circular law) to prove convergence to. In our case, \emph{a priori}, we don't know the form for our limiting object or if it has nice properties. 

Now, to work with \eqref{eq:sketch-Un}, we apply Girko's ``hermitization'' method (see e.g.~\cite{BC12}) to express the logarithmic potential in terms of the singular values of the shifted random matrices $A_n - zI_n$
\begin{equation} \label{eq:girko}
U_{n}(z) =  -\frac{1}{n} \sum_{j = 1}^{n} \log\big(\sigma_j\big(A_n-zI_n\big)\big).
\end{equation} 
Here we are using the notation 
$\sigma_1(M) \geq \cdots \geq \s_m(M)$ to denote the (right) singular values of the $n \times m$ matrix $M$.

The big advantage of the expression \eqref{eq:girko} is that it is in terms of  \emph{singular values}, rather than eigenvalues, which have the advantage that they are the eigenvalues of the Hermitian matrix $(A_n-zI_n)^\dagger(A_n-zI_n)$, and are therefore real numbers. Thus, if we define the measure
\begin{equation}\label{eq:def-nu} \nu_{z,n} = n^{-1} \sum \delta_{\sigma}, \end{equation}
where the sum is over the singular values $\s \in \{ \s_n(A_n-zI_n),\ldots,\s_1(A_n-zI_n) \}$, we can recover the bulk behavior of $\nu_{z,n}$
by simply computing the trace moments of $(A_n-zI_n)^\dagger(A_n-zI_n)$ and using the identity
\begin{equation}\label{eq:sketch-trace-moments} n^{-1}\EE\, \big((A_n-zI_n)^\dagger(A_n-zI_n)\big)^k = n^{-1} \sum_{j=1}^n (\sigma_j(A_n-zI_n))^k = \EE_{\sigma \sim \nu_{z,n}} \, \sigma^k, \end{equation}
where the last expectation is over a random sample of $\sigma \sim \nu_{z,n}$. Thus, via standard methods, we can conclude that, for each $z$, there exists a deterministic measure $\nu_z$ for which 
\[ \nu_{z,n} \rightsquigarrow \nu_z, \] in probability, as $n\rightarrow \infty$. This means that almost all of the terms in the sum \eqref{eq:girko} converge to a well defined limit. We point the reader to Section~\ref{sec:singular-convergence} for precise statements of these results.

While this is a good (and far from novel) first step, this \emph{does not} imply the convergence of the log potential $U_n(z)$ since 
it is (\emph{a priori}) possible for even a single term in the sum \eqref{eq:girko} to disrupt the convergence of $U_n(z)$. There are two ways this can happen. The first way is if the largest term $\sigma_1(A_n-zI_n)$ blows up. But luckily there is little worry of this happening. In fact we expect $\sigma_1(A_n-zI_n) = O(\sqrt{n} + |z|)$, by standard estimates so the logarithmic contribution of this term is negligible. 

The real concern is if there are abnormally \emph{small} singular values. To exclude this possibility, one traditionally needs to prove estimates of the general shape 
\begin{equation}\label{eq:sing-val-bound}
\PP\big( \sigma_{n-k}(A_n-zI) \leq \exp(-\eps n/k ) \big) = o(1),
\end{equation}
for any $\eps>0$, all $k \leq n-1$, and for Lebesgue almost every\footnote{Throughout the paper we will say ``almost all $z \in \C$'' to mean ``all $z \in \C$, apart from a set of Lebesgue measure zero''.} $z\in\C$. Indeed, \emph{heuristically} we typically have $\sigma_{n-k} = \Theta( k d^{1/2} n^{-1})$ and thus \eqref{eq:sing-val-bound} perhaps appears to be an easily surmounted obstacle, as it represents (what we expect to be) extremely abnormal behaviour. However, obtaining bounds of this type has recently been \emph{the} significant challenge in this area. For example, proving \eqref{eq:sing-val-bound} in the case of dense matrices was one of the principle achievements of the work Tao and Vu \cite{TV10} in their work on the circular law. For sparse matrices, the challenge is greater still as there is less ``randomness'' to use. For their sparse circular law, Rudelson and Tikhomirov \cite{RT19} develop a whole toolbox of sophisticated techniques to prove singular value estimates of the type\footnote{In fact, both Tao and Vu \cite{TV10} as well as Rudelson and Tikhomirov \cite{RT19} obtain better quantitative estimates on the least singular value than required, but this is not relevant to our discussion here.} \eqref{eq:sing-val-bound}.
However, these techniques are of limited use in the context of this paper and we are thus led to develop a different approach.

\subsection{The evolution of windows of singular values} In this paper, we don't directly look to prove \eqref{eq:sing-val-bound} but instead look to control the bottom \emph{window} of singular values
\begin{equation}\label{eq:W0-def} W_{n,0}(z) = - \frac{1}{n}\sum_{j=0}^{\delta n} \log\big( \s_{n-j}(A_n-zI_n) \big), \end{equation} for almost all $z$, where $\delta >0$ is chosen to be sufficiently small relative to $d,z,\eps$. 
And in place of \eqref{eq:sing-val-bound}, we shall show that for all $\eps > 0$ and almost all $z\in \C$, we have that 
\begin{equation}\label{eq:window-goal} \PP\big( W_{n,0}(z) \leq \eps   \big)  = 1-o(1),\end{equation}
as $n$ tends to infinity. (We point the reader to Lemma~\ref{lem:crucial} for a formal statement.)

While this, so far, is not much of a departure from the task of proving \eqref{eq:sing-val-bound}, our principal difference comes from how we approach $W_{n,0}(z)$, which is fundamentally dynamic: we build up the randomness in the matrix bit by bit and track how the singular values evolve. 

To best explain this, we first outline our method in the case $p=d/n$ where $d$ is large but fixed. We then go on to describe a more complicated revelation process, in Section~\ref{sec:extending-to-threshold}, that works for all $d$ down to the threshold $d>1$.

Now, we let\footnote{Throughout the paper we will write $\delta \ll \eps$ to mean, given $\eps>0$, we can choose $\delta>0$ to be sufficiently small in terms of $\eps$.} $\eps \ll 1/d$, and define $m = (1-\eps)n$ and reveal the top left $m\times m $ sub-matrix\footnote{Here we let $A_{s,t}$ denote the submatrix of $A_n$ defined by $(A_{ij})_{i\in [s],j \in [t]}$. We define $A_m = A_{m,m}$.} $A_m = A_{m,m}$ of $A_n$. We will then ``build up'' the matrix $A_n$ by alternately adding rows and columns until we fill out all of $A_n$,
\[ A_{m,m} \rightarrow A_{m,m+1} \rightarrow A_{m+1,m+1} \rightarrow \cdots \rightarrow A_{n,n}. \] 
While it might be most intuitive to try to directly control \eqref{eq:W0-def} as this process evolves, it is very unclear how to do this in practice. Instead, we control a window of larger eigenvalues and then incrementally ``slide'' this window downward, as the process evolves. 

More precisely, we define the \emph{window} $W_{t,r}(z)$ of singular values of $A_t-zI_t$ \emph{at height} $r$ to be the sum
\begin{equation}\label{eq:window-def} W_{t,r}(z) = - \frac{1}{n}\sum_{j=r}^{r+\delta n} \log\big( \s_{t-j}(A_t-zI_t) \big),\end{equation}
thus generalizing the definition of $W_{n,0}$ at \eqref{eq:W0-def}. We then initialize our process by showing that for all $\eps>0$ and almost all $z \in \C$ there exists $\delta = \delta(z,\eps,d)$ so that 
\begin{equation}\label{eq:initial-bound} W_{m,r_0}(z) \leq \eps^4/2 \qquad \text{ where }  \qquad r_0 = \eps^4m , \end{equation}
with high probability. (We point the reader to Lemma~\ref{lem:initial-step}, for the formal statement of this initialization step.)

We then run our process, exposing rows and columns one at a time. We will show that at each step of the process, with reasonably high probability, we have
\begin{equation}\label{eq:window-gets-pushed-0}
     W_{t+1, r-1}(z) \leq W_{t,r}(z)  + \eta_{r},
\end{equation} where the $\eta_r>0$ are positive numbers with the property that 
\begin{equation}\label{eq:prop-eta_r}  \sum_{r=0}^{r_0} \eta_r = o(1).\end{equation}
Thus if our current window is at height $r$ and \eqref{eq:window-gets-pushed-0} holds, we slide our window downward by $1$, (i.e. $r \rightarrow r-1$) thus bringing us closer to our goal of controlling the bottom window. Thus our goal is to show that at the end of the process, we will have slid our window all the way to $r=0$, allowing us to conclude that 
\[ W_{n,0}(z) \leq W_{m,r_0}(z) + \sum_{r=1}^{r_0} \eta_r  = o_{\eps \rightarrow 0} (1) ,\]
with high probability, for almost all $z$, as desired. 

\subsection{The random walk}\label{sec:sketch-random-walk}
This process defines a random walk on the height of our window, which we now turn to describe in a little more detail. Let us index time in our process by $t \in [m,n]$, so that at time $t$ we have exposed $A_{t}$. When we progress from $t$ to $t+1$, we add a column to $A_{t}$ and then a row, so that we have filled out $A_{t+1}$ by the end of the step. Define the \emph{height} of the process at time $t$ as the random variable $X_t$, where we initialize $X_m = r_0 = \eps^4m$ (cf. \eqref{eq:initial-bound}). Now if $X_t = r>0$, we have 
\begin{equation}\label{eq:window-gets-pushed}
     W_{t+1, r-1}(z) \leq W_{t,r}(z)  + \eta_{r}, 
\end{equation} 
we then set $X_{t+1} = X_{t}-1$. Otherwise, if \eqref{eq:window-gets-pushed} fails, we use the (trivial) fact that
\[ W_{t+1,r+1}(z) = W_{t,r}(z)  , \]
and set $X_{t+1} = X_{t}+1$. (Actually the definition of of our random process is slightly more complicated and we refer the reader to the proof of Lemma~\ref{lem:crucial} for the precise definition).

Thus, we are led to study the random walk defined by $X_t$. In particular, we look to show that 
\begin{equation}\label{eq:X_nis0} \PP\big( X_n = 0 \big) = 1-o(1), \end{equation} from which, we can conclude \eqref{eq:window-goal}.

There are two main ingredients in proving this. The first is to show that our random walk can at least get close to $0$ towards the end of the process. For this, it is enough to show that this random walk has a significant downward drift in each step. We show that 
\begin{equation} \label{eq:window-push} \PP\big(\hspace{0.5mm} W_{t+1, r-1}(z) \leq W_{t,r}(z)  + \eta_{r}\big) \geq 1-c_d,\end{equation}
where we can make $c_d \geq 0$ arbitrary small by making $d$ large. We can now use a standard martingale analysis: If each step $i$ has a downward drift close to $1$
(assuming $X_i >0$ and $d$ is large) and there are $\eps^3 n$ steps total, we have a total downward drift close to $\eps^3n$, which easily allows us to traverse the distance to the bottom window of singular values, since we started at height $r_0 = \eps^4n$. This allows us to conclude that
\[ \PP( X_n = O(1) ) = 1 -o(1), \]
and thus our random walk gets close to $0$ towards the end of the process. 

Before going on to describe how we ensure that our random walk sticks to $0$, we mention that the proof of \eqref{eq:window-push} is the most technical element of the paper and accounts for much of its bulk. We will discuss it quite a bit more in Section~\ref{sec:sketch-enough-drift}.

The second step towards proving $X_n = 0$, with high probability, requires us to take advantage of another element of our set up. 

\begin{comment}
Which is close to  we want, but crucially falls short. Moreover, in a sense, this analysis is actually sharp; one cannot hope that the random walk sticks to $0$, once reaching it. However, we are able to get around this issue by taking advantage of another symmetry that we have not yet mentioned.  

The proof of this is the most technical element of the paper and we will discuss it quite a bit more in Section~\ref{sec:sketch-enough-drift}. But for now we assume it as see why that this ensures our walk must be close to $0$.

Now, since we have $\eps n$ steps in our process, each step has downward drift of nearly $1$ (making use of the assumption that $c_d$ is large), and our starting point is $X_m = \eps^4m < \eps^4n$, a standard martingale analysis reveals that we have sufficient downward drift to ensure 
\[ \PP( X_n = O(1) ) = 1 -o(1). \]

\end{comment}

\subsection{The end of the process}\label{sec:end-of-the-process}
For an iid Bernoulli matrix, with $p = d/n$, the number of ones in a row or column is approximately distributed as a Poisson random variable with parameter $d$. Thus, the number of $i\in [n]$ for which the $i$th row and $i$th column both have exactly $k$ ones should be about 
\[ n\cdot \PP( \on{Po}(d) = k )^2 = n \cdot (e^{-d}/k!)^2,\] with high probability. Thus we expect many row/column pairs with significantly more than the average number of ones in them. The key idea is that, by the symmetry of rows and columns, we can arrange for these heavy rows/column pairs to appear at the very end of the process and ``push'' our random walk to $0$.

In practice, we show that we can couple our matrix $A_n$ to a random matrix $B_n$ which is identically distributed to $A_n$ apart from the fact that the last $\ell = (\log n)^2$ rows and columns are independent of the other entries and are iid Bernoullis with parameter\footnote{These numbers are chosen somewhat arbitrarily - we really just need $\omega(1)$ rows with $\omega(1)$ ones in each, but we have chosen to work concretely).} $\tau = \sqrt{\log n}/n$. Thus we can work almost entirely with the matrix $B_n$ and ensure that the last row/column pairs have lots of ones, with high probability. 

Intuitively, this is possible since the standard deviation of the number of such heavy row/column pairs is $n^{1/2-o(1)}$ and thus forcing an additional $\ell = (\log n)^{2}$ such pairs only moves the distribution by a small amount in total variation

Thus, after replacing $A_n$ with $B_n$, we can upgrade \eqref{eq:window-push} to the stronger estimate
\begin{equation}\label{eq:window-gets-pushes-heavy} \PP_B\big( W_{t+1,r-1}(z) \leq W_{t,r}(z) + \eta_r \big) \geq 1-(\log n)^{-\Omega(1)}, \end{equation} for $t \in [n-\ell,n]$. This means 
\[\PP_B( X_{t+1} = X_t-1 ) \geq 1-(\log n)^{-\Omega(1)},\] for $t \in [n-\ell,n]$ and when $X_t >0$, which allows us to upgrade our initial bound of $\PP( X_n = O(1)) = 1-o(1)$ to $\PP(X_n=0) = 1-o(1)$, exactly as we wanted. 

Formally this coupling step is Lemma~\ref{lem:TV-estimate}, which allows us to work with the random matrix $B = B_n$, rather than $A = A_n$, for much of the paper. It will be quite common for us to deal with the two \emph{epochs} of the process
\[ t \in [m,n-\ell] \qquad  \text{ and } \qquad t \in [n-\ell,n] \] slightly differently. Thus we refer to these two intervals of time as the \emph{first} and the \emph{second epochs}, respectively.

\subsection{Extending to the threshold}\label{sec:extending-to-threshold}
In our discussion above we were assuming $d$ was sufficiently large throughout. We now turn to discuss how to adjust the proof that we sketched above to accommodate all $d$ down to the threshold $d > 1$. 
% It makes sense to address this first, since our strategy here will influence the structure of the paper.

Now, if we naively attempt to run the process defined above for $d \sim 1$, we quickly run into a problem; the drift at \eqref{eq:window-push} is not necessarily sufficient to ensure that our random walk satisfies $X_n = O(1)$, with high probability. 

To get around this, we start by noting that we are only using the final $\eps^4m$ rows and columns in our process. Thus, taking inspiration from the idea of rearranging the rows and columns that we saw above, we try to simulate a matrix with larger $d$ by moving the heaviest rows and columns to the end of the process. Recalling that we expect about $n \cdot e^{-2d}/(k!)^2$ row/column pairs with $k$ ones, it seems reasonable that we could arrange for the last $\eps^4m $ row/columns pairs to contain $(\log 1/\eps)^{1-o(1)}$ ones, thus allowing us to work as if the graph we encounter in our process has degree $\gg d$. (Recall $\eps>0$ can be chosen arbitrarily small here).  
The mechanism for moving these heavy rows and columns to the end of the process is a little different from what we described above and more similar to an idea used in the recent papers~\cite{FKSS23,GKSS23}. 

To describe this a little more carefully, we think of our matrix as a directed graph on $[n] = \{1,\ldots, n\}$ where $(i,j)$ is an edge if the $ij$th entry is $1$. We then begin our exposure process by setting aside the vertices $[(1-\eps)n]$, corresponding to the top left principal submatrix. We then expose the quantities
\begin{equation}\label{eq:sketch-deg-def} \deg^{+}\big(j, [n-\ell]\hspace{0.3mm}\big) = \big|N^{+}(j) \cap [n-\ell]\big|  \quad \text{ and } \quad \deg^{-}\big(j,[n-\ell]\hspace{0.3mm}\big) = \big|N^{-}(j) \cap [n-\ell]\big|,\end{equation}
for all $j \in [(1-\eps)n, n-\ell]$. (Here we are still setting aside the last $\ell = (\log n)^2$ steps, as discussed in Section~\ref{sec:end-of-the-process}.) We then define the \emph{value} of a vertex $j \in [(1-\eps)n, n-\ell]$ to be the minimum of these degrees in \eqref{eq:sketch-deg-def}.

We then move the $\eps^3n$ vertices among $[(1-\eps)n, n-\ell]$ with the largest value to the end of our ordering. Crucially these vertices will have value $\geq (\log 1/\eps)^{1-o(1)}$, with high probability. It is important that we are not revealing the actual rows and columns (or equivalently the neighborhoods of vertices) here and only these partial row/column sums (equivalently the degrees of vertices). Thus there is still quite a bit of randomness to expose in the running of the process. 

We will then prove a variant of \eqref{eq:window-push} of the form
\begin{equation} \label{eq:window-push-modified} \PP\big( W_{t+1, r-1}(z) \leq W_{t,r}(z)  + \eta_{r}\big) = 1-(\log 1/\eps)^{-\Omega(1)}, \end{equation}
which we can make arbitrary close to $1$ by making $\eps$ small. We note here that the probability in \eqref{eq:window-push-modified} is over a random sample of a matrix with the given degree sequence (equivalently partial row/column sums) and thus we are led to work with a configuration model for a random digraph with a given degree sequence to estimate these probabilities.

Thus, with \eqref{eq:window-push-modified},  we recover the same behavior we saw for $d$ large in \eqref{eq:window-push}, thereby allowing us to conclude that $X_{n-\ell} = O(1)$, with high probability. We then boost this to $X_n = 0$, with high probability, in the same way we saw above. Indeed, we have already arranged that the last $\ell = (\log n)^2$ rows and columns are heavy in $B_n$ and thus push our random walk to zero.

\subsection{Obtaining sufficient downward drift}\label{sec:sketch-enough-drift} So far we have only discussed our strategy in the broadest of terms, defining our process and sketching how we expect it to evolve. In what follows we elaborate on the key technical ingredients that make this strategy work. We will focus on a discussion of the proof of \eqref{eq:window-gets-pushes-heavy} and \eqref{eq:window-push-modified}, which will consume our attention for the bulk of this paper. 

To prove these statements, we start with a purely deterministic result which relates a window of singular values with an adjacent window in a matrix with a row added. Let $B_t$ be our $t \times t $ matrix at time $t$, and let $B_{t,t+1}$ be the matrix obtained by adding the row $X$ to $B_t$. (Recall that we are now working with the matrices $B_t$ after the replacement step described in Section~\ref{sec:end-of-the-process}). We prove that 
\begin{equation}\label{eq:lin-alg-prod} -\frac{1}{n}\sum_{i=r}^{r+ b} \log \s_{t-i}(B_{t,t+1}) \leq  -\frac{1}{n}\sum_{i=r+1}^{r+1+b}\log \s_{t-i}(B_{t}) + f(B_t,X), \end{equation}
where $b = \delta n$ and the quantity $f$ is defined by
\begin{equation}\label{eq:def-f(M,X)} f(B_t,X) = \frac{1}{2n}\log\bigg(\frac{\snorm{X}_2^2 + \sigma_{t-(r+1+b)}(B_t)^2}{\|P X \|_2^2}\bigg),\end{equation}
where $P$ is the orthogonal projection onto the span of the $r$ smallest right-singular vectors of $B_t$. (We point the reader to Proposition~\ref{prop:walk-row-modified} for the formal statement of this result). 

We remark here that what we have described in \eqref{eq:lin-alg-prod} represents only a half-step. Thus we also have a very similar inequality relating a window of singular values of $A_{t,t+1}$ with the singular values of $A_{t+1}$. Putting these together constitutes a full step in the process.

Thus \eqref{eq:lin-alg-prod} reduces our task to showing that it is unlikely
for the projection of $X$ onto the $r$ smallest singular directions of $M$ to be extremely small -- ignoring the terms $\|X\|_2$ and $\sigma_{t-(r+1+b)}(B_t)$ in \eqref{eq:def-f(M,X)}, which are relatively easy to bring under control. While this is not true for a deterministic matrix $B_t$
and a random row $X$, we will be able to show that if $B_t$ satisfies some particular quasi-randomness conditions (which we will come to define shortly)
we can obtain such a result. In particular we will prove the following key anticoncentration estimate 
\begin{equation}\label{eq:ker-anti-concentration}
\PP_X\big( \|P X\|_2 < \exp(-n\eta_r) \big) = o_{\eps \rightarrow 0}(1),
\end{equation}
in the first epoch and the anticoncentration estimate
\begin{equation}\label{eq:ker-anti-concentration-2}
\PP_X\big( \|P X\|_2 < \exp(-n\eta_r) \big) = O((\log n)^{-\Omega(1)}),
\end{equation}
in the second epoch, assuming that $B_t$ satisfies the relevant quasi-randomness conditions in each. 

Now, the fact that we have reduced the problem of proving a spectral law to proving an anticoncentration estimate, will not come as a surprise to those familiar with this area. Indeed, proving statements of the general shape \eqref{eq:ker-anti-concentration} amounts to one of the key technical challenges in this area. However, what \emph{is} perhaps surprising is that the anticoncentration result we need here is vastly weaker than what is required in other examples in the literature. Indeed, in other papers in this area, anticoncentration estimates have to be sufficiently strong so that one can sum them over an $\eps$-net of a $n$-dimensional sphere. Moreover the ``quasi-randomness'' condition that we use will be a relatively weak and flexible notion of graph expansion and will be easily proved to hold with high probability (modulo some computations with the configuration model).

\subsection{Unique neighbourhood expansions}\label{sec:sketch-une}
We now turn to discuss our approach to proving \eqref{eq:ker-anti-concentration} and \eqref{eq:ker-anti-concentration-2}. This will lead us to define the notion of ``unique neighbourhood expansions'' which we will then, in turn, use to define our important quasi-randomness properties $\cU(r)$. 

But before we get to this, we pause to note that an estimate of the form \eqref{eq:ker-anti-concentration} need not hold in general. Indeed if $P$ is the matrix with rows $e_1,\ldots, e_r$ (i.e. the projection matrix onto the first $r$-coordinates), then 
\begin{equation}\label{eq:sketch-bad-subspace-eg} \PP_X(\|PX\|_2 = 0 ) \geq \PP_X( X_1 = 0, \ldots, X_r = 0 ) = (1-p)^r ,\end{equation} which is  $1-o(1)$ whenever $r = o(n)$. (Here we are ignoring the shift $z$ on the diagonal for the moment). So for bounds of the form~\eqref{eq:ker-anti-concentration}, \eqref{eq:ker-anti-concentration-2} we need to use some additional information about our matrix $B_{t}$ to exclude these bad examples. Here we introduce a simple notion based on the expansion properties of the underlying directed graph.

%Of course we don't expect that in our application our projection $P$ will look like this; it is defined by the singular directions of a random matrix. So, intuitively speaking, should not be too structured. We formalize this notion by first introducing a notion of the underlying directed graph, which we call the ``unique neighbourhood expansion''. 
%\J{not sure how I want to parameterize the below}
In particular, let $M$ be an $m\times s$ matrix and let $S\subset [s]$ be a set of columns of $M$. We define the \emph{unique neighbourhood} of $S$, denoted $U(S)\subset [m]$, as follows. We first define 
\[ U(S) \setminus S = \big\{ i \in [m]\setminus S\colon B_{ij} = 1 \text{ for a unique } j \in S\big\} \]
and then define
\[ \hspace{-2.5em} U(S) \cap S = \big\{i\in[m]\cap S\colon B_{ij} = 0 \text{ for all } j \in S \big\}. \] 
Our simple, yet essential, quasi-randomness property is that the unique neighbourhoods $U(S)$ have to be ``large'' for all sets $S$ of appropriate size. More precisely, we define $\alpha(x) = (\log(n/x))^{-2}$ and then define our quasi-randomness event $\cU(r)$, for $r \in [n]$, to be 
\begin{equation}\label{eq:sketch-une} |U(S)|\ge\alpha(|S|)|S|,\end{equation}
for all sets $S$ with $r\leq |S| \leq \kappa n$, where $\eps \ll \kappa \ll 1/d$. 

%We will expand on the exact use of this property in the next subsection, it is not hard to see how this property might come in handy. 
As we will discuss more in the following subsection, the property $\cU(r)$ will be used to show that vectors that are ``close'' to the kernel of $B_t-zI_t$ spread their mass over their coordinates. While it is a bit harder to see this and how this fits into the context of proving \eqref{eq:ker-anti-concentration}, it is not hard to see why something like this might be true. Indeed, let us first (carefully) observe that if $M$ is an $n\times n$ matrix for which $U(r)$ holds and $u \in \C^n$ then 
\[ |\supp((M-zI_n)u)| \geq |U(\supp(u))|,\]
when $z\not=0$. Thus if $u$ is actually \emph{in} the kernel of $M-zI_n$, and $|\supp(u)|\geq r$ then we can actually conclude that $|\supp(u)|\geq \kappa n$. That is, the mass of $u$ is ``spread'' or ``unstructured'', in a sense.

In Section~\ref{sec:UNE} we formally define these events and show that the events 
\[ \cU((\log n)^{3/2}) \qquad \text{ and } \qquad \cU((\log\log n)^2),\] hold with high probability for the matrices $B_t$ in the first and second epochs of the process, respectively. 

In fact, our strategy for proving that the events $U(r)$ hold is quite straightforward; we simply show that the probability that a set $S$ fails \eqref{eq:sketch-une} is small enough that we can union bound over all appropriate sets $S$. In practice, this is a little tricky since we are forced to work with the configuration model rather than with iid samples in the first epoch of the process. As we discussed above, this is because we have already revealed the row and column sums when we were rearranging the matrix (as we discussed in Section~\ref{sec:extending-to-threshold}).

\subsection{Spreadness of near kernel vectors}\label{sec:sketch-spreadness}

We now sketch how the unique neighborhood expansion is linked to the probability in \eqref{eq:ker-anti-concentration}. This is formally carried out in Section~\ref{sub:unstructuredness}, where we prove near-kernel vectors are unstructured (which we discuss below), and in Section~\ref{sec:anticoncentration}, where we use this unstructured property to prove~\eqref{eq:ker-anti-concentration}.

We recall that the projection $P$ in \eqref{eq:ker-anti-concentration} is the orthogonal projection onto the span of the eigenvectors corresponding to the $r$ smallest singular values of $B_{t}$, where $r$ is the current window height. Without loss, we can assume that these singular values are reasonably small, and in particular, 
\[ \sigma_t(B_t-zI_t) \leq \cdots \leq \sigma_{t-r+1}(B_t-zI_t) < \exp(-n\eta_r) =: \eps_r, \]
otherwise we could have pushed our window down further. We let $v_1,\ldots ,v_r$ be the corresponding singular vectors and define $V_t = V(B_t,r)$ to be the linear span of the $v_i$. Thus, for any unit length $u \in V_t$ we have 
\begin{equation}\label{eq:near-kernel-cond} \big\| (B_{t} - zI_t)u \big\|_2 < \eps_r.\end{equation}
This last expression is quickly seen to be relevant since if $u_1,\ldots u_r$ is any orthonormal basis of $V_t$, we see that 
\[ \|PX\|_2^2 = \la u_1,X \ra^2 + \cdots +\la u_r,X\ra^2 \]
and thus the anticoncentration properties of $\|PX\|_2$ are determined by the anticoncentration properties of the vectors $u_i$, which, by \eqref{eq:near-kernel-cond}, are near-kernel vectors of $B_t$.

We then will prove that $\|PX\|_2^2$ is anti-concentrated 
in two slightly different ways, depending on the epoch, and resulting in the estimates \eqref{eq:ker-anti-concentration} and \eqref{eq:ker-anti-concentration-2}, respectively. 

In the first epoch, we can assume $r \geq (\log n)^{3/2}$ and, in this case, we use a result of Litvak, Lytova, Tikhomirov, Tomczak-Jaegermann, and Youssef (see Lemma~\ref{lem:basis}) to find a basis $u_1,\ldots ,u_r$ of $V_t$ where the first $r$ entries of each vector $u_i$ have absolute value $\Omega(r^{-1/2}n^{-1})$. We can then use this, along with \eqref{eq:near-kernel-cond} and the property $\cU((\log n)^{3/2})$, to learn that none of the $u_i$ are close to being constant. We prove that for all $i$, we have 
\begin{equation}\label{eq:sketch-spreadness}
\sup_{\theta \in \R} \big|\big\{j \colon| (u_i)_j-\theta|\le n^{-1/2}\exp(-c(\log 2n/k)^7)\big\}\big|\le(1-c)n,\end{equation}
for some $c>0$. (See Proposition~\ref{prop:unstructured-main-1} for a formal statement.)

We then use this result in Section~\ref{sec:anticoncentration}, along with a fairly standard anticoncentration estimate, (see Lemma~\ref{lem:slice-anti}) to prove \eqref{eq:ker-anti-concentration} holds in the first epoch. We remark that here \eqref{eq:sketch-spreadness} is the natural condition for a minimal anticoncentration estimate since in the first epoch the row and column sums are already determined. Thus if $X$ is a random row or column from this epoch, $\la \1 , X\ra$ is constant.

In the second epoch, the rows and columns we add consist of iid $\Ber(\tau)$ entries, where $\tau = \sqrt{\log n}/n$ and so we can get away with a weaker spreadness condition. Here we can assume that our window height $r$ is typically $O(1)$ and thus $\dim V_t = O(1)$. In this case, we simply focus on a single unit vector $u\in V_t$ and use \eqref{eq:near-kernel-cond} along with the property $\cU(\log \log n)$ to ensure that $u$ is relatively well spread on its support. Formally, we show
\[ \big|\{ i : |u_i| \ge n^{-1/2} \cdot \exp(-C(\log n)^7) \big\}\big| \geq cn ,\]
for some constants $C,c>0$ (see Propositions~\ref{prop:unstructured-main-2} and \ref{prop:unstructured-main-3}). We then use this, in Section~\ref{sec:anticoncentration}, to prove 
\[
\PP_X\big( \|P X\|_2 < \exp(-n\eta_r) \big) = (\log n)^{-\Omega(1)},\]
in the second epoch. We actually require a secondary quasi-randomness property $\cU^{\ast}$ for this epoch, but since its definition and use are similar to the above, we don't describe this in any more detail.

\subsection{Initializing the process}\label{sec:sketch-initializing}

So far, we have given a fairly complete sketch of the proof of \eqref{eq:window-push-modified}, which allows us to say the random walk has enough downward drift to get pushed, and then eventually stick to $X_n = 0$. We have not said anything, however, of how we initialize our process. That is, we have said nothing about how we assume control of the \emph{first} window. This step is carried out in Section~\ref{sec:non-atomic-singular} and is of a different flavor from the rest of the proof. It relies on an simple analytic idea, which we take a moment to sketch here. 

Here the problem reduces to showing that, for almost all $z \in \C$, the singular value measure of $B_m-zI_m$, which we denote by 
\[ \nu_{z,n}' = \frac{1}{m} \sum \delta_{\sigma}, \]
tends to a deterministic limit, $\nu_{z,n}' \rightsquigarrow \nu_z'$ in probability, and moreover this limit is non-atomic at $0$; meaning, $\nu'_z(\{0\}) = 0$. Indeed, if we have this, we can ensure that all of the singular values in the sum $W_{m,\gamma n}(z)$, for any fixed $\g>0$, are bounded away from $0$. And from this, we can conclude that 
\begin{equation}\label{eq:Wmgn} W_{m,\gamma n}(z) = o_{\delta \rightarrow 0}(1), \end{equation} for almost all $z$, as desired.

To prove that $\nu_z'$ is non-atomic at $0$, for almost all $z$, we work by contradiction. If $\{ z : \nu'_m(\{0\}) > 0 \}$ has positive measure, there exists $\g>0$ and a $\rho>0$ so that  
\[ S_{\g,\rho} = \big\{ z : |z| = \rho \text{ and } \nu'_{z}(\{0\}) > \g   \big\} \] has positive measure relative to the uniform probability measure on the circle $\{ z : |z| = r\}$. Using this, one can show that (being vague here about our notion of convergence)
\begin{equation}\label{eq:UBm-large}\int_0^{2\pi} U_{B_m}(\rho e^{i\theta}) \rightarrow  + \infty,\end{equation}
as $m$ tends to infinity. Here we are writing $U_{B_m}$ to denote the logarithmic potential of the spectral measure of $B_m$. Roughly speaking, this holds since a) $U_{B_{m}}(z)$  is large and positive when $z\in S_{\g,\rho}$ b) $S_{\g,\rho}$ has positive measure on the contour we are integrating over c) $U_{B_m}(z)$ typically cannot be too negative and so there is no significant cancellation. However, \eqref{eq:UBm-large} contradicts the deterministic fact of complex analysis 
\[ \int_{0}^{2\pi} U_{B_m}(z) = \int_0^{2\pi} \frac{1}{m}\sum_i \log |z- \lambda_i| = O_{|z|,\rho}(1),\]
which can be seen by applying, for example, Jensen's formula. 
Thus we obtain a contradiction and see that the set $\{ z : \nu_n'(z)(\{0\}) >0 \}$ must have measure zero. Thus we can conclude \eqref{eq:Wmgn} for almost all $z \in \C$, as desired and ensure that for almost all $z$ we can initialize  our process.

\subsection{Organization of the paper}\label{sub:organization}

In Section~\ref{sec:setup} we formally define the random process that we use to incrementally expose the matrix $B_n$. We then go on to define the various notations regarding the random walk which will be used throughout the paper. In Section~\ref{sec:coupling}, we prove the coupling estimate between our iid $\on{Ber}(d/n)$ matrix $A_n$ and the random matrix $B = B_n$. This allows us to restrict our study to the random matrix $B_n$ for the remainder of the paper. 

In Section~\ref{sec:deg-seq} we prepare for Section~\ref{sec:UNE} by proving some basic properties about the row and column sums of our matrix $B_n$. In Section~\ref{sec:UNE}, we prove that the appropriate quasi-randomness properties $\cU(r)$ hold in both epochs. In Section~\ref{sub:unstructuredness}, we show how the properties $\cU(r)$ imply spreadness and unstructuredness of near minimal singular vectors. In Section~\ref{sec:anticoncentration}, we prove that these spreadness results imply that the image of the new row or column added is unlikely to lie near the subspace spanned by near minimal singular vectors. In Section~\ref{sec:random-walk}, we prove a basic result regarding random walks with drift.

In Section~\ref{sec:update-formula} we prove the crucial linear algebra lemma which allows one to convert the projection onto near minimal singular vectors into a shift of the window of singular values downward. In Section~\ref{sec:singular-convergence} we prove convergence of the singular values of shifted random matrices. In Section~\ref{sec:non-atomic-singular} we prove that the shifted singular value measures have no atom at $0$ for Lebesgue almost all $z$. In Section~\ref{sec:crucial}, we run our random walk argument to show that the bottom window of singular values is bounded with high probability. In Section~\ref{sec:graph-argument} we prove Theorem~\ref{thm:main} when $d\le 1$ using graph theoretic techniques. Finally, in Section~\ref{sec:proof} we complete the proof, by using our bound on the bottom window of singular values to control the convergence of the logarithmic potential.

\section{Definition of the random process}\label{sec:setup}
We define our main object of study $A = A_m$ to be an $n\times n$ matrix with independent $\{0,1\}$-entries such that 
\begin{equation}\label{eq:Adef}
\mb{P}(A_{ij} = 1) = d/n. 
\end{equation}
We now let $\ell = \lfloor(\log n)^2\rfloor$ and then define $B = B_n$ to be the $n\times n$ random matrix with independent entries in $\{0,1\}$ such that 
\begin{equation}\label{eq:Bdef}
\mb{P}(B_{ij} = 1)= \begin{cases} \, d/n \hspace{6.5em} \text{ if } \max(i,j)\le n-\ell \text{ and } \\ \,\sqrt{\log n}/n \hspace{4em} \text{ otherwise.}\end{cases}
\end{equation}
As we discussed in Section~\ref{sec:end-of-the-process}, we will relate our given matrix with the matrix $B$ and then spend most of the proof working with $B$. In what follows we formally define the random process which ``reveals'' $B$.

\subsection{Formal definition of the process}\label{sec:formal-def-of-process} We now formally define our random process, which we sketched in Sections~\eqref{sec:end-of-the-process} and \ref{sec:extending-to-threshold}. This is a careful way of exposing the rows and columns of our matrix $B$. 

Throughout the analysis of our random process we assume that $d$ is fixed and $d > 1$. (We treat $d\leq 1$ by other means in Section~\ref{sec:graph-argument}). Our process takes as a parameter a integer $\Delta$, which we send to infinity in the end of the proof. We let 
\[ \eps = \mb{P}\big( \on{Po}(d)\ge\Delta \big) .\]
We now partition the $[n]  = T_1 \cup T_2 \cup T_3$ by defining, 
\begin{equation}\label{eq:T1T2T3} T_1 = [\lfloor n(1-\eps)\rfloor],~T_2 = [\lfloor n(1-\eps)\rfloor+1, n - \ell],\text{ and }T_3 = [n-\ell+1, n].\end{equation}
Now for each $j \in T_2$, define the \emph{value} of $j$ to be 
\[\on{val}(j) = \min\big\{\deg_B^+(j,[n-\ell])\, ,\, \deg_B^-(j,[n-\ell])\big\},\] where $\deg_B^+(j,[n-\ell]),\, \deg_B^-(j,[n-\ell])$, denote the in (respectively) out neighbors of $j$ in $[n-\ell]$.
Here we are associating $\{0,1\}$-matrices to digraphs in the natural way (see Appendix~\ref{sub:digraph-notation}). %(We will use the conventions there in the remainder of the paper without further mention.)

We now define $H \subset T_2$ to be the $\lfloor\eps^3 n\rfloor$ indices $j$ in $T_2$ with largest $\on{val}(j)$ (breaking ties by choosing earlier indices in the integer ordering first).

We now iteratively build our matrix $B$ such that $B_n = B$. Define $m = m(\eps) = n-\ell - \lfloor \eps^{3}n\rfloor$. Then define a sequence of random sets $S_t$ for $m\le t\le n$ and random indices $v_t$ for $m+1\le j\le n$ as follows. First set 
\[ S_m = T_1 \cup (T_2\setminus H). \]
Now for $j$ satisfying $m \leq j < n-\ell$, we select 
\[ v_{j+1} \in H \setminus S_j, \]
uniformly at random and set $S_{j+1} = S_j \cup \{ v_{j+1} \}$. For $n-\ell \leq j < n$ we simply set 
\[ v_{j+1} = j+1 \]
and then define $S_{j+1} = S_j \cup \{ v_{j+1} \}$.
We then define the sequence of random matrices\footnote{If $M$ is a matrix $S$ is a subset of the rows and $T$ is a subset of the columns, we let $M[S,T]$ denote the $|S| \times |T|$ matrix $(M_{i,j} : i \in S , j \in T )$ }
\[B_t = B[S_j,S_j], \]  for $ m\le t\le n$ and 
\[ B_{t}^{\ast} = B[S_{t-1},S_{t}], \] for all $m+1\le t\le n$. 

\begin{comment}
iteratively as 
\begin{align*}
S_{m} &= T_1 \cup (T_2\setminus H),\\
v_{j+1} &= \on{Unif}(H\setminus S_{j}),&&\text{for }  m\le j< n-\ell,\\
v_{j+1} &= j,&&\text{for } n-\ell\le j<n,\\
S_{j+1} & = S_{j}\cup \{v_{j+1}\},&&\text{for } m\le j<n.
\end{align*}
Finally, define
\[B_j = B[S_j,S_j]\text{ for }m\le j\le n\text{ and }B_{j}^{\ast} = B[S_{j-1},S_{j}]\text{ for }m+1\le j\le n.\]
\end{comment}

We refer to the time steps $t \in [m,n-\ell]$ as the \emph{first epoch} of the process and $t \in  [n-\ell + 1,n]$ as the \emph{second epoch}. 

%Less formally, for the last $\ell$ steps of the process we build the matrix $B$ by adding back in columns and then rows in the obvious manner. For the first $\lfloor \eps^{3}n\rfloor$ steps of the process though, we extract the vertices in $T_2$ with high in-degree and out-degree (with respect to the digraph given by restricting attention $T_1\cup T_2$), and add them back in a random order.

\subsection{Elementary properties of the process}
We now turn to state a series of elementary properties regarding the random walk and the matrices $B_t$. 
\begin{fact}\label{fact:end-ind}
We have $S_{n-\ell} = [n-\ell]$ deterministically. Furthermore, indices $B_{i,j}$ for $\max(i,j)\ge n-\ell+1$ are jointly independent of $B_t,B_t^{\ast}$ for $t\le n-\ell$.
\end{fact}
\begin{proof}
The fact $S_{n-\ell} = [n-\ell]$ follows by construction of the sequence $v_j$. For the second part of the claim, note that the ordering of indices to form $S_{n-\ell}$ is only dependent on the entries of $B_{i,j}$ with $\max(i,j)\le n-\ell$ (and independent randomness). Since $B_t,B_t^\ast$ are submatrices of $B_{n-\ell}$ defined in terms of randomness independent from the last $\ell$ rows and columns, the desired result follows from the initial definition of $B$.
\end{proof}

We now state the state the obvious but crucial symmetry property of $B$.
\begin{fact}\label{fact:symmetry}
Conditional on the value $\sum_{1\le i,j\le n-\ell} B_{ij}$, the matrix $B_{n-\ell}$ is uniform over all $(n-\ell)\times(n-\ell)$ matrices with $\{0,1\}$ entries and exactly $\sum_{1\le i,j\le n-\ell} B_{ij}$ many $1$s.
\end{fact}

We next prove that the distribution of $B_t$ is a mixture of degree-constrained random graphs. This property will be used to establish unique-neighborhood expansion facts.
\begin{fact}\label{fact:deg-seq}
The set $H$ is measurable given the degree sequence $(\mbf d_{n-\ell}, \mbf d_{n-\ell}')$ of $B_{n-\ell}$. Additionally, given $j$ such that $m\le j\le n-\ell$, conditional on the index set $S_j$, and conditional on the degree sequence $(\mbf d_j, \mbf d_j')$ of $B_j$, the random variable $B_j$ is a uniformly random bipartite graph with degree sequence $(\mbf d_j, \mbf d_j')$.
\end{fact}
\begin{remark*}
The analogous result holds for $B_j^{\ast}$; this will not be required for the proof. 
\end{remark*}
\begin{proof}
The first claim follows as $H$ is defined via examining $\on{val}(j)$ for $j\in T_2$, which is measurable in terms of the degree sequence of $B_{n-\ell}$ by construction. Since $\sum_{1\le i,j\le n-\ell} B_{i,j}$ is measurable in terms of the degree sequence of $B_{n-\ell}$, applying Fact~\ref{fact:symmetry} we have that $B_{n-\ell}$ is a uniformly random digraph given its degree sequence. Note that the sets $S_j$ are determined given the degree sequence of $B_{n-\ell}$ and independent randomness. Furthermore, if we reveal the degree sequence of $B_j=B_{n-\ell}[S_j,S_j]$ then the conditional digraph $B_j$ is independent of the ``outside'' digraph corresponding to edges of $B_{n-\ell}$ not fully contained in $S_j$. These facts allow us to deduce the second claim.
\end{proof}

We now define the filtration $\mc{F}_j$ under which the random walk will occur. We abuse notation and identify a $\sigma$-algebra with a collection of random variables; such a $\sigma$-algebra is thus defined as the minimal $\sigma$-algebra such that the collection of random variables is jointly measurable. Define the set of pairs of indices
\[\mc{R} = (S_m \times S_m)\cup(T_2 \times H)\cup(H\times T_2).\]
$\mc{R}$ will correspond to the set of revealed entries of $B$ for the initial setup of the random walk. In particular $\mc{R}$ reveals all edges in the initial $S_m$ and all edges between $H$ and $T_2$

Define 
\begin{align*}
\mc{F}_m &= \{B_{ij}\}_{(i,j)\in\mc{R}} \cup \{\deg_{B}^+(i,[n-\ell]),\deg_{B}^-(i,[n-\ell])\}_{i\in T_2},\\
\mc{F}_j &= \mc{F}_{j-1} \cup\{S_{j-1},B_{j-1}\},&&\text{for } m+1\le j<n,\\
\mc{F}_j' &= \mc{F}_j\cup\{S_{j}\},&&\text{for } m+1\le j<n.
\end{align*}
$\mc{F}_m$ can be viewed as revealing all edges of $B_{n-\ell}$ in $\mc{R}$ and the degrees of all vertices in $T_2$ with respect to the digraph $B_{n-\ell}$. $\mc{F}_j$ is obtained from $\mc{F}_{j-1}$ by revealing the vertex $v_{j-1}$ and its edges to all vertices in $S_{j-1}$. $\mc{F}_j'$ is obtained from $\mc{F}_j$ by revealing the identity of $v_j$. (So, $\mc{F}_{j+1}$ compared to $\mc{F}_j'$ is only additionally revealing the in- and out-neighborhood of $v_j$ within $B_j$.)

We now state a series of claims regarding the $\sigma$-algebra $\mc{F}_j$ which can be easily seen by chasing definitions.
\begin{fact}\label{fact:sigma-algebra}
We have the following:
\begin{itemize}
    \item $\mc{F}_j$ for $m\le j\le n$ forms a filtration;
    \item $H$ is measurable with respect to $\mc{F}_m$;
    \item $\deg_B^+(i,T_1)$ and $\deg_B^-(i,T_1)$ for $i\in T_2$ are measurable with respect to $\mc{F}_m$;
    \item $(B_j)_{k,\ell}$ are $\mc{F}_{j}'$-measurable except if $v_j\in \{k,\ell\}$ and $\{k,\ell\}\cap T_1 \neq \emptyset$;
    \item $\deg_{B_j}^+(v_j)$ and $\deg_{B_j}^-(v_j)$ are $\mc{F}_j'$-measurable.
\end{itemize}
\end{fact}
The last bullet point follows since $\mc{F}_m$ reveals $\mc{R}$, which includes all pairs in $H\times H$. Now, the crucial property of the sequence of $\sigma$-algebras $\mc{F}_j$ is that the ``remaining'' randomness of $B_{j}$ given $\mc{F}_j$ and $S_j$ is particularly simple. 
\begin{fact}\label{fact:row-random}
Given $\mc{F}_{j}'$ we have that $\{(B_j)_{v_j,i}\}_{i\in T_1}$ and $\{(B_j)_{i,v_j}\}_{i\in T_1}$ are independent uniformly random $\{0,1\}$-vectors conditional on the sums $\sum_{i\in T_1}(B_j)_{i,v_j}$ and $\sum_{i\in T_1}(B_j)_{v_j,i}$, respectively. (These sums are deterministic given $\mc{F}_j'$, by the last bullet point of Fact~\ref{fact:sigma-algebra}.) 
\end{fact}

\section{Replacing $A$ with $B$}\label{sec:coupling}
In this short section we make rigorous the idea we discussed in Section~\ref{sec:end-of-the-process}, that we can replace our random iid matrix $A$ with the random matrix $B$ (see \eqref{eq:Bdef}). As we remarked in Section~\ref{sec:end-of-the-process} the proof of this result is intuitive; it is very hard to distinguish the models $A$ and $B$, since the number of heavy rows we are adding to form $B$ is much smaller than the standard deviation of the number of heavy rows in $A$.

\begin{lemma}\label{lem:TV-estimate}
Let $A$ and $B$ be as in \eqref{eq:Adef} and \eqref{eq:Bdef} and let $P_{\sigma}$ be a uniformly random $n\times n$ permutation matrix. For $n$ sufficiently large,
\[\on{TV}(A,P_{\sigma}^{T}BP_{\sigma})\le n^{-1/4}.\]
\end{lemma}

Lemma~\ref{lem:TV-estimate} follows immediately by iterating the following result, which says that the total variation distance is small if we replace a single random $\Ber(d/n)$ row and column with a random $\Ber(\sqrt{\log n}/n)$ row and column. 

%We remark that beyond implying Lemma~\ref{lem:TV-estimate}, Lemma~\ref{lem:coupling} in fact shows that one can somewhat freely pass permutation-symmetric events that hold whp between an independent model with probability $p=d/n$ and an independent model where a small number of rows or columns have an altered probability; we will briefly remark when we are doing this.

\begin{lemma}\label{lem:coupling}
Suppose that $d\in[(\log n)^{-1/3},(\log n)^{1/3}]$ and $\tau\in [(\log n)^{1/2}/(2n),2(\log n)^{1/2}/n]$. Define probability distributions $\mc{P}_1,\mc{P}_2$ on $n\times n$ $\{0,1\}$-matrices $M$:
\begin{itemize}
    \item $\mc{P}_1$: Each entry $M_{ij}$ is $1$ with probability $d/n$ independently at random.
    \item $\mc{P}_2$: Let $\sigma$ be a uniformly random element of $[n]$. If $\sigma\in\{i,j\}$, then $M_{ij} = 1$ with probability $\tau\in [(\log n)^{1/2}/(2n),2(\log n)^{1/2}/n]$; else $M_{ij} = 1$ with probability $d/n$.
\end{itemize}
We have 
\[\mr{TV}(\mc{P}_1,\mc{P}_2)\le n^{-1/3+o(1)}.\]
\end{lemma}
\begin{proof}
Let $L = (\log n)(\log\log n)^{-1/2}$ and note that
\[\mb{P}_{\mc{P}_1}\bigg(\max_{i\in [n]}\sum_{j=1}^nM_{ij}\ge L\bigg)\le n\binom{n}{L}\bigg(\frac{d}{n}\bigg)^{L}\le n \cdot \bigg(\frac{en}{L}\bigg)^{L}\bigg(\frac{d}{n}\bigg)^{L}\le n^{-\omega(1)},\]
and similar for columns. For a matrix $M$, let
\[S_\ell(M) = \#\Big\{k\colon\sum_{i=1}^nM_{ik}+\sum_{j=1}^nM_{kj}-M_{kk}=\ell\Big\}\]
and note for $\ell\le 2L$ that
\[\mb{E}_{\mc{P}_1}S_\ell(M) = n\bigg(\frac{d}{n}\bigg)^{\ell}\binom{2n-1}{\ell}\bigg(1-\frac{d}{n}\bigg)^{2n-1-\ell} = (1\pm n^{-1 + o(1)})\frac{n(2d)^{\ell}e^{-2d}}{\ell!}\]
and 
\begin{align*}
\mb{E}_{\mc{P}_1}S_\ell(M)^2 &\le \mb{E}_{\mc{P}_1}S_\ell(M) + (1\pm n^{-1+o(1)})n^2\bigg(\frac{(2d)^{\ell}}{\ell!}e^{-2d}\bigg)^{2}\\
&\le (1\pm n^{-1+o(1)})(\mb{E}_{\mc{P}_1}S_\ell(M) + (\mb{E}_{\mc{P}_1}S_\ell(M))^2).
\end{align*}
Thus by Chebyshev's inequality, we have for all $0\le\ell\le 2L$ that
\[|S_\ell(M)-\mb{E}_{\mc{P}_1}S_\ell(M)|\le n^{1/6 + o(1)}\sqrt{\mb{E}_{\mc{P}_1}S_\ell(M)} + n^{-1/3 + o(1)}\mb{E}_{\mc{P}_1}S_\ell(M)\]
with probability at least $1-n^{-1/3}$. Let $\mc{G}$ denote the set of matrices $M\in\{0,1\}^{n\times n}$ such that
\[|S_\ell(M)-\mb{E}_{\mc{P}_1}S_\ell(M)|\le n^{1/6+o(1)}\sqrt{\mb{E}_{\mc{P}_1}S_\ell} + n^{-1/3+o(1)}\mb{E}_{\mc{P}_1}S_\ell\]
for $1\le \ell\le 2L$ and $S_\ell(M) = 0$ for $\ell>2L$ hold simultaneously. Therefore we find
\begin{align*}
\on{TV}(\mc{P}_1,\mc{P}_2) &=\sum_{M'\in \{0,1\}^{n\times n}} (\mb{P}_{M\sim \mc{P}_1}(M=M') -\mb{P}_{M\sim \mc{P}_2}(M=M'))\mbm{1}_{\mb{P}_{M\sim \mc{P}_1}(M=M')\ge \mb{P}_{M\sim \mc{P}_2}(M=M')}\\
&\le \sum_{\substack{M'\in \{0,1\}^{n\times n}\\M'\in \mc{G}}} \Big|\mb{P}_{M\sim \mc{P}_1}(M=M') -\mb{P}_{M\sim \mc{P}_2}(M=M')\Big| + \mb{P}_{M\sim \mc{P}_1}(A\notin \mc{G})\\
&\le \sum_{\substack{M'\in \{0,1\}^{n\times n}\\M'\in \mc{G}}} \mb{P}_{M\sim \mc{P}_1}(M=M')\cdot \Bigg|1 -\frac{\mb{P}_{M\sim \mc{P}_2}(M=M')}{\mb{P}_{M\sim \mc{P}_1}(M=M')}\Bigg| + n^{-1/3}.
\end{align*}
However, for all $M'\in\mc{G}$ we have
\begin{align*}
\Bigg|1 -&\frac{\mb{P}_{M\sim \mc{P}_2}(M=M')}{\mb{P}_{M\sim \mc{P}_1}(M=M')}\Bigg|=\Bigg|1-\frac{1}{n}\sum_{0\le \ell\le 2L}S_{\ell}(M')\bigg(\frac{\sqrt{\log n}}{d}\bigg)^{\ell}\bigg(1-\frac{\sqrt{\log n}}{n}\bigg)^{2n-1-\ell}\bigg(1-\frac{d}{n}\bigg)^{-(2n-1-\ell)}\Bigg|\\
&\le \bigg|1-\frac{1}{n}\sum_{0\le \ell\le 2L}\frac{n(2d)^{\ell}e^{-2d}}{\ell!}\bigg(\frac{\sqrt{\log n}}{d}\bigg)^{\ell}e^{-2\sqrt{\log n}+2d}\bigg|+n^{-1+o(1)}\\
&+\frac{1}{n}\sum_{0\le \ell\le 2L}\bigg(n^{1/6+o(1)}\bigg(\frac{n(2d)^{\ell}e^{-2d}}{\ell!}\bigg)^{1/2} + n^{-1/3+o(1)}\bigg(\frac{n(2d)^{\ell}e^{-2d}}{\ell!}\bigg)\bigg)\bigg(\frac{\sqrt{\log n}}{d}\bigg)^{\ell}e^{-2\sqrt{\log n}+2d}\\
&\le \bigg|1-\frac{1}{n}\sum_{0\le \ell\le 2L}\frac{n(2d)^{\ell}e^{-2d}}{\ell!}\bigg(\frac{\sqrt{\log n}}{d}\bigg)^{\ell}e^{-2\sqrt{\log n}+2d}\bigg| + n^{-1/3 + o(1)}\le n^{-1/3+o(1)}.
\end{align*}
Therefore
\[\on{TV}(\mc{P}_1,\mc{P}_2) \le \sum_{\substack{M'\in \{0,1\}^{n\times n}\\M'\in \mc{G}}} \mb{P}_{M\sim \mc{P}_1}(M=M')\cdot \Bigg|1 -\frac{\mb{P}_{M\sim \mc{P}_2}(M=M')}{\mb{P}_{M\sim \mc{P}_1}(M=M')}\Bigg| + n^{-1/3+o(1)}\le n^{-1/3+o(1)},\]
and we are done.
\end{proof}

\section{Regularity of degree sequences}\label{sec:deg-seq}

In this technical section, we prepare for our study of unique neighborhood expansions in $B_t$. Since in the first epoch of the process the row and column sums sums are already exposed, we are naturally led to study the configuration model in a random directed graph (or equivalently a undirected bipartite graph) with a given degree sequence. In this section we show that we may assume that these given given row and column sums are appropriately typical. We shall then go on, in Section~\ref{sec:UNE}, to show that a random directed graph with these ``typical'' degree sequences have good unique neighbour expansion. 

The key definition here is the notion of a $(d,\mu,C)$-regular degree sequence, which we define formally here. 

\begin{definition}\label{def:deg-seq}
Consider degree sequences $\mbf{d} = (d_i)_i$ and $\mbf{d}' = (d_i')_i$ of length $m$. We say the degree sequence $(\mbf{d},\mbf{d}')$ is $(d,\mu,C)$-regular if
\begin{itemize}
    \item $m/n \in [1-\mu, 1+\mu]$;
    \item $\sum_{i\in S}(d_i+d_i')\le C(d+\log(m/|S|))|S|$ for all $S\subseteq[m]$;
    \item $\sum_{i=1}^md_i = \sum_{i=1}^md_i' = (1\pm\mu)dm$;
    \item $\sum_{i=1}^md^{-d_i}d_i' = (1\pm\mu)ed\exp(-d)m$.
\end{itemize}
In what follows, we will use $\mbf{d}$ to denote the degree sequence of the columns and $\mbf{d}'$ for the rows. Equivalently, we think of $\mbf{d},\mbf{d'}$ as the ``left'' and ``right'' parts of the degree sequence of a bipartite graph, respectively. \end{definition}

We now define the quasi-randomness event $\mc{D}$ that will concern us in this section. We say that a $\{0,1\}$-matrix $M$ is in $\mc{D}$ if the degree sequence associated to $M$ is $(d, \eps^{1/2}, 16)$-regular and for all sets $S$ of size at most $\sqrt{\log n}$ we have $\sum_{i,j\in S}M_{ij}\le |S|$.

Our goal in this section will be to prove that the event $\mc{D}$ holds with high probability for all of the steps in the process.
\begin{lemma}\label{lem:degree-event}
We have 
\[\mb{P}\bigg(\bigcap_{m\le t\le n}\{B_t^\dagger\in\mc{D}\}\cap \bigcap_{m\le t\le n}\{B_t\in \mc{D}\}\bigg) \ge 1-n^{-1+o(1)}.\]
\end{lemma}

We now set ourselves up to give a proof of Lemma~\ref{lem:degree-event}, which is fairly straight-froward, but somewhat technical. For the following definition we recall the definitions from Section~\ref{sec:formal-def-of-process}, $d$ is fixed $d>1$, $\Delta>0$ is fixed and large and $\eps = \PP( \on{Po}(d) \geq \Delta )$.
\begin{definition}\label{def:degree-distribution}
We define \[\gamma = \frac{\eps}{d} \cdot \mb{E}\, X\mbm{1}_{\min(X,Y)\ge \Delta} ,\]
where $X,Y\sim\on{Pois}(d)$ are independent. Now, for $j,k\ge 0$, we define
\begin{align*}
\rho_{j,k}(\Delta) &= (1-\eps^3)^{-1}\sum_{a\ge j}\sum_{a'\ge k}\binom{\ell}{j}\binom{\ell'}{k}\gamma^{a+a'-j-k}(1-\gamma)^{j+k}((1-\eps) + \eps \mbm{1}_{\min(a,a')<\Delta})\frac{d^a e^{-d}}{a!}\frac{d^{a'}e^{-d}}{a'!}.
\end{align*}
\end{definition}

The following lemma tells us what the degree sequences are in the random matrices $A$ and $B$.

\begin{lemma}\label{lem:degree-convergence}
With probability at least $1-n^{-\omega(1)}$, we have the following. For all $x,y\ge 0$,
\[\big|\big\{i\in[n]\colon(\deg_A^+(i),\deg_A^-(i))=(x,y)\big\}\big|=\frac{d^xe^{-d}}{x!}\frac{d^ye^{-d}}{y!}n+O(n^{1/2 + o(1)}).\]
For all $x,y\ge 0$,
    \[\big|\big\{i\in V(B_m)\colon(\deg_{B_m}^+(i),\deg_{B_m}^-(i))=(x,y)\big\}\big| =\rho_{x,y}(\Delta)m+O(n^{1/2 + o(1)}).\]

\end{lemma}
\begin{proof}
The first item follows from straightforward concentration arguments. We use permutation concentration. We can condition on the number of edges in $A$, which is $dn+O(\sqrt{n}\log n)$ with probability $1-n^{-\omega(1)}$. Then the probability a single vertex has degree $(x,y)$ is $(d^xe^{-d}/x!)(d^ye^{-d}/y!)+O(n^{-1/2+o(1)})$, so the expectation matches our prediction. Finally, we can use Lemma~\ref{lem:injection-concentration} to obtain the desired result, noting that we can view this conditioned model as a uniformly random injection of approximately $dn$ edges into $n^2$ total possibilities, and noting that changing an edge changes our statistic by at most $4$.

In fact, by the same argument we may deduce that the first $n-\ell-\lfloor\eps n\rfloor$ vertices and the next $\lfloor\eps n\rfloor$ vertices have the same in/out-degree distribution in this sense within the random model $B_{n-\ell}$.

We now deduce the second item from this fact. After revealing the degrees within $B_{n-\ell}$, we know the digraph is uniform over digraphs with this degree sequence. We thus use the configuration model, which we can easily check succeeds with positive probability using Lemma~\ref{lem:config}. We now work within the digraph configuration model to prove the desired result.

For every vertex in $\{n-\ell-\lfloor\eps n\rfloor+1,\ldots,n-\ell\}$ with minimum of in-degree and out-degree at least $\Delta$, mark it red. Then mark its incoming and outgoing edges red. Let the rest of the edges and vertices of $B_{n-\ell}$ be green. We care about the green degree distribution among the vertices left after deleting the red vertices: with probability $1-n^{-\omega(1)}$ this recovers $B_m$ except that we may delete $O(n^{1/2+o(1)})$ more or fewer vertices than intended. With probability $1-n^{-\omega(1)}$ the maximum degree of $B_{n-\ell}$ is at most $\log n$, so such deletions affect the degree statistics by an error of at most $O(n^{1/2+o(1)})$. So, let us focus on this ``red-deletion'' model, which has adjacency matrix $M$.

We claim that the number of vertices that go from having degrees $(a,a')$ in $B_{n-\ell}$ to having degrees $(x,y)$ in this final green digraph $M$ is
\[\binom{a}{x}\gamma^{a-x}(1-\gamma)^x\binom{a'}{y}\gamma^{a'-y}(1-\gamma)^y((1-\eps)+\eps\mbm{1}_{\min(a,a')<\Delta})\frac{d^a e^{-d}}{a!}\frac{d^{a'}e^{-d}}{a'!}n+O(n^{1/2+o(1)})\]
with probability $1-n^{-\omega(1)}$.

Indeed, note that a $\gamma+O(n^{-1/2+o(1)})$ fraction of out-stubs are red (call this fraction $\gamma'$), and the same for in-stubs, with appropriately high probability. To study the configuration model, we care about a uniformly random matching of out-stubs and in-stubs, and we specifically consider the number of edges at each green vertex that are not partially formed from a red stub.

The joint distribution of green in/out-degree statistics computed over green vertices can easily be compared to the model where among the green vertices, each stub is retained with probability $1-\gamma'$ independently. We start with roughly $(1-\eps)n$ vertices which are definitely green (which contributes roughly fraction $(d^ae^{-d}/a!)(d^{a'}e^{-d}/a'!)$ to the count of green vertices of total in/out-degree $(a,a')$) and roughly $\eps n$ vertices, of which the same fraction is contributed as long as $\mbm{1}_{\min(a,a')<\Delta}=1$. Then we retain the resulting stubs as green independently with probability $1-\gamma'$, which leads to the above estimate for the number of green vertices with green in/out-degree $(x,y)$ with probability $1-n^{-\omega(1)}$. This finishes the proof of the claim.

Now, summing over in/out-degrees $(a,a')$ of magnitude at most $\log n$ (which the maximum degree is bounded by with probability $1-n^{-\omega(1)}$), we obtain the desired result, recalling that $(1-\eps^3)^{-1}m\approx n$.
\end{proof}

We will need the following claim, which essentially captures that the neighborhoods of the vertices in $H=V(B_{n-\ell})\setminus V(B_m)$ are sufficiently uniform relative to a large portion of the digraph coming from $T_1$.
\begin{lemma}\label{lem:large-degree}
Assume the setup in Section~\ref{sec:setup} and suppose that $1/n\ll \eps\ll 1/d$. With probability $1-n^{-\omega(1)}$, all but $\eps^5n$ many $t\in T_2$ satisfy 
\[ \min\big\{ \deg_{B_{n-\ell}}^+(t,T_1),\deg_{B_{n-\ell}}^-(t,T_1)\big\} \ge\sqrt{\log(1/\eps)} \] and 
\[ \max\big\{\deg_{B_{n-\ell}}^+(t),\deg_{B_{n-\ell}}^-(t)\big\} \le(\log(1/\eps))^2. \]
\end{lemma}
\begin{proof}
The second part holds for all but at most $\eps^5n/3$ vertices immediately given the second item of Lemma~\ref{lem:degree-convergence} and basic computation.

Let us consider the number of total vertices in $T_2$ with at least $7$ out-neighbors in $T_2$, which is a set of size $\lfloor\eps n\rfloor$. By Chernoff, it is easy to see that with probability $1-n^{-\omega(1)}$, we have at most $\eps^5n/3$ vertices with at least $7$ out-neighbors in $T_2$. The same holds for in-neighbors. We obtain a combined exceptional set in $T_2$ is of size at most $\eps^5n$. It suffices to check that for $t\in T_2$ not in this exceptional set, we have $\min(\deg_{B_{n-\ell}}^+(t,T_1),\deg_{B_{n-\ell}}^-(t,T_1))\ge\sqrt{\log(1/\eps)}$.

Now note that $B_{n-\ell}$ and $A$ can be coupled so that $B_{n-\ell}$ is a sub-digraph of $A$ \emph{almost surely}. We can thus apply the first item of Lemma~\ref{lem:degree-convergence} (holding with probability $1-n^{-\omega(1)}$).

Given this event, by the definition of $T_2$ as the highest $\eps^2$ quantile (of minimum in- and out-degree within $B_{n-\ell}$) from a fixed set of size $\eps n$, we have with probability $1-n^{-\omega(1)}$ that $\min(\deg_{B_{n-\ell}}^+(v),\deg_{B_{n-\ell}}^-(v))\ge2\sqrt{\log(1/\eps)}$ for all $v\in T_2$. Since $\deg_{B_{n-\ell}}^+(v,T_1)\ge\deg_{B_{n-\ell}}^+(v)-7\ge\sqrt{\log(1/\eps)}$ outside of our exceptional set, and the same holds for in-degrees, we are done.
\end{proof}

\vspace{1em}

Next we will require the following basic estimate regarding the maximum number of entries in a small diagonal block of the matrix.
\begin{lemma}\label{lem:diagonal-sum}
With probability at least $1-n^{-1 + o(1)}$, for all $m\le t\le n$ and any $S$ with $|S|\le \sqrt{\log n}$ we have
\[\sum_{i,j\in S} (B_{t})_{ij}\le |S|.\]   
\end{lemma}
\begin{proof}
Noting the monotone nature of the estimate, it suffices to prove that 
\[\sum_{i,j\in S} B_{ij}\le |S|\]
for all $|S|\le \sqrt{\log n}$. We have

\[ \mb{P}\bigg(\exists S\subseteq [n]\colon |S|\le \sqrt{\log n},~\sum_{i,j\in S} B_{ij}\ge |S| + 1\bigg)\le \sum_{k\le \sqrt{\log n}}\binom{n}{k}\binom{k^2}{k+1}\bigg(\frac{\sqrt{\log n}}{n}\bigg)^{k+1},\] which is 
\[\le \sum_{k\le \sqrt{\log n}}\bigg(\frac{en}{k}\bigg)^{k}(ek)^{k+1}\bigg(\frac{\sqrt{\log n}}{n}\bigg)^{k+1} \le \sum_{k\le \sqrt{\log n}}k (e^2\sqrt{\log n})^{k+1}n^{-1},\]
hence this failure probability is bounded by $n^{-1+o(1)}$ as desired.
\end{proof}

%We will also require the following variant event which handles sets that are substantially larger. 

\subsection{Proof of Lemma~\ref{lem:degree-event}}

\begin{proof}[Proof of Lemma~\ref{lem:degree-event}]
That the degree sequence of $B$ is $(d,n^{-1/4},8)$-regular follows from Lemma~\ref{lem:degree-convergence}, noting that $A$ and $B$ can be coupled to differ in at most $(\log n)^3$ entries with probability $1-n^{-\omega(1)}$, and noting that the maximum degree of $B$ is bounded by $\log n$ with probability $1-n^{-\omega(1)}$ (which allows control of the second bullet point of Definition~\ref{def:deg-seq} for small $S$).

Noting that $B_t$ is obtained from $B$ by removing at most $\eps n + \ell$ vertices and the largest $\eps n$ vertices have total in- and out-degree bounded by $O(\eps(\log(1/\eps) + d)n)$ from regularity of the degree sequence of $B$, the desired $(d,\eps^{1/2},16)$-regularity for all $B_t$ follows immediately since $\eps\ll 1/d$. Here regular degree sequence is as in Definition~\ref{def:deg-seq}. This establishes the first part of $\mc{D}$ for all relevant matrices.

The second part of the definition of $\mc{D}$ holds for all our matrices with probability $1-n^{-1+o(1)}$ due to Lemma~\ref{lem:diagonal-sum}.
\end{proof}

\section{Unique neighbourhood expansions}\label{sec:UNE}
In this section we formally define the set of unique neighbors of a set $S$ of columns of a $\{0,1\}$-matrix, a concept which we introduced
in the discussion in Section~\ref{sec:sketch-une}. We then go on to define the quasi-randomness event $\cU(r)$ that we use to prove our key anti-concentration estimates.

Let $M$ be an $m\times s$ matrix and let $S\subset [s]$ be a set of columns of $M$. We define $U(S)\subset [m]$ to be a subset of rows in two stages. We first define 
\[ U(S) \setminus S = \big\{ i \in [m]\setminus S\colon B_{ij} = 1 \text{ for a unique } j \in S\big\} \]
and then define
\[ \hspace{-2em} U(S) \cap S = \big\{i\in[m]\cap S\colon B_{ij} = 0 \text{ for all } j \in S \big\}. \] 
We now define our three key quasi-randomness properties involving unique neighborhood expansion that we shall need.  We assume $\kappa>0$ is a small constant such that  
\begin{equation}\label{eq:def-kappa} \eps\ll\kappa\ll 1/d .\end{equation}
We then define the function  
\begin{equation}\label{eq:def-alpha} \alpha(x) = (\log(n/x))^{-2} .\end{equation}

We now define our quasi-randomness event. Say that a $\{0,1\}$-matrix $M$ is in $\mc{U}(r)$ if for all sets $S$  with $|S| \in [r,\kappa n]$, we have
\[|U(S)|\ge\alpha(|S|) |S|.\]
The following is the main quasi-randomness event of the first epoch of our process. 

\begin{lemma}\label{lem:expansion-epoch1}
We have 
\[\mb{P}\bigg(\bigcap_{m\le t\le n-\ell}\big\{B_t^\dagger\in \mc{U}((\log n)^{3/2})\big\}\cap \bigcap_{m\le t\le n-\ell}B_t\in \big\{\mc{U}((\log n)^{3/2})\big\}\bigg) \ge 1-n^{-\omega(1)}.\]
\end{lemma}
The following lemma is the main quasi-randomness event for the second epoch.

\begin{lemma}\label{lem:expansion-epoch2}
We have 
\[\mb{P}\bigg(\bigcap_{n-\ell\le t\le n}\big\{B_t^{\dagger}\in \mc{U}((\log\log n)^{2})\big\} \cap \bigcap_{n-\ell\le t\le n}\big\{B_t^{\ast}\in \mc{U}((\log\log n)^2)\big\}\bigg) \ge 1-(\log n)^{-\omega(1)}.\]
\end{lemma}

We now introduce a variant of the event $\cU(r)$ for the second epoch of the process. Let $M$ be a matrix with rows indexed by $T$ and columns indexed by $T \cup \{t\}$, $t \not\in T$. We say that $M$ is in $\cU^{\ast}$ if all $S \subseteq T \cup \{t\}$, with $t \in S $, satisfy $U(S)\ge 1$.
%We say that a $\{0,1\}$-matrix $M$ of with row indices $T$ and column indices $T\cup\{t\}$ for $t\notin T$ is in $\mc{U}^\ast$ if all column subsets $S\subseteq T\cup\{t\}$ containing $t$ satisfy 
%\[U(S)\ge 1.\] 
For the second epoch of our process we also need the following simple quasi-randomness statement. 
\begin{lemma}\label{lem:expansion-epoch2-spec}
We have 
\[\mb{P}\bigg(\bigcap_{n-\ell\le t\le n}\big\{B_t^{\ast}\in \mc{U}^{\ast}\big\}\bigg) \ge 1-(\log n)^{-\omega(1)}.\]
\end{lemma}

In the remainder of this section we will prove the above three lemmas. Towards this goal, we start with some preparatory calculations in the second epoch of the process, as it is less technical to work with. We will then move on to some preparatory calculations for the first epoch, before proving Lemmas~\ref{lem:expansion-epoch1},\ref{lem:expansion-epoch2} and \ref{lem:expansion-epoch2-spec} at the end of the section.

\subsection{The second epoch: calculations in the independent model}
We now bound the number of small sets which do not have many unique neighbors. We remark the analysis here is nearly identical to that in our companion paper \cite[Lemma~4.2]{SSS23}. 

\begin{lemma}\label{lem:unstructured-graph-1}
For $\delta > 0$ and an integer $m$ with $m/n \in [1/2,2]$, let $\ell\in\{m,m+1\}$, and $pm\in[1+\delta,\delta^{-1}]$. Let $M$ be an $m\times\ell$ matrix with iid $\on{Ber}(p)$ entries. Then there exist constants $c > 0$, $C > 0$, depending only on $\delta$, such that the following holds.

For all $k\in[0,cn]$, we have
\[ \mb{E}_{M}\Bigg|\bigg\lbrace S\in\binom{[\ell]}{k}\colon |U(S)| <  \alpha(k)k \bigg\rbrace \Bigg| \leq Ce^{-c k}. \]
\end{lemma}
\begin{remark*}
The assumption that $pm\ge 1+\delta$ is used in a crucial manner although the proof may be adjusted to handled $pm \le 1-\delta$. 
\end{remark*}
\begin{proof}
For the sake of simplicity we will consider the case when $\ell = m+1$; the other case is strictly simpler. Fix a set $S \subseteq [\ell]$ of columns of size $k$. Note that 
\[ |U(S)| = \sum_{i \in [m]} \1(i \in U(S) ), \]
where the sum is over the rows and therefore a sum of independent random variables. Note that for $i\in[\ell]$ we have
\begin{equation} \label{eq:NU-notS}
\PP( i \in U(S)  ) = (1-p)^k =: q_1 \text{ if }i \in S\text{ and } \PP( i \in U(S) )= (1-p)^{k-1}pk =: q_2\text{ if }i \not\in S.
\end{equation}

Let $k' = |S\cap [m]|\in\{k-1,k\}$ and set $T = \alpha(k) k$. Now $|U(S)\cap S|$ is distributed as binomial random variable $B(k',q_1)$ and $|U(S)\setminus S|$ is distributed as $B(m-k',q_2)$. Note by Bernoulli's inequality that $1-q_1\le pk$ and $q_2\ge (1-p)^kpk\ge (1-pk)pk$. Take $\eta$ to be a sufficiently small constant with respect to $\delta$ to be chosen later. By taking $c$ sufficiently small in terms of $\eta$, we have $pk\le\eta$ and $k\le\eta m$.

Therefore we have
\begin{align*}
\EE\bigg|\bigg\{S\in\binom{[\ell]}{k}\colon&|U(S)| < T \bigg\}\bigg| \le \EE\bigg|\bigg\{S\in\binom{[\ell]}{k}\colon|U(S)\cap S| < T,~|U(S)\setminus S| < T\bigg\}\bigg| \\
&\le \sum_{\substack{i,j<T\\k'\in \{k-1,k\}}} \binom{m}{k'} \mb{P}(B(k',q_1) = i)\cdot \mb{P}(B(m-k',q_2) = j)\\
&\le \sum_{\substack{i,j<T\\k'\in \{k-1,k\}}} \binom{m}{k'} \binom{k'}{i}(1-q_1)^{k'-i}\cdot \binom{m-k'}{j}(1-q_2)^{m-k'-j}\\
&\le (T+1)^2\binom{m}{\lfloor T\rfloor}^2\sum_{k'\in \{k-1,k\}} \binom{m}{k'} (pk)^{k'-T}\cdot(1-pk + (pk)^2)^{m-k'-T}\\
&\le 2(T+1)^2\binom{m}{\lfloor T\rfloor}^2(pk)^{-2T}(emp)^k\cdot(1-pk + (pk)^2)^{m-k}.
\end{align*}
Using $\binom{a}{b}\le (\frac{ea}{b})^{b}$ and $1-x\le e^{-x}$, we have 
\begin{align*}
2(T+1)^2\binom{m}{\lfloor T\rfloor}^2&(pk)^{-2T}(emp)^k\cdot(1-pk + (pk)^2)^{m-k}\\
&\le \exp\big(O(T\log(m/T)))\cdot (emp)^k\cdot(1-pk(1-\eta))^{m(1-\eta)}\\
&\le \exp\big(O(T\log(m/T)))\cdot (emp)^ke^{-pkm(1-\eta)^2}\\
&= \exp\big(O(T\log(m/T)))\cdot(emp e^{-pm(1-\eta)^2})^{k}\le Ce^{-ck}.
\end{align*}
The final line follows since if $\delta \gg \eta \gg c$ then for $x\ge 1+\delta$, we have $ex\le e^{x(1-\eta)^2 - 2cx}$.
\end{proof}

We will also require the following exceptional case which is used to handle steps near the end of the process. 

\begin{lemma}\label{lem:case-exceptional}
Let $\delta > 0$, let $m,\ell$ be such that $\ell=m+1$, let \[ pm \in [1+\delta,\delta^{-1}] \quad \text{ and let } \quad \tau m\in [(1/2)\sqrt{\log m},2\sqrt{\log m}].\] If $M$ is a random matrix with independent entries where $M_{ij}\sim\mr{Ber}(p)$, for $j\in[m]$, and $M_{i\ell}\sim\mr{Ber}(\tau)$. Then for all $k\in[0,(\log m)^{1/3}]$, we have
\[ \mb{E}_M\Bigg|\bigg\lbrace S \in \binom{[\ell]}{k+1}\colon \ell\in S \emph{ and } |U(S)| = 0 \bigg\rbrace \Bigg| \le C \exp(-(\log m)^{1/2}/16), \]
where $C >0$ depending only on $\delta$.
\end{lemma}
\begin{proof}
As in Lemma~\ref{lem:unstructured-graph-1}, we have 
\begin{align*}
\mb{P}(i\in U(S))&\ge (1-p)^{k}(1-\tau)\ge1-pk-\tau&&\text{ if } i\in S,\\
\mb{P}(i\in U(S))&\ge \mb{P}(M_{i\ell} = 1)\cdot (1-p)^{k}\ge\sqrt{\log m}/(4m)&&\text{ if }i\notin S.
\end{align*}
Therefore
\begin{align*}
\mb{E}_{M}&\Bigg|\bigg\lbrace S \in \binom{[\ell]}{k+1}\colon\ell\in S,~|U(S)| = 0 \bigg\rbrace \Bigg| \le \binom{m}{k}\cdot (pk + \tau)^{k} (1-\sqrt{\log m}/(4m))^{m-k}\\
& \le e^{k}(m/k)^{k}\cdot\bigg(\frac{3\sqrt{\log m}}{m}\bigg)^{k} \exp(-\sqrt{\log m}/8)\le \exp(-(\log m)^{1/2}/16),
\end{align*}
where we have used that $m$ is sufficiently large with respect to $\delta^{-1}$ hence $pk\le\sqrt{\log m}$.
\end{proof}

\subsection{The first epoch: calculations in the configuration model}\label{sub:unique-neighbor-config}
We now turn to work in the first epoch of the process. For this we will need the notion of a $(d,\mu,C)$-regular degree sequence from Definition~\ref{def:deg-seq}.

\begin{lemma}\label{lem:unstructured-graph-2}
Let $\delta>0$ and $C>0$. Then there exists $c=c(\delta,C)>0$ such that the following holds. For $d\in[1+\delta,\delta^{-1}]$ let $(\mbf{d},\mbf{d}')$ be a degree sequence which is $(d,c,C)$-regular. 

Let $M$ be sampled, uniformly at random, among all bipartite graphs with degree sequence $(\mbf{d},\mbf{d}')$ and let $M$ be the bipartite adjacency matrix of this graph. 

For $k\in[0,cm]$, let $\mc{Y}_k$ be the collection of sets $S$ of the left hand side of the bipartition of size $k$ satisfying
\[ \big|\sum_{i\in S}d_i-k\big|\le k/\sqrt{\log (m/k)} \quad \text{ and } \quad  |U(S)|\le \alpha(k)k.\]
We have that  
\[\mb{E}\, |\mc{Y}_k|\le c^{-1}\exp(-ck).\]
\end{lemma}
\begin{proof}
To perform our calculations, we will operate in the configuration model. Using the second and third bullet points of Definition~\ref{def:deg-seq} and by Lemma~\ref{lem:config}, we have that the associated configuration model is simple with probability $\Omega_{C,d}(1)$ and therefore it suffices to prove the result where the graph associated to $M$ is sampled from the configuration model.

We first bound the probability of the event $S\in\mc{Y}_k$ for a fixed subset $S\subseteq [m]$. We first isolate the large degree verticies on the left hand side of the partition 
\[ V = \{i\colon d_i'\ge(\log(m/k))^3\},\]
by the second item of Definition~\ref{def:deg-seq} we have that $|V|\le \alpha(k)k$. We then break up the left hand side based on the two different ways a vertex can be in $U(S)$
\[ U_1 = S\cap([m]\setminus U(S))\cap([m]\setminus V), \quad \text{ and } \quad 
U_2 = ([m]\setminus S)\cap([m]\setminus U(S))\cap([m]\setminus V). \]

In order to compute the probability that $S\in\mc{Y}_k$, we seek to understand the probability that $j\in U_1$ has no neighbors in $S$ while for $j\in U_2$ we need to understand the probability that it has exactly one neighbor in $S$. Then we will take a union bound over possible revelations of $U_1,U_2$ and study the probability that all vertices in $U_1,U_2$ satisfy the relevant condition.

It suffices to understand the number of stubs attached to each vertex on the right which connect to $S$. This distribution is given precisely by choosing each stub on the right to connect to $S$ independently with probability $(\sum_{i\in S}d_i)/(\sum_{i\in[m]}d_i)$ and conditioning on exactly $\sum_{i\in S}d_i$ stubs being chosen. Note that the probability that the associated binomial distribution has value exactly $\sum_{i\in S} d_i$ is $\Omega(1/k)$.

Let $q=(\sum_{i\in S} d_i)/(\sum_{i\in [m]}d_i)$ and $c' = (\log(1/c))^{-1/4}$. By the first item in the lemma assumptions and the third item of Definition~\ref{def:deg-seq}, we see that $q\in(1\pm c')k/(dm)$. For given possible values of $U_1,U_2$, the probability that $S\in\mc{Y}_k$ is bounded by
\begin{align*}
O_{C,d}(k)\cdot\prod_{i\in U_1}(qd_i')\prod_{i\in U_2}(1-qd_i'(1-q)^{d_i'-1})&\lesssim_{C,d}(1+ 2c')^k\prod_{i\in U_1}\frac{kd_i'}{dm}\prod_{i\in U_2}(1-qd_i'+(qd_i')^2)\\
&\lesssim_{C,d} e^{2c'k}\prod_{i\in U_1}\frac{kd_i'}{dm}\prod_{i\in U_2}\bigg(1-(1-2c')\frac{kd_i'}{dm}\bigg)\\
&\lesssim_{C,d} e^{2c'k}\prod_{i\in U_1}\frac{kd_i'}{dm}\cdot\exp\bigg(-(1-2c')\sum_{i\in U_2}\frac{kd_i'}{dm}\bigg).
\end{align*}
Next note that as $c$ is sufficiently small with respect to $C$ and $d$, we have
\[\sum_{i\in [m]\setminus U_2}\frac{kd_i'}{dm} \le \frac{2k}{dm}\cdot C(d + \log(m/k))k\le c^{1/2}k.\]
Note that we are using that $[m]\setminus U_2\le |S| + |V| + |U(S)|\le 2k$ by assumption.

Now fix a constant $\eta$ sufficiently small with respect to $\delta$, and suppose $c$ is chosen sufficiently small with respect to $\eta$. Note that $|S\setminus U_1|\le|V|+|U(S)|\le 2\alpha(k)k$. It follows that
\begin{align*}
e^{2c'k}\prod_{i\in U_1}&\frac{kd_i'}{dm}\cdot\exp\bigg(-(1-2c')\sum_{i\in U_2}\frac{kd_i'}{dm}\bigg)\le e^{3c'k}\prod_{i\in U_1}\frac{k(d_i'+\eta)}{dm}\cdot\exp\bigg(-(1-2c')\sum_{i\in [m]}\frac{kd_i'}{dm}\bigg)  \\
&\qquad\qquad\le e^{6c'k}\prod_{i\in S}\frac{k(d_i'+\eta)}{dm}\cdot \prod_{i\in S\setminus U_1}\frac{dm}{k\eta}\cdot\exp(-k)\\
&\qquad\qquad\le e^{6c'k-k}\prod_{i\in S}\frac{k(d_i'+\eta)}{dm}\cdot \bigg(\frac{dm}{k\eta}\bigg)^{2\alpha(k)k}\le e^{7c'k-k}\prod_{i\in S}\frac{k(d_i'+\eta)}{dm}
\end{align*}
Since $|S\setminus U_1|,|(m\setminus U(S))\setminus U_2|\le 2\alpha(k)k$ by $S\in\mc{Y}_k$, we have that there are at most 
\[\sum_{j\le 2\alpha(k)k}\binom{k}{j}\cdot\sum_{j\le 2\alpha(k)k}\binom{m}{j}\le\bigg(\frac{em}{k\alpha(k)}\bigg)^{5\alpha(k)k}\le\exp(c'k)\]
choices for $U_1,U_2$ (given $S$). Therefore by the union bound 
\[\mb{P}(S\in\mc{Y}_k)\lesssim_{C,d}e^{8c'k-k}\prod_{i\in S}\frac{k(d_i'+\eta)}{dm}.\]

We now bound the sum over all $S$. Namely, we use the inequality
\[\sum_{S\in\binom{[m]}{k}}\prod_{i\in S}y_i\le\frac{1}{k!}\bigg(\sum_{i=1}^my_i\bigg)^k\]
with the choices $y_i=d^{-d_i}k(d_i'+\eta)/(dm)$. From this we deduce
\[d^{-k-(c')^2k}\mb{E}|\mc{Y}_k|\lesssim_{C,d}e^{8c'k-k}\cdot\frac{1}{k!}\bigg(\sum_{i=1}^md^{-d_i}\frac{k(d_i'+\eta)}{dm}\bigg)^k;\]
note that if $S\in\mc{Y}_k$ is counted then $\sum_{i\in S}d_i\le k+(c')^2k$. By the third item of Definition~\ref{def:deg-seq} recall $\sum_{i=1}^md^{-d_i}d_i'=(1\pm c)ed\exp(-d)m$. Additionally, Stirling's formula shows $k!\lesssim k(k/e)^k$. Recalling $k\le cm$ and that $c'$ is small with respect to $d$, we find
\[\mb{E}Y_k\lesssim_{C,d}e^{9c'k}((1+c)ed\exp(-d)+\eta)^k.\]
Now since $c,c'\ll\eta\ll\delta$ we have $(1+c)ed\exp(-d)+\eta\le\exp(-9c'-c)$ for all $d\ge1+\delta$ (since $ex\exp(-x)<1$ for $x>1$, similar to the conclusion of the proof of Lemma~\ref{lem:unstructured-graph-1}). This concludes the proof. 
\end{proof}

\subsection{Proofs of quasi-randomness properties}

We now turn to prove Lemmas~\ref{lem:expansion-epoch1}, \ref{lem:expansion-epoch2} and \ref{lem:expansion-epoch2-spec}. For this we first record the following basic property.

\begin{lemma}\label{lem:dense-sets}
Let $A$ be as in Section~\ref{sec:setup}. There exists an absolute constant $C>0$ such that the following holds. Let $1\le t\le s\le n\exp(-Cd)$. Then with probability at least $1-e^{-s}n^{-1/2}$, we have 
\[\sup_{\substack{|S| = s\\|T| = t\\ S\cap T = \emptyset}} \sum_{\substack{i\in S\\ j\in S\cup T}} A_{ij}\le s + t + \frac{Cs\log(2d)}{\log(n/s)}.\] 
\end{lemma}
\begin{proof}
Let $g(s) = \Big\lfloor\frac{Cs\log(2d)}{\log(n/s)}\Big\rfloor$. We have
\begin{align*}
\mb{P}\bigg[\sup_{\substack{|S| = s\\|T| = t\\ S\cap T = \emptyset}}\sum_{\substack{i\in S\\ j\in S\cup T}}A_{ij}\ge s+t+g(s)+1\bigg]&\le \binom{n}{s}\binom{n}{t}\mb{P}\bigg[\sum_{\substack{i\in [s]\\ j\in [s+t]}}A_{ij}\ge s+t+g(s)+1\bigg]\\
&\le \bigg(\frac{en}{s}\bigg)^{s}\bigg(\frac{en}{t}\bigg)^{t}\binom{s(s+t)}{s+t+g(s)+1}\bigg(\frac{d}{n}\bigg)^{s+t+g(s)+1}\\
&\le \bigg(\frac{en}{s}\bigg)^{s}\bigg(\frac{en}{t}\bigg)^{t}(es)^{s+t+g(s)+1}\bigg(\frac{d}{n}\bigg)^{s+t+g(s)+1}\\
&\le (ed)^{4(s+t)}n^{-g(s)-1}s^{-s}t^{-t}s^{s}s^{t}s^{g(s)+1}\\
&\le(ed)^{8s}(s/t)^{t}(s/n)^{g(s)+1}\le e^s(ed)^{8s}(s/n)^{g(s)+1}\\
&\le e^{-s}n^{-1/2}.
\end{align*}
The final inequality is trivial to check when $g(s)\neq 0$ and when $g(s)=0$, $(3d)^{40s}\le n$ and hence the desired inequality also holds.
\end{proof}

\begin{proof}[Proof of Lemma~\ref{lem:expansion-epoch1}]
It suffices to prove that for any $m\le t\le n-\ell,$ we have
\[\mb{P}\big(B_t\in \mc{U}((\log n)^{3/2})\big) \ge 1-n^{-\omega(1)};\]
the case of $B_t^\dagger$ follows by symmetry. Note that $B_t$ has a $(d, \eps^{1/2}, 16)$-regular degree sequence with probability $1-n^{-\omega(1)}$. Therefore by Lemma~\ref{lem:unstructured-graph-2} applied with $c = c(\min(d-1,1/d),16)$, it follows that the probability there exists a set $S$ of $k$ with $|\sum_{v\in S}\deg_{B_t}^{+}(v)-k|\le k/\sqrt{\log(n/k)}$ and $|U(S)|\le \alpha(k)k$ is at most $c^{-1}\exp(-ck) + n^{-\omega(1)}$ by Markov's inequality. We can take a union bound over $k\ge(\log n)^{3/2}$ to handle these $S$.

We now handle $S$ such that $|\sum_{v\in S}\deg_{B_t}^{+}(v)-k|\ge k/\sqrt{\log(n/k)}$.
If $\sum_{i,j\in S}(B_t)_{ij}\le k - k/(4\sqrt{\log(n/k)})$, note that 
\[|S\cap U(S)|\ge |S|-\sum_{i,j\in S}(B_t)_{ij}\ge k/(4\sqrt{\log(n/k)})\ge k\alpha(k).\]
Thus we may restrict to sets $S$ such that $\sum_{v\in S}\deg_{B_t}^{+}(v)\ge k + k/\sqrt{\log(n/k)}$ and $\sum_{i,j\in S}(B_t)_{ij}\ge k - k/(4\sqrt{\log(n/k)})$.

Further assuming that $U(S)\le\alpha(k)k$ occurs, there exists a set $T$ such that $|T| \le 3k/(8\sqrt{\log(n/k)})$ with $T\cap S = \emptyset$ such that 
\[\sum_{\substack{i\in S\\j\in S\cup T}}(B_t)_{i,j}\ge k + \frac{k}{2\sqrt{\log(n/k)}}.\]
In particular, $T$ can be taken to be the set of vertices outside $S$ with at least $2$ neighbors in $S$, truncating $T$ to the appropriate size if it is too large. This contradicts Lemma~\ref{lem:dense-sets} (which holds with probability $\ge 1-e^{-k}n^{-1/2}$) since
\[k + |T| + \frac{Ck\log(2d)}{\log(n/k)}< k + \frac{k}{2\sqrt{\log(n/k)}}\]
with $C$ as in Lemma~\ref{lem:dense-sets} and since $\kappa\ll1/d$.
\end{proof}

\begin{proof}[Proof of Lemma~\ref{lem:expansion-epoch2}]
The result follows immediately from Lemma~\ref{lem:unstructured-graph-1} and taking the union bound over sizes larger than $(\log\log n)^2$ if $B_t$ has independent $\on{Ber}(d/n)$ entries. We can couple to the correct model with a polynomial loss in TV-distance by the remarks after Lemma~\ref{lem:TV-estimate}.
\end{proof}

\begin{proof}[Proof of Lemma~\ref{lem:expansion-epoch2-spec}]
Unique-neighbor expansion of sets for size larger than $(\log\log n)^2$ follows immediately from Lemma~\ref{lem:expansion-epoch2}. The remaining result follows immediately from Lemma~\ref{lem:case-exceptional} and the remarks following Lemma~\ref{lem:TV-estimate} to account for slight differences in the random model.
\end{proof}

\section{Spreadness of near kernel vectors}\label{sub:unstructuredness}

We now prove the crucial vector ``spreadness'' estimates for vectors that are ``close'' to the kernel of our matrix, as we discussed in Section~\ref{sec:sketch-spreadness}, and in particular in \eqref{eq:sketch-spreadness}. We will show that matrices that satisfy the quasi-randomness events $\mc{D}$, $\mc{U}(r)$, and $\mc{U}^\ast$ (which we defined in Sections~\ref{sec:deg-seq} and \ref{sec:UNE}) have the property that vectors close to their kernel are spread.  

%for our results, Propositions~\ref{prop:unstructured-main-1},~\ref{prop:unstructured-main-2}, and~\ref{prop:unstructured-main-3}. Recall the quasi-randomness events $\mc{D}, \mc{U}(r),and \mc{U}^\ast$, which we defined in Sections~\ref{subsec:deg-seq} and \eqref{sec:UNE} respectively. 

The main objective of this section is to prove the following three propositions. Our first proposition, Proposition~\ref{prop:unstructured-main-1}, will handle the first epoch, while Propositions~\ref{prop:unstructured-main-2} and \ref{prop:unstructured-main-3} will handle the second epoch. It is perhaps useful to recall from our discussion in Section~\ref{sec:sketch-spreadness} that the condition $v_k^\ast\ge(k/t)^2/\sqrt{t}$ in the following proposition will ultimately be guaranteed by a lemma of Litvak, Lytova, Tikhomirov, Tomczak-Jaegermann, and Youssef (Lemma~\ref{lem:basis}) that provides us with suitable a basis of the space spanned by the smallest singular directions. Here $v_{k}^{\ast}$ denotes the magnitude of the $k$-th largest coordinate of a vector $v$. 

%In the following lemma $m,n$ should be understood in the global context of the proof. 
%\JJ{we are not really clear what we mean by $\eps \ll 1/d$ }\edit{MS: All we mean is a constant which is sufficiently small in terms of $d$? ($d$ is an ansolute constant so $\eps\le \eps_{\ref{prop:unstructured-main-1}}(d)$ is all that is meant.}
\begin{prop}\label{prop:unstructured-main-1}
Let $\eps \ll 1/d$ and $z\in\mb{C}\setminus\{0\}$, then there exist constants $c=c(d)>0$ and $C'=C'(d,z)>0$ so that the following holds. 

Let $m\le t\le n-\ell$ and let $M\in\{B_t,B_t^\dagger\}$ be such that $M\in\mc{U}((\log n)^{3/2}) \cap \mc{D}$. If $(\log n)^{7/4}\le k\le n$, $v\in\mb{C}^t$ is a vector with $v_k^\ast\ge(k/t)^2/\sqrt{t}$, and
\[\snorm{(M-zI_t)v}_2\le\exp(-C'(\log(2n/k))^6),\]
then 
\[\sup_{\theta\in\mb{R}}\Big|\Big\{ i\colon|v_i-\theta|\le n^{-1/2} \cdot \exp(-C'(\log(2n/k))^7) \Big\}\Big|\le(1-c)n.\]
\end{prop}

The following two propositions will be used to handle the second epoch. For the following, we recall that we set $\ell = \lfloor (\log n)^2 \rfloor$ in Section~\ref{sec:setup}.

\begin{prop}\label{prop:unstructured-main-2}
Let $\eps\ll 1/d$ and $z\in\mb{C}\setminus\{0\}$. Then there are constants $c=c(d)>0$ and $C'=C'(d,z)>0$ so that the following holds. 

Let $n-\ell\le t\le n-1$ and let $M=B_t^{\dagger}$ be such that $M\in\mc{U}((\log\log n)^2) \cap \mc{D}$. If $v\in\mb{C}^t$ is a unit vector and
\[\snorm{(M-zI_t)v}_2\le\exp(-C'(\log n)^6),\]
then
\[v_{\lfloor cn\rfloor}^\ast\ge n^{-1/2} \cdot \exp(-C'(\log n)^7) .\]
\end{prop}

For the second epoch we also need the following Proposition. Actually the proof of this is strictly more complicated than the proof of Proposition~\ref{prop:unstructured-main-2}, and so we omit the proof of the previous proposition. 

\begin{prop}\label{prop:unstructured-main-3}
Let $\eps\ll 1/d$ and $z\in\mb{C}\setminus\{0\}$ be such that $|z|\neq 1$. Then there are constants $c=c(d)>0$ and $C'=C'(d,z)>0$ such that the following holds. 

Let $n-\ell\le t\le n-1$ and let $M=B_t^\ast$ satisfy $M\in\mc{U}((\log\log n)^2) \cap \mc{U}^{\ast} \cap \mc{D}$. If $v\in\mb{C}^t$ is a unit vector and
\[\snorm{(M-zI_{(t-1)\times t})v}_2\le\exp(-C'(\log n)^6),\]
then
\[v_{\lfloor cn\rfloor}^\ast\ge n^{-1/2} \cdot \exp(-C'(\log n)^7).\]
\end{prop}

\vspace{1em}

\subsection{Initial estimates and setup}
We first require the connection between unique neighborhood expansion and the images of vectors.
\begin{observation}\label{obs:NU-obs}
Let $M$ be a $(t-1)\times t$ or $t\times t$ dimensional $\{0,1\}$-matrix. For $\ell\leq t$, let $v \in \C^{t}$ and let $S$ be the set of the $\ell$ largest coordinates of $v$ in absolute value. Then $|((M-zI)v_S)_i| \geq v^{\ast}_{\ell}\min(|z|,1)$ for all $i \in U(S)$ where $U(S)$ is defined with respect to the matrix $M$ and $I$ is the identity matrix with dimensions corresponding to $M$.
\end{observation}
\begin{proof}
We consider two cases. If $i\in U(S)\setminus S$ there is unique $j\in S$ with $(M-zI)_{ij}\neq 0$. As $M$ is a $\{0,1\}$-matrix, we additionally have $M_{ij} = 1$ and we have $|((M-zI)v_S)_i|=|v_i|\ge v^{\ast}_{\ell}$. This proves the observation in this case.

For $i\in U(S)\cap S$ we have $M_{ij} = 0$ for all $j\in S$. So 
\[|((M-zI)v_S)_i|=|(M-zI)_{ii}v_i| = |(-z)v_i|=|z||v_i| \geq v_\ell^\ast|z|,\]
which proves the observation.
\end{proof}

Recall that above we defined the function $\alpha(x) = (\log (n/x) )^{-2}$. Here we define the function 
\begin{equation}\label{eq:g(k)-def}  
g(x)= \left\lceil \frac{\alpha(x)x}{2^{15}(d+\log (n/x))}\right\rceil.
\end{equation} 

We require the following trivial iteration lemma.
\begin{lemma}\label{lem:iteration}
Let $k_0 = k$ and define
\[ k_i = k_{i-1} + g(k_{i-1}),\]
for all $i \geq 1$. Let $\tau$ be the minimal value such that $k_{\tau}\ge n/2.$ Then $\tau \le 2^{17}d(\log (n/k))^4$.
\end{lemma}
\begin{proof}
For $k\ge n/2$, the result is trivial. Furthermore note that it takes at most $2^{15}(d+\log(n/k))/\alpha(k)\le 2^{16}d(\log(n/k))^3$ steps to double. As there are at most $2\log(n/k)$ doublings required, the desired result follows immediately.
\end{proof}

We finally will require the following graph-theoretic estimate which will allow us to eliminate graphs with extremely large level sets.

\begin{lemma}\label{lem:large-zero-level-set}
Consider an $\ell\times m$ dimensional $\{0,1\}$-matrix $M$ with $\ell\in \{m-1,m\}$, $z\neq 0$, and let $\theta > 0$. Suppose that $M$ has at least $e^{-d}n/2$ vertices with in-degree zero (i.e., zero rows). Then for a unit vector $v\in \mb{C}^m$ such that 
\[\snorm{(M-zI_{\ell\times m})v}\le\theta|z|,\]
there are at least $e^{-d}n/4$ indices $j$ such that 
\[|v_j|\le 2e^{d/2}\theta n^{-1/2}.\]
\end{lemma}
\begin{proof}
Note that for any index $i\in[\ell]$ such that $\deg_M^-(i) = 0$, we have 
\[((M-zI_{\ell\times m})v)_i = -zv_i.\]
Therefore since $\snorm{(M-zI_{\ell\times m})v}\le \theta |z|$, we have 
\[\sum_{\deg_M^-(i) = 0}|zv_i|^2 \le \theta^2|z|^2.\]
Applying Markov's inequality we derive the desired conclusion. 
\end{proof}

\subsection{Unstructured almost-kernel vectors for the first epoch}
We are now in position to prove Proposition~\ref{prop:unstructured-main-1}.
\begin{proof}[Proof of Proposition~\ref{prop:unstructured-main-1}]
As $M\in \mc{D}$, we have that $M$ has at least $e^{-d}n/2$ vertices with out-degree zero. Therefore by Lemma~\ref{lem:large-zero-level-set}, we have that 
\[\big|\{|v_i|\le 2e^{d/2}|z|^{-1}\exp(-C'(\log(2n/k))^6)/\sqrt{n}\}\big|\ge e^{-d}n/4.\]
Now taking $C'$ sufficiently large, we have, for all $|\theta|\ge 2\exp(-C'(\log(2n/k))^{7})/\sqrt{n}$, that 
\[\Big|\Big\{ i\colon|v_i-\theta|\le n^{-1/2} \cdot \exp(-C'(\log(2n/k))^7) \Big\}\Big| \ge e^{-d}n/4.\]
Thus it suffices to prove that 
\[\Big|\Big\{ i\colon|v_i|\le4 n^{-1/2} \cdot \exp(-C'(\log(2n/k))^6)\Big\}\Big|\le (1-c)n.\]

By assumption, we have that $v_k^{\ast}\ge (k/t)^2/\sqrt{t}$. We claim that it suffices to prove for all $k'\in [k,cn]$ that 
\[v_{k'+g(k')}^{\ast}\ge v_{k'}^{\ast} \cdot \left( \frac{k'}{d n \min(|z|,|z|^{-1})} \right)^2.\]
This immediately implies the desired result since then
\[v_{cn}^\ast\ge v_k^{\ast}\cdot \prod_{i=1}^{\tau}\left( \frac{k_i}{d n \min(|z|,|z|^{-1})} \right)^2\ge \exp(-C'\log(n/k)^5)/\sqrt{n},\]
choosing $C'$ sufficiently large and defining $k_i,\tau$ with bounds as in Lemma~\ref{lem:iteration}.

To prove the claim, suppose that 
\[v_{k'+g(k')}^{\ast}<v_{k'}^{\ast} \cdot \left( \frac{k'}{d n \min(|z|,|z|^{-1})} \right)^2\]
and let $k'$ is chosen to be minimal such value in $[k,cn]$. Let $S$ denote the indices of the largest $k'$ coordinates of $v$ and $S'$ denote the indices of the coordinates with magnitude from the $(k'+1)$st largest to the $(k' + g(k'))$th largest. By Observation~\ref{obs:NU-obs}, for all $j\in U(S)$ we have
\[|((M-zI)v_S)_j|\ge v_{k'}^{\ast}\min(|z|,1).\]
Since $M\in\mc{U}((\log n)^{3/2})$, we have $|U(S)|\ge \alpha(k')k'$. Since $M\in\mc{D}$, by the second item of Definition~\ref{def:deg-seq}, we have that there are at most 
\[16(d + \log(m/g(k')))g(k')\le \alpha(k')k'/4\]
neighbors of $S'$. Furthermore by the second item of Definition~\ref{def:deg-seq}, for $t\ge 32d$ there are at most $\exp(-t/32)n$ vertices of in-degree larger than $t$. In particular, there are at most $\alpha(k')k'/4$ of in-degree larger than $n/k'$ assuming that $c$ is a sufficiently small function of $d$.

Therefore, define a row index $j$ in $U(S)$ to be suitable if it has in-degree bounded by $n/k'$ and is not adjacent to any column index in $S'$. Note we have proven that there are at least $\alpha(k')k'/2$ suitable indices. For each suitable index $j$, we have that 
\begin{align*}
|((M-zI)v)_j| &= |((M-zI)(v_S + v_{S'} + v_{([t]\setminus (S\cup S'))})_j|\\
&= |((M-zI)(v_S + v_{([t]\setminus (S\cup S'))})_j|\\
&\ge v_{k'}^{\ast}\min(|z|,1) - \bigg(\frac{n}{k'} + |z|\bigg)v_{k'+g(k')}^{\ast}\ge v_{k'}^{\ast}\min(|z|,1)/2.
\end{align*}
This gives us a contradiction as $v_{k'}^{\ast}\ge \exp(-C'(\log(2n/k))^5)/\sqrt{n}$ and therefore
\begin{align*}
\snorm{((M-zI)v)}_2&\ge \big(\alpha(k')k'/2\big)^{1/2}\cdot \exp(-C'(\log(2n/k))^5)\min(|z|,1)/(2\sqrt{n})\\
&\ge \exp(-C'\log(n/k)^6)
\end{align*}
since $C'$ is sufficiently large as a function of $d$ and $z$.
\end{proof}

\subsection{Unstructured almost-kernel vectors for the second epoch}
In the second epoch, we split our analysis into two cases. For large support almost-kernel vectors, we use a similar argument to the proof of Proposition~\ref{prop:unstructured-main-1}. On the other hand, small support almost-kernel vectors must essentially be kernel vectors of circulant matrices, which we can explicitly handle using the following Lemma~\ref{lem:lsv-explicit}.
\begin{lemma}\label{lem:lsv-explicit}
Let $Y$ be the $s\times s$ matrix where $Y_{ij} = 1$ if and only if $i \equiv j + 1 \pmod{s}$. Then 
\[\sigma_s(Y-zI_s)\ge|z^s-1|/(|z|+1)^{s-1}.\]
\end{lemma}
\begin{proof}
Notice that by direct computation that 
\[\det((Y-zI_s)^{\dagger}(Y-zI_s)) = |z^s-1|^2\]
and that
\[\sigma_1(Y-zI_s)\le|z| + 1.\]
The desired result follows from 
\[\det((Y-zI_s)^{\dagger}(Y-zI_s)) \le\sigma_s(Y-zI_s)^2 \cdot\sigma_1(Y-zI_s)^{2(s-1)}\]
and dividing.
\end{proof}

\begin{proof}[Proof of Proposition~\ref{prop:unstructured-main-3}]
It suffices to prove that $v_{k'+1}^{\ast}\ge v_{k'}^{\ast}/(C'n)$ for $k'\le(\log\log n)^3$. This is sufficient since after this point a modification of the argument in Proposition~\ref{prop:unstructured-main-1} easily completes the proof.

Suppose that $v_{k'+1}^{\ast}\le v_{k'}^{\ast}/(C'n)$ for a minimal such $k'\le(\log\log n)^3$, and note that since $v$ is a unit vector we have $v_1^{\ast}\ge 1/\sqrt{n}$. Let $S$ be the set of indices corresponding to the $k'$ largest coordinates. Furthermore notice that the in- and out-degree of every vertex is bounded by $(\log n)^2$ since $M\in\mc{D}$. If $|U(S)|\ge 1$, for any $j\in U(S)$ we have
\[|((M-zI_{(t-1)\times t})v)_j|\ge v_{k'}^{\ast}\cdot \min(|z|,1) - ((\log n)^2 + |z|)v_{k' + 1}^{\ast}\ge v_{k'}^{\ast}\cdot \min(|z|,1)/2.\]
Notice that $v_{k'}^{\ast}\ge\exp(-C'(\log n)^2)$ which provides the desired contradiction in this case.

Therefore we may assume that $U(S)=\emptyset$. By $\mc{U}^{\ast}$, we deduce that $t\notin S$. Furthermore, since $t\notin S$ and $U(S) = \emptyset$, every vertex in $S$ has at least $1$ in-neighbor from $S$. By the final condition of $M\in\mc{D}$, we have
\[\sum_{i,j\in S}M_{ij}\le|S|\]
and thus every vertex in $S$ has exactly $1$ in-neighbor from $S$. We refine $S$ as follows. If there is no vertex in the current set with out-degree zero terminate; else remove a vertex of out-degree zero and iterate. This process terminates with a set $T$ in which the induced directed subgraph is exactly a collection of cycles (possibly of length $1$). Furthermore for any vertex in $T$, we have that it has no in-neighbor from $S\setminus T$. Adding $-z$ to the diagonal entries of the adjacency matrix of the induced digraph $M[T]$, we get a disjoint collection of circulant matrices of exactly the form in Lemma~\ref{lem:lsv-explicit}. This argument uses crucially that $t\notin S$ and therefore $t\notin T$.

Finally, applying Lemma~\ref{lem:lsv-explicit} we have, writing $s=(\log\log n)^3$,
\begin{align*}
\snorm{(M-zI_{(t-1)\times t})v}_2 &\ge \snorm{((M-zI_{(t-1)\times t})v)_T}_2 \\
& = \snorm{((M-zI_{(t-1)\times t})v_S)_T}_2 - \sqrt{n}((\log n)^2 + |z|)v_{k'+1}^{\ast}\\ 
& = \snorm{((M-zI_{(t-1)\times t})v_T)_T}_2 - \sqrt{n}((\log n)^2 + |z|)v_{k'+1}^{\ast}\\ 
& \ge \snorm{v_T}_2 \cdot\min\{|z^t-1|/(|z|+1)^{t-1}\colon t\in[s]\}-\sqrt{n}((\log n)^2 + |z|)v_{k'+1}^{\ast}\\ 
& \ge v_{k'}^{\ast} \cdot ||z|-1|/(|z|+1)^{s-1}-\sqrt{n}((\log n)^2 + |z|)v_{k'+1}^{\ast}\\ 
& \ge v_{k'}^{\ast} \cdot ||z|-1|\exp(-C'(\log\log n)^{4})-\sqrt{n}((\log n)^2 + |z|)v_{k'+1}^{\ast}\\
&\ge \exp(-C'(\log n)^2)
\end{align*}
given that $C'$ is a sufficiently large function of $z$ and $d$. This contradicts the assumption in the lemma statement and completes the proof.
\end{proof}

\section{Anticoncentration of the projection}\label{sec:anticoncentration}
In this section we use the ``spreadness'' estimates which we proved in the previous section to show anticoncentration of the projection of a random row vector onto the bottom $r$ singular directions, a step which we discussed in and around \eqref{eq:ker-anti-concentration} in Section~\ref{sec:sketch-spreadness}. 

To set up the main results of this section, we define the following parameters: a constant $K = K(d,z)$, which we  later choose to be sufficiently large, and 
\[\eps_r = \exp(-K(\log(n/r))^9).\]
We will now consider the result of adding in vertex $v_t$ in our walk. From now on we will use the common abuse of writing $e_t$ for the column identity vector which is supported only on the index corresponding to $v_t$. (In fact, we may relabel so that $v_t$ is labeled $t$ if we so desire, since conjugation by a permutation matrix does not change the spectrum nor the singular values.) 

The following two lemmas can be seen as rigorous formulations of \eqref{eq:ker-anti-concentration}, in the first and second epochs, respectively. To state our result on the first epoch, we recall the notation $T_1 = [\lfloor n(1-\eps)\rfloor]$ from \eqref{eq:T1T2T3}.

We also use the following notation in this section. For a matrix $M$ and integer $r$, let $P_{r,M}$ denote the projection onto the smallest $r$ right-singular vectors of $M$. 

\begin{lemma}\label{lem:proj-anti-epoch1} 
Let $z\neq 0$, $m+1\leq t \leq n-\ell$ and $r\ge (\log n)^{5/3}$. Let $B_{t-1}^\dagger,B_{t-1} \in \mc{U}((\log n)^{3/2})\cap\mc{D}$ be such that\footnote{Note that this event is $\mc{F}_t'$-measurable.}
\[\min\big\{ \deg_{B_{n-\ell}}^+(v_t,T_1),\, \deg_{B_{n-\ell}}^-(v_t,T_1) \big\} \ge \sqrt{\log(1/\eps)}.\]

 If $\sigma_{(t-1)-r/2}(B_{t-1}-zI_{t-1,t-1})\le \eps_r$ then, for all $u\in \mb{C}^{t-1}$,
\[ \PP\big( \|P_{r,B_{t-1}^{\dagger} - \ol{z}I_{t-1}}(B_{t}^{\ast}e_{t} + u)\|_2 < \eps_r \big| \mc{F}_t' \big) \le C(\log(1/\eps))^{-1/4}, \]
where $C = C(d)>0$ is a constant depending only on $d$.

Furthermore if $\sigma_{(t-1)-r/2}(B_{t-1}-zI_{t-1,t-1})\le \eps_r$, then for all $u\in \mb{C}^{t}$,
\[ \PP\big( \|P_{r,B_{t}^{\ast} - zI_{(t-1)\times t}}(B_{t}^{\dagger}e_{t} + u)\|_2 < \eps_r | \mc{F}_t', B_{t}^{\ast}\big) \le C(\log(1/\eps))^{-1/4},\]
where $C = C(d) >0$ is a constant depending only on $d$.
\end{lemma}

We now state our result on anticoncentration of projections for the second epoch of our process. 

\begin{lemma}\label{lem:proj-anti-epoch2} 
Let $z \in \mb{C}$ satisfy $|z|\neq 0, 1$, and let $n-\ell+1\le t\le n$. 

If $B_{t-1}$ is such that $\sigma_{t-1}(B_{t-1} - z I_{t-1})\le \eps_1$ and $B_{t-1}^{\dagger}\in \mc{U}((\log\log n)^2) \cap \mc{D}$, then for all $u\in\mb{C}^{t}$,
\[ \PP\big( \|P_{1,B_{t-1}^{\dagger} - \ol{z}I_{t-1}}(B_{t}^{\ast}e_{t} + u)\|_2 < \eps_1 \big| \mc{F}_t \big) \le C(\log n)^{-1/4}. \]
Furthermore, if $B_{t}^{\ast}\in \mc{U}((\log\log n)^2) \cap \mc{U}^{\ast}$ and $B_{t-1}^{\dagger}\in \mc{D}$ then for all $u\in\mb{C}^t$,
\[ \PP\big( \|P_{1,B_{t}^{\ast} - zI_{(t-1)\times t}}(B_{t}^{\dagger}e_{t} + u)\|_2 < \eps_1 \big| \mc{F}_t, B_{t}^{\ast}\big) \le C(\log n)^{-1/4}, \]
where $C = C(d)>0$ is a constant depending only on $d$.
\end{lemma}
\begin{remark}
Observe that in the second part of Lemma~\ref{lem:proj-anti-epoch2}, we will not require a condition on the least singular value of $B_{t}^{\ast}-zI_{(t-1),t}$. This is due to the fact that the matrix is rectangular and therefore is forced to have a kernel vector
\end{remark}

%As we saw in Section~\ref{sec:sketch-enough-drift}, these projection inequalities are designed to complement the crucial linear algebra input in this paper which is Proposition~ \ref{prop:walk-row-modified}. We briefly comment on the assumptions in both Lemma~\ref{lem:proj-anti-epoch1} and Lemma~\ref{lem:proj-anti-epoch2}. In Lemma~\ref{lem:proj-anti-epoch1}, in order to guarantee a substantial ``boost'' we use examine the projection onto the smallest $r$ singular vectors. Ariori these vectors have no structure; however in the case of interest we may assume (say) $\sigma_{(t-1)-r/2}(B_{t-1}-zI_{t-1,t-1})\le \eps_r$ and this guarantee at least $r/2$ of them have large support by our previous analysis. In Lemma~\ref{lem:proj-anti-epoch2}, we will no longer need to track $r$ so carefully and can use only the projection into the bottom singular value at each step. This is ultimately a function of the fact that we simply need out boosts to satisfy $\sum \log(1/\eps_r) \ge -o(n)$ and hence for the final $\ell$-steps we have considerable flexibility. \J{Return to this}

\subsection{Some basic anticoncentration estimates}\label{sec:anti-concentration-estimates}
Before we proceed with the proof of our two main results of this section, Lemmas~\ref{lem:proj-anti-epoch1} and \ref{lem:proj-anti-epoch2}, we lay out a few tools that we will need. 

We first define the classical L\'evy concentration function. For a (real or complex) random variable $\Gamma$, we define 
\[\mc{L}(\Gamma,t) = \sup_{z\in \mb{C}}\mb{P}(|\Gamma-z|\le t).\]

We will also require the following classical anticoncentration inequality due to Kolmogorov, L\'evy and Rogozin \cite{Kol58, Rog61} (see e.g. \cite[Lemma~3.2]{RT19}). 
\begin{lemma}\label{lem:LKR}
Let $\xi_1,\ldots,\xi_{n}$ be independent real or complex random variables. Then, for any real numbers $r > 0$, we have
\[ \mc{L}\bigg(\sum_{i=1}^{n}\xi_i, r\bigg) \le C\bigg(\sum_{i=1}^{n}\big( 1-\mc{L}(\xi_i, r)\big)\bigg)^{-1/2},\]
where $C>0$ is an absolute constant. 
\end{lemma}

We will also require the following ``slice'' anticoncentration inequality; the proof is essentially a quantitative version of the \cite[Lemma~4.2]{FKS22}.

\begin{lemma}\label{lem:slice-anti}
For $1\leq m\leq n/2$, sample $\xi\in \{0,1\}^n$ uniformly among all such vectors with $\sum_{i=1}^{n}\xi_i = m$.
For $\delta, \gamma > 0$, let $v = (v_1,\ldots,v_n) \in \mb{C}^{n}$ be such that 
\[\sup_{\theta\in \mb{C}}\big|\big\{i\colon|v_i-\theta|\le \delta\big\}\big| \le (1-\gamma)n.\]
Then
\[\mc{L}\Big(\sum_{i=1}^{n}\xi_iv_i, \delta\bigg) \le C\Big((\gamma m)^{-1/2} + \exp(-C^{-1}\gamma^2 m)\Big),\]
where $C>0$ is an absolute constant. 
\end{lemma}
\begin{proof}
Let $\pi$ be a uniformly random
injective function $\{1,\ldots,2m\}\to\{1,\ldots,n\} $
and let $\mbf{x}=(x_{1},\ldots,x_{m})$ be
a sequence of independent $\mr{Ber}(1/2)$ random variables. Then choose the positions of the $m$ ones in $\xi$
as follows. For each $i\in\{1,\ldots,m\} $, if $x_i=0$
we set $\xi_{\pi(i)}=1$ and $\xi_{\pi(i+m)}=0$, and if $x_i=1$ let $\xi_{\pi(i+m)}=1$ and $\xi_{\pi(i)}=0$. All other indices of $\xi$ are set to $0$. It is clear that $\xi$ has the correct distribution.

We call an index $i$ separated if 
\[|v_{\pi(i)}-v_{\pi(i+m)}|\ge\delta\]
and let $Y$ denote the number of separated indices, which is dependent only on $\pi$. Note that 
\[\big|\big\{(i,j)\colon|v_i-v_j|\ge\delta\big\}\big| \ge \gamma n^2/2,\]
else there exists an index $i$ such that 
\[\big|\big\{j\colon|v_i-v_j|\le\delta\} \ge (1-\gamma)n\big\}\big|\]
and taking $\theta = v_i$ we have violated our assumption. Thus $\mb{E}Y \gtrsim \gamma m$ and by applying Lemma~\ref{lem:injection-concentration} we easily see $\mb{P}(Y\le \mb{E}Y/2)\le \exp(-\Omega(\gamma^2m))$.

Therefore, 
\begin{align*}
&\sup_{\theta\in\mb{C}} \mb{P}\Big(\Big|\sum_{i=1}^n\xi_iv_i-\theta\Big|\le \delta\Big) \\
&\le \mb{P}(Y\le \mb{E}Y/2) + \sup_{\theta, \pi} \mbm{1}(Y\ge \mb{E}Y/2) \cdot \mb{P}\Big(\Big|\sum_{i=1}^m\xi_{\pi(i)}v_{\pi(i)} +\xi_{\pi(i+m)}v_{\pi(i+m)} -\theta\Big|\le \delta \, \Big|\, \pi\Big) \\
&\le \mb{P}(Y\le \mb{E}Y/2) + \sup_{\theta, \pi} \mbm{1}(Y\ge\mb{E}Y/2) \cdot \mb{P}\Big(\Big|\sum_{i=1}^m(v_{\pi(i)}+x_i(v_{\pi(i+m)} - v_{\pi(i)}))-\theta\Big|\le \delta\,  \Big|\, \pi\Big) \\
&\lesssim \exp(-\Omega(\gamma^2m)) + (\gamma m)^{-1/2}
\end{align*}
where we have used that $(x_i)_{i\in[m]}$ is distributed as $\mr{Ber}(1/2)^{\otimes m}$ given $\pi$ and then applied Lemma~\ref{lem:LKR} to the set of separated $i$ where $|v_{\pi(i+m)}-v_{\pi(i)}|\ge\delta$ (i.e., those counted by $Y\ge\mb{E}Y/2\gtrsim\gamma m$).
\end{proof}

\vspace{1em}

We also require the following result that provides us with a relatively spread basis of an arbitrary subspace. The following lemma is due to Litvak, Lytova, Tikhomirov, Tomczak-Jaegermann, and Youssef \cite[Lemma~4.3]{LLTTY21}. 

\begin{lemma}\label{lem:basis}
Let $V \subseteq \C^n$ be a $k$-dimensional $\C$-vector space. There exists an orthonormal basis $B$ of $V$ so that for all $v \in B$, we have $v^{\ast}_{ck} \geq  ck^{1/2}n^{-1}$, where $c>0$ is an absolute constant. 
\end{lemma}

\subsection{The proof of Lemma~\ref{lem:proj-anti-epoch1} and ~\ref{lem:proj-anti-epoch2}}

We are now in position to prove Lemma~\ref{lem:proj-anti-epoch1}.
\begin{proof}[Proof of Lemma~\ref{lem:proj-anti-epoch1}]
We only prove the second item of the lemma; the first item is strictly simpler. Recall that $B_{t}^{\ast}$ is a $(t-1)\times t$ matrix and note that $B_{t}^{\dagger}e_t$ is the Hermitian conjugate of the row added to $B_{t}^{\ast}$ to obtain $B_{t}^{\dagger}$. By assumption we have 
\[\sigma_{(t-1) - r/2}(B_{t-1} - zI_{t-1})\le \eps_r.\] By Cauchy interlacing applied for singular values (see e.g.~Fact~\ref{fact:cauchy-interlacing}) we have that 
\[ \sigma_{t - r/2}(B_{t}^{\ast} - z I_{(t-1)\times t})\le \eps_r.\]

Therefore there exists a vector space $W\subseteq \mb{C}^t$ of dimension $r/2$ so that for all unit vectors $v$ in $W$ we have $\snorm{(B_{t}^{\ast} - z I_{(t-1)\times t})v}_2\le \eps_r$. Let $W' = W\cap \{v\in \mb{C}^t:v_t=0\}$ and note that the dimension of $W'$ is at least $r/2-1\ge r/4$. Letting $\pi\colon(x_1,\ldots,x_t)\to (x_1,\ldots,x_{t-1})$, we have for all unit vectors $v$ in $W'$ that
\[\snorm{(B_{t}^{\ast} - z I_{(t-1)\times t})v}_2 = \snorm{(B_{t-1}-zI_{t-1})\pi(v)}_2\le \eps_r.\] Applying Lemma~\ref{lem:basis}, there exist orthogonal unit vectors $w_1,\ldots,w_{r/4}$ in $W'$ such that 
\[(w_j)^{\ast}_{cr}\geq cr^{1/2}n^{-1}\]
for all $j\le r/4$ and an absolute constant $c>0$. Since \[ \snorm{(B_{t}^{\ast} - z I_{(t-1)\times t})w_j}_2 = \snorm{(B_{t-1}-zI_{t-1})w_j}_2\le \eps_r,\] we may use Proposition~\ref{prop:unstructured-main-1}. We have
\[\sup_{\theta\in\mb{C}}\big|\big\{\, i \, \colon \, |(w_j)_i-\theta|\le n^{-1/2} \cdot \exp(-C'(\log(n/r))^{7})\big\}\big| \le (1-c')n\]
for all $1\le j\le r/4$, where $c'$ is a constant depending only on $d$ while $C'$ is a function of $d$ and $z$.

Note that $B_t^{\dagger}e_t$ given $\mc{F}_t'\cup B_t^{\ast}$ is deterministic on the indices outside of $T_1$ and on the indices $T_1$ is a uniformly random $\{0,1\}$-vectors with a fixed sum by Fact~\ref{fact:row-random}. Furthermore note that since $|T_1| = \lfloor n(1-\eps)\rfloor \ge n(1-c'/4)$ (as $\eps\ll1/d$), we still have that the largest approximate level set of each $w_j$ when restricted to $T_1$ occupies at most a $(1-c'/2)$-fraction.

Note that the fixed sum of $B_t^{\dagger}e_t$ on $T_1$ is at least $\sqrt{\log(1/\eps)}$ by assumption and at most $(\log n)^2$ by the maximum degree assumption implicit in $\mc{D}$. Therefore for any deterministic vector $u$ we have
\begin{align*}
&\mb{P}\big( |\sang{w_k,B_t^{\dagger}e_t + u}|\le \exp(-C'(\log(n/r))^7)n^{-1/2}|\mc{F}_t',B_t \big) \\
&\le\sup_{\theta\in\mb{C}}\mb{P}\big(|\sang{(w_k)_{T_1},(B_t^{\dagger}e_t)_{T_1}}-\theta|\le \exp(-C'(\log(n/r))^7)n^{-1/2}|\mc{F}_t',B_t\big) \lesssim (\log(1/\eps))^{-1/4},
\end{align*}
where the inequality follows from Lemma~\ref{lem:slice-anti} (note that the implicit constant here depends only on $c'$ in the size of the largest level set and hence only on $d$). 

Note that for a set of nonnegative random variables $X_1,\ldots,X_k$, by Markov's inequality
\begin{align*}
\mb{P}\bigg(\sum_{i=1}^{k}X_i\le \tau\bigg) \le \mb{P}\big(|\{i\colon X_i\le 2\tau/k\}|\ge k/2\big)\le \frac{2}{k}\mb{E}\bigg[\sum_{i=1}^{k}\mbm{1}_{X_i\le 2\tau/k}\bigg]\le 2\sup_{i\in[k]}\mb{P}(X_i\le 2\tau/k).
\end{align*}
Therefore
\begin{align*}
\PP\big( \|P_{r,B_{t}^{\ast}-zI_{(t-1)\times t}}(B_{t}^{\dagger}e_{t} + u)\|_2 < \eps_r | \mc{F}_t' \cup B_{t}^{\ast}\big) &\le \PP\bigg( \sum_{k=1}^{r/4}|\sang{w_k,B_{t}^{\dagger}e_{t} + u}|^2 < \eps_r^2 | \mc{F}_t', B_{t}^{\ast}\bigg)\\
&\le \sup_{k}2\PP\big(|\sang{w_k,B_{t}^{\dagger}e_{t} + u}| < \eps_r\cdot (r/8)^{-1/2} | \mc{F}_t',B_{t}^{\ast}\big) \\
&\lesssim (\log(1/\eps))^{-1/4},
\end{align*}
which is precisely the desired result. Here we have used that the constant $K$ defining $\eps_r$ is chosen as a sufficiently large function of $d$ and $z$. 
\end{proof}

\vspace{1em}

We now turn to prove Lemma~\ref{lem:proj-anti-epoch2}; since the randomness in the second epoch is purely independent, the analysis simplifies substantially in this case.
\begin{proof}[Proof of Lemma~\ref{lem:proj-anti-epoch2}]
We will only prove the second case; the first is essentially identical except we use Proposition~\ref{prop:unstructured-main-2}. Let $w$ denote the least singular vector of $B_t^{\ast} - zI_{(t-1)\times t}$; as this is a $(t-1)\times t$ matrix we have $(B_t^{\ast}-zI_{(t-1)\times t})w = 0$. By Proposition~\ref{prop:unstructured-main-3}, we have 
\[w^{\ast}_{cn}\ge \exp(-C'(\log n)^{7})n^{-1/2} =: \gamma\]
with $C' = C'(d,z)>0$ and $c = c(d)>0$. Therefore
\begin{align*}
\PP\big( \|P_{1,B_{t}^{\ast}-zI_{(t-1)\times t}}(B_{t}^{\dagger}e_{t} + u)\|_2 < \eps_1 | \mc{F}_t, B_{t}^{\ast}\big)&=\PP\big(|\sang{w,B_{t}^{\dagger}e_{t} + u}|\le \eps_1| \mc{F}_t, B_{t}^{\ast}\big)\\
&\le \sup_{\theta\in \mb{C}}\PP\big(|\sang{w,B_{t}^{\dagger}e_{t} }-\theta|\le \eps_1 | \mc{F}_t,B_{t}^{\ast}\big)\\
&= \sup_{\theta\in\mb{C}}\PP_{\delta_j\sim\on{Ber}(\sqrt{\log n}/n)}\Big(\Big|\sum_{j=1}^{t}w_j\delta_j - \theta\Big|\le \eps_1 \Big)\\
&\le \frac{1}{\sqrt{\sum_{j=1}^{t}(\sqrt{\log n}/n)\cdot \mbm{1}_{|w_j|\ge \gamma}}}\lesssim (\log n)^{-1/4}
\end{align*}
We have used Fact~\ref{fact:end-ind} to rewrite $\sang{w,B_{t}^{\dagger}e_{t}}$ as a sum of weighted independent random Bernoulli random variables and applied Lemma~\ref{lem:LKR} with $r = \gamma$ and $\xi_i = w_i \delta_i$. The implicit constant in the final inequality comes from the implicit constant on the number of coordinates in $w$ with size larger than $\gamma$, which depends only on $d$ by Proposition~\ref{prop:unstructured-main-3}.
\end{proof}

\section{Random walk lemma}\label{sec:random-walk}
In this short section we state and prove a simple probabilistic lemma on random walks of the type discussed in Section~\ref{sec:sketch-random-walk} of the sketch. In particular, using the notation introduced in that section, we show that if $(X_t)_t$ has sufficient downward drift then $X_n = O(1)$ with high probability. Results of this form originate in work of Costello, Tao, and Vu \cite{CTV06} on singularity of symmetric random matrices, and have been used to study singularity and rank in sparse random matrices \cite{FKSS23,GKSS23}.
\begin{lemma}\label{lem:random-walk}
Let $(X_s)_{s = 1}^k$ be a sequence of random variables and let $(\cF_s)_{s=1}^k$ be a filtration for which $X_s$ is $\cF_s$-measurable. Suppose $X_1\le k/4$,  $X_{i+1}\le X_i+1$, and if $X_t\ge\lfloor(k-t)/8\rfloor$, we have $\mb{P}(X_{t+1}\le X_t-1+\mbm{1}_{X_t=0}|\mc{F}_s)\ge 1-p$. 
Then for $t\ge 1$, we have
\[\mb{P}(X_k\ge t)\le (Cp)^{t/2},\]
where $C>0$ is an absolute constant. 
\end{lemma}
\begin{proof}
We will choose $C$ to be sufficiently large at the end of the proof. We may assume that $p\le C^{-1}$ as otherwise the result is vacuous. Furthermore we may assume that $k\ge 5$; for $k\le 4$ we have $\mb{P}(X_k \ge t)\le \binom{k}{t}p^{t}\le 16p^{t/2}$. 

Define $Y_t=(1/\sqrt{p})^{X_t}-1$. We claim that 
\begin{equation}\label{eq:walk}
\mb{E}[Y_{t+1}|\mc{F}_t]\le p^{-\lfloor(k-t)/8\rfloor/2}+ (2\sqrt{p})Y_t + 2\sqrt{p}\le 3p^{-\lfloor (k-t)/8\rfloor/2} + (2\sqrt{p})Y_t.
\end{equation}
Indeed, notice that if $X_t<\lfloor (k-t)/8\rfloor$ then we have $X_{t+1}\le X_t+1\le\lfloor(k-t)/8\rfloor$ and the result follows immediately. Else if $X_t>0$ and $X_t\ge\lfloor (k-t)/8\rfloor$ then we have
\[\mb{E}[Y_{t+1}+1|\mc{F}_t]\le\sqrt{p}(Y_t+1)+p\cdot(1/\sqrt{p})(Y_t+1),\]
which also implies \eqref{eq:walk}. Finally if $X_t=0$ and $X_t\ge\lfloor(k-t)/8\rfloor$ then we have
\[\mb{E}[Y_{t+1}+1|\mc{F}_t]\le (1/\sqrt{p}-1)p\le\sqrt{p},\]
which gives \eqref{eq:walk}; thus we have verified \eqref{eq:walk} in all cases. Let $Z_t = p^{(k-t)/8}Y_t$ and note for $t\le k-1$ that
\[\mb{E}[Z_{t+1}|\mc{F}_t]\le p^{(k- t - 1)/8}(3p^{-\lfloor(k-t)/8\rfloor/2} + (2\sqrt{p})Y_t)\le 3 + p^{1/3}Z_{t}\]
if $p\le C^{-1}\le1/2^{24}$. Iterating this bound we have
\begin{align*}
\mb{E}Z_k&\le3+3\cdot p^{1/3}+3\cdot(p^{1/3})^2+\cdots+p^{(k-1)/3}\mb{E}Z_1\\
&\le 4 + p^{(k-1)/3}p^{(k-1)/8}p^{-k/8}\le 5
\end{align*}
as $p\le C^{-1}$ and $k\ge 5$. By Markov's inequality, for $p\le C^{-1}$ we have
\[\mb{P}(X_k\ge t) = \mb{P}(Z_k\ge p^{-t/2}-1)\le 6p^{t/2},\]
as desired.
\end{proof}

\section{A deterministic singular value update formula}\label{sec:update-formula}
The goal of this section is to prove the following relationship between a window of singular values of a matrix $M$ and a similar window of the matrix $M$ with a row/column added, as we discussed at \eqref{eq:lin-alg-prod}. This inequality is central to our analysis of the random walk on the window height, which we defined in Section~\ref{sec:sketch-random-walk}. We note that this section is completely self-contained and concerns only deterministic properties of matrices. 

\begin{prop}\label{prop:walk-row-modified}
Let $n\le m$, let $M$ be an $n\times m$ matrix, and let $M'$ be an $(n+1)\times m$ matrix obtained by adding the row $X$ to $M$. For $1\le k - 1 \le \ell < m$, we have 
\begin{equation}\label{eq:prop:walk-row-modified}-\frac{1}{n}\sum_{i=k}^{\ell+1} \log \s_i(M') \leq   - \frac{1}{n}\sum_{i=k-1}^{\ell} \log \s_i(M) + f(M,X)\end{equation} where 
\[ f(M,X) = \frac{1}{2n}\log\Bigg( \frac{\snorm{X}_2^2 + \sigma_{k-1}(M)^2}{\|P X^\dagger \|_2^2}\Bigg). \]
Here $P$ is the orthogonal projection onto the span of the $m-\ell$ smallest right-singular vectors of $M$.
\end{prop}

%We remark that, unlike the corresponding lemma in our companion paper \cite[Lemma~7.2]{SSS23} , the number of singular values in the windows in each side of \eqref{eq:prop:walk-row-modified} is the same. We make progress as we move our chunk of singular values ``closer'' in index to the minimal singular values for $M'$. 

We start towards proving Proposition~\ref{prop:walk-row-modified}, by noting the following two basic facts. First the Cauchy interlacing theorem for singular values. 

\begin{comment}
singular value update formula, Proposition~\ref{prop:walk-row-modified}, which we mentioned in \eqref{eq:lin-alg-prod} of our sketch, and is central to our analysis of the random walk we defined in Section~\ref{sec:sketch-random-walk}. This section is self-contained and we do not adopt any of the global variables that are present in the other parts of the paper. 
\end{comment}

\begin{fact}\label{fact:cauchy-interlacing}
Let $M$ be an $n\times m$ matrix and let $M'$ be $M$ with a row added. Then
\[ \s_m(M) \le \s_m(M') \le \s_{m-1}(M) \le \s_{m-1}(M') \le \cdots \le \s_1(M) \le \s_1(M'). \]
\end{fact}

Next we require the following basic fact from linear algebra. 
\begin{fact}\label{fact:lin-alg}
Let $M$ be an $n\times m$ matrix with $n\le m$. We have $\sigma_i(M) = \sigma_i(M^{\dagger})$ for $i\le n$ and $\sigma_i(M) = 0$ for $n+1\le i\le m$.
\end{fact}

 Our proof of Proposition~\ref{eq:prop:walk-row-modified} is based on a careful analysis of the following nice relationship between the singular values of $M'$ given the singular vectors of $M$.

\begin{fact}\label{fact:cauchy-interlacing2}
Let $M$ be an $n\times m$ matrix and let $M'$ be an $(n+1)\times m$ matrix obtained by adding the row $X$ to $M$. Let $v_i $ denote the $i$th right singular vector of $M$. Then the roots of the equation 
\begin{equation}\label{eq:walk-key}
\prod_{i=1}^{m}\big(\sigma_i(M)^2-x\big)+\sum_{i=1}^{m}\big|\sang{v_i,X^\dagger}\big|^2 \cdot \prod_{j\neq i}\big(\sigma_i(M)^2-x\big) = 0
\end{equation}
are precisely $\sigma_i(M')^2$ for $i\in[m]$.
\end{fact}
\begin{proof}
Note that $(M')^{\dagger}M' = M^{\dagger}M + X^{\dagger}X$. Also, we have the rank one orthogonal decomposition
\[M^\dagger M=\sum_{i=1}^{m}\sigma_i(M)^2 v_i(M)v_i(M)^\dagger\]
and hence for all $x\notin\{\sigma_i(M)^2\colon i\in[m]\}$ we have
\[(M^\dagger M-xI_m)^{-1}=\sum_{i=1}^m(\sigma_i(M)^2-x)^{-1}v_i(M)v_i(M)^\dagger.\]
Next, the matrix determinant lemma gives $\det(A+uv^{\dagger})=(1+v^{\dagger}A^{-1}u)\det A$. Therefore by direct computation, for generic $x$ we have
\begin{align*}
\det((M')^{\dagger}M'-xI) &=(1 + X(M^{\dagger}M-xI_m)^{-1}X^\dagger)\det(M^{\dagger}M - xI_m) \\
&= \prod_{i=1}^{m}(\sigma_i^2(M)-x)+\sum_{i=1}^{m}|\sang{v_i,X^\dagger}|^2\prod_{j\neq i}(\sigma_i^2(M)-x).
\end{align*}
The resulting equality is in fact valid for all $x\in\mb{C}$ by the identity theorem. Finally, the roots of $\det((M')^{\dagger}M' - xI) = 0$ are the squares of the singular values of $M'$, so the desired result follows. 
\end{proof}

\vspace{1mm}

We are now in a position to prove Proposition~\ref{prop:walk-row-modified}.

\begin{proof}[Proof of Proposition~\ref{prop:walk-row-modified}]
Note that the statement is vacuous for $\ell\ge n+1$ as the right-hand side is zero; thus it suffices to consider $\ell\le n$. We fix unit right-singular vectors $v_1(M), \ldots, v_m(M)$ such that $\snorm{Mv_i(M)}_2 = \sigma_i(M)$. We may write
\[M=\sum_{i=1}^n\sigma_i(M)u_i(M)^\dagger v_i(M)\]
where $u_i(M)^\dagger$ is the $i$th left-singular vector of $M$.

Via a continuity argument, it suffices to assume that $\sigma_i(M)$ are distinct for $i\le n$ and $\sang{v_i(M),X^\dagger}\neq 0$ for all $i$. In particular, let $M^{\eps} = M + \sum_{i=1}^{n}\eps Z_i u_i(M)^{\dagger}v_i(M)$ where $Z_i$ are uniform in $[0,1]$ and let $X^{\eps} = X + \eps Z'$ where $Z'$ is a standard $m$-dimensional Gaussian. For any sufficiently small fixed $\eps > 0$, with probability $1$, $X^{\eps}$ and $M^{\eps}$ satisfy $\sang{v_i(M),(X^{\eps})^\dagger}\neq 0$, $M^{\eps}$ has the same left- and right-singular vectors as $M$ (up to reordering the ones which were for the same singular value), and the singular values of $M^{\eps}$ are distinct. Taking $\eps\to 0^{+}$ gives the desired result. 

Now define 
\[F(x) = 1+\sum_{i=1}^{m}\frac{|\sang{v_i(M),X^\dagger}|^2}{\sigma_i(M)^2-x}.\]
The roots of $F(x)$ are $\sigma_i(M')^2$ for $i\le\min(n+1,m)$ by Fact~\ref{fact:cauchy-interlacing2} and since the $\sigma_i(M)$ are distinct. $F(x)$ is increasing in $(\sigma_i(M)^2,\sigma_{i-1}(M)^2)$ (where we write $\sigma_{0}(M) = +\infty$) for $1\le i\le n$. Finally, for $1\le i\le n$ we have
\[\lim_{x\to \sigma_i(M)^{+}}F(x) = -\infty \text{ and }\lim_{x\to \sigma_i(M)^{-}}F(x) = +\infty.\]

Let $\eta = \sum_{i=\ell + 1}^{m}|\sang{v_i(M),X^\dagger}|^2$ and note that $\eta=\snorm{PX^\dagger}_2^2$. For $x\in(\sigma_{\ell+1}(M)^2,\sigma_{k-1}(M)^2)$, we have 
\begin{align*}
F(x) &\le 1+\frac{\sum_{i=1}^{k-1}|\sang{v_i(M),X^\dagger}|^2}{\sigma_{k-1}(M)^2-x}+\sum_{i=k}^m\frac{|\sang{v_i(M),X^\dagger}|^2}{\sigma_i(M)^2-x}\\
&\le 1+\frac{\sum_{i=1}^{k-1}|\sang{v_i(M),X^\dagger}|^2}{\sigma_{k-1}(M)^2-x}+\sum_{i=k}^{\ell}\frac{|\sang{v_i(M),X^\dagger}|^2}{\sigma_i(M)^2-x} + \sum_{i=\ell + 1}^{m}\frac{|\sang{v_i(M),X^\dagger}|^2}{-x}\\
& = 1+\frac{\sum_{i=1}^{k-1}|\sang{v_i(M),X^\dagger}|^2}{\sigma_{k-1}(M)^2-x}+\sum_{i=k}^\ell\frac{|\sang{v_i(M),X^\dagger}|^2}{\sigma_i^2(M)-x}+\frac{\eta}{-x} =: G(x). 
\end{align*}
Via direct inspection $G(x)$ has $\ell - k + 3$ zeros $\tau_0>\tau_1>\cdots>\tau_{\ell-k+2}$. Using $F(x)\le G(x)$ we have that $0\le \tau_i\le \sigma_{i+k-1}(M')$ for $1\le i\le \ell - k + 2$ and $\tau_0>\sigma_{k-1}(M)$. By Vieta's formula we have
\[\prod_{i=0}^{\ell-k+2}\tau_i =\eta\prod_{i=k-1}^\ell\sigma_i(M)^2.\] Finally, note that $G$ is increasing and starts near $-\infty$ for $x$ close to $\sigma_{k-1}^2$ from above and
\[G(\sigma_{k-1}(M)^2+\snorm{X}_2^2)\ge 1-\frac{1}{\snorm{X}_2^2}\bigg(\sum_{i=1}^{k-1}|\sang{v_i(M),X^\dagger}|^2+\sum_{i=k}^\ell|\sang{v_i(M),X^\dagger}|^2+\eta\bigg)=0.\]
Thus the additional root satisfies the bound $\tau_0\le\sigma_{k-1}^2+\snorm{X}_2^2$. Finally, we deduce
\[\prod_{i=k}^{\ell+1}\sigma_i(M')^{2}\ge\prod_{i=1}^{\ell-k+2}\tau_i =\eta\tau_0\prod_{i=k-1}^\ell\sigma_i^2\ge\eta(\snorm{X}_2^2+\sigma_{k-1}^2)^{-1}\prod_{i=k-1}^\ell\sigma_i^2.\]
Taking logarithms completes the proof. 
\end{proof}

\vspace{1em}

We remark that the auxiliary function $G(x)$, defined in the proof, has a natural ``physical'' interpretation. It corresponds to the situation when all of the $L^2$-mass of $X^{\dagger}$ that is on 
\[\Span\{ v_{k-1}(M),\ldots,v_{m}(M) \} \] is concentrated on the single direction $v_{k-1}(M)$; and all the $L^2$-mass of $X^{\dagger}$ on 
\[ \Span\{ v_{\ell+1}(M),\ldots,v_m(M)\}\] is concentrated at $0$. The key observation in the above is that as we deform the equation $F(x) = 0 $ into $G(x) = 0 $, our roots of interest move monotonically downward.

A physical picture, which may be useful for some readers, is to imagine repulsive charges placed at the values $\sigma_i(M)$ with charge $|\sang{v_i(M),X^{\dagger}}|^2$. We can now interpret the roots of $F(x)$ as the places on the real line where these charges precisely cancel out. Thus, in this interpretation, when we deform $F(x)$ into $G(x)$ we push and merge all the charges $\geq \sigma_{k-1}(M)$ into a single charge at $\sigma_{k-1}(M)$ and, similarly push and merge all of the charges below $\sigma_{\ell+1}(M)$ into a single charge at $0$. Thus, we intuitively see that these locations of perfect cancellation (i.e. the roots we are interested in) move monotonically downward.

\section{Convergence of the singular value measures}\label{sec:singular-convergence}
In this section we carry out the (standard) steps mentioned at \eqref{eq:def-nu} and \eqref{eq:sketch-trace-moments} in the sketch and prove the following lemma on the convergence of the singular value measures of $B_m - zI_m$ and $B_n-zI_n$. We prove this by a completely standard application of the method of moments and so we will allow outselves to be somewhat light on on details in this section.

\begin{lemma}\label{lem:converge-singular}
Let $z\in \mb{C}$ and $\eps>0$ be small but fixed. Let $B_m$ and $B$ be as in Section~\ref{sec:setup}. Let $\nu_{z,n}', \nu_{z,n}$ be the empirical singular value measures of $B_m-zI$ and $B-zI$ respectively. There exist deterministic measures $\nu_{|z|,\eps}'$ and $\nu_{|z|}$ such that $\nu_{z,n}'\rightsquigarrow \nu_{|z|,\eps}'$ and $\nu_{z,n}\rightsquigarrow \nu_{|z|}$. 
\end{lemma}

We prove this by proving the following result on the trace moments of our random matrices. For this we introduce some notation. For a matrix $M$, define 
\[ M^{(1)}=M  \qquad \text{ and } \qquad M^{(-1)}=M^{\dagger}.\]
We also let $B_m$ and $B = B_n$ be as in Section~\ref{sec:setup} throughout.

\begin{lemma}\label{lem:moment-convergence}
Let $d>0$, $\eps >0$ be fixed and for $r\le \log \log n$, let $\vec{s} = (s_1,\ldots,s_{r})\in \{-1,1\}^{r}$. Then there exists a constant $M(\vec{s},\eps)$ such that 
\[
\mb{P}\bigg(\bigg|\frac{1}{m}\on{Tr}\prod_{i=1}^r B_m^{(s_i)}-M(\vec{s},\eps)\bigg|\ge n^{-1/2+o(1)}\bigg) \leq n^{-\omega(1)},\]
and a constant $M'(\vec{s} \hspace{0.5mm} )$ such that 
\[ \mb{P}\bigg(\bigg|\frac{1}{n}\on{Tr}\prod_{i=1}^r B^{(s_i)}-M'(\vec{s} \hspace{0.5mm} )\bigg|\ge n^{-1/2+o(1)}\bigg) \leq n^{-\omega(1)}.\]
As a result, we have for any fixed $z \in \C$, 
\[ \bigg|\mb{E}\bigg[\frac{1}{m}\sum_{i=1}^{m}\big(\s_i(B_m - zI)^{2r} -\s_i(B_m - |z|I)^{2r}\big) \bigg]\bigg|\leq n^{-1/2+o(1)},\]
and 
\[ \bigg|\mb{E}\bigg[\frac{1}{n}\sum_{i=1}^{n}\big(\s_i(B - zI)^{2r} -\s_i(B - |z|I)^{2r}\big)\bigg]\bigg|\leq n^{-1/2+o(1)},\]
where the $o(1)$ terms may depend on $z$.
\end{lemma}

\begin{proof}
We prove the claims only for $B_m$; the analogous claim for $B$ is rather simpler noting that $B$ can be coupled with a $\on{Ber}(d/n)$ matrix by changing at most $(\log n)^3$ many entries with probability $1-n^{-\omega(1)}$.

Note that $\on{Tr}\big(\prod_{i=1}^\ell B_m^{(s_i)}\big)$ can be interpreted as the number of walks of length $\ell$ such that the $i$th step is taken on the digraph associated to $B$ if $s_i = 1$ and is taken on the digraph associated to $B^{\dagger}$ if $s_i = -1$. We prove concentration conditional on the degree sequence of $B_m$ in the configuration model. To show the different possible outcomes of the degree sequence have close means, a standard expectation computation in the configuration model and the control from Lemma~\ref{lem:degree-convergence} suffices.

To prove concentration in the configuration model, the pairing used to define the associated random graph can be viewed as a uniformly random permutation between the left and right stubs. Furthermore, changing two pairings in this random permutation creates or destroys at most $(\log n)^{\log\log n} = n^{o(1)}$ walks counted by the moment under the assumption that degree sequence of the digraph associated to $B_m$ has maximum in- and out-degree bounded by $\log n$ (which we may assume, occurring with probability $1-n^{-\omega(1)}$). The desired concentration then follows from Lemma~\ref{lem:injection-concentration}.

For the second part of the lemma for $B_m$, note that 
\[\frac{1}{m}\sum_{i=1}^{m}\s_i(B_m - zI)^{2r} = \frac{1}{m} \tr(((B_m-zI)^\dagger(B_m-zI))^{r}).\]
One may view this as walks of length $2r$ such that the odd steps correspond to $B_m$ with potential self-loops of weight $-z$ and the even steps correspond to backwards transversal on $B_m$ with potential self-loops of weight $-\ol{z}$. We break up the contributions to the expectation of $\frac{1}{m} \tr(((B_m-zI)^\dagger(B_m-zI))^{r})$ based on the precise graph-theoretic structure (within $B_m$), decorating special self-loops with either $-z$ or $-\ol{z}$. The underlying directed graph is connected.

In order to contribute to the leading term of the expectation, this digraph must have $v$ vertices with $v-1$ total directed edges (not counting special self-loops) and thus is a tree (when viewed as an undirected graph). Additionally, all self-loops must be decorated by $-z$ or $-\ol{z}$. The walk corresponds to a back-and-forth transversal of this tree which goes through each directed edge at least once in each direction, along with transversal of the decorated edges, and it returns to its original point.

Fix such a transversal and analyze the vertex $w$ furthest from the root. We see that the walk traverses an edge to reach $w$, then self-loops through $w$ an even number of times (which contributes $|z|^{2k}$ since the weights must alternate between $-z,-\ol{z}$), and then transverses back along this edge. Factoring out this contribution and proceeding downward inductively, one can observe that all dominant terms have expectations which are a function of $|z|$ (not simply $z$).
\end{proof}

\vspace{1mm}

Next we require the following bound on the operator norm of a matrix in terms of the $\ell_1$ norms of its rows and columns, due to Schur \cite{Sch11}. Indeed, recall that 
for a matrix $M$ the $1\to 1$ norm of $M$, denoted $\snorm{M}_{1\to 1}^{1/2}$, is the maximum $\ell_1$ norm of a column. 
\begin{lemma}\label{lem:schur-bound}
For any matrix $M$,  
\[\snorm{M}_{\mr{op}}\le \snorm{M}_{1\to 1}^{1/2}\snorm{M^{\dagger}}_{1\to 1}^{1/2}.\]
\end{lemma}

We will also require a variant of Weyl's inequality which follows immediately from the Courant–Fischer theorem applied to singular values. 
\begin{fact}\label{fact:inequality}
Fix matrices $A,B\in \mb{C}^{n\times n}$. We have for all $1\le i\le n$ that 
\[|\sigma_i(A) - \sigma_i(B)|\le \snorm{A-B}_{\mr{op}}.\]
\end{fact}

The last ingredient we need in the proof of Lemma~\ref{lem:converge-singular} is control on the growth rate of the moments of the singular values. This will allow us to apply Carleman's condition to conclude that there exists a probability distribution with a given sequence of moments. 
\begin{lemma}\label{lem:singular-val-upper}
Let $d>1$, $\eps >0$ be fixed and let $k \leq \log \log n$. With probability $1-n^{-\omega(1)}$ we have
\[\sum_{i=1}^{m}\sigma_i(B_m-zI)^k\le \sum_{i=1}^{n}\sigma_i(B-zI)^{k}\le n\bigg(\frac{Cdk}{\log(k+1)} + 4|z|\bigg)^{k},\]
where $C>0$ is an absolute constant. 
\end{lemma}
\begin{proof}
By Fact~\ref{fact:cauchy-interlacing} and Fact~\ref{fact:inequality} we have that $\sigma_i(B_m-zI)\le \sigma_i(B-zI)\le \sigma_i(B) + |z|$. Thus it suffices to consider $B$ and consider the case when $z = 0$. Furthermore one can couple $A$ and $B$ so that $A$ and $B$ differ in at most $(\log n)^{3}$ entries with probability $1-n^{-\omega(1)}$; recall that $A$ is a matrix where the entries are independent $\on{Ber}(d/n)$. 

Note that 
\[\mb{P}(\on{Ber}(d/n)\ge t)\le \binom{n}{t}(d/n)^{t}\le \bigg(\frac{ed}{t}\bigg)^{t}.\]
Therefore applying Chernoff's inequality to the rows and columns of $A$ separately implies that there are at most $(Cd/t)^tn$ vertices with degree larger than $t$ for $t\in [1, \sqrt{\log n}]$ and the maximum degree is at most $\log n$ with probability $1-n^{-\omega(1)}$. Therefore, using Lemma~\ref{lem:schur-bound} we have
\[\sum_{i=1}^{n}\sigma_i(A)^{k}\le n +  \int_{x\ge 1}kx^{k-1}\big(x/Cd\big)^{-x/C}~dx + n(\log n)^{k}\cdot \exp\hspace{-0.5mm}\big(\hspace{-0.5mm}-\hspace{-0.5mm}\Omega\big(\sqrt{\log n}\big) \big),\] which is
\[ \leq   n\bigg(\frac{Cdk}{\log(k+1)}\bigg)^{k} , \]
for some $C>0$. As $A$ and $B$ differ in at most $(\log n)^{3}$ entries with appropriate probability, an analogous bound holds for $B$. 
\end{proof}

\vspace{1mm}

We now turn to prove the main result of this section. 

\begin{proof}[Proof of Lemma~\ref{lem:converge-singular}]
We will restrict attention to convergence of the singular values of $B_m-zI_m$; the proof is identical for $B_n - zI_n$. Note that 
\[\sum_{i=1}^m\sigma_i(B_m-zI_m)^{2k} = \on{Tr}\big(\big((B_m-zI_m)^\dagger(B_m-zI_m)\big)^{k}\big).\]

We use the second part of Lemma~\ref{lem:moment-convergence} to replace $z$ with $|z|$ and then use expansion and the first part to deduce that there are constants $C(|z|,k)$ such that with probability $1-n^{-\omega(1)}$ we have 
\[\bigg|\sum_{i=1}^m\sigma_i(B_m-zI_m)^{2k} - mC(|z|,k)\bigg|\lesssim_{K,|z|} n^{-1/3}\]
for all $k\le \log\log n$. Furthermore by Lemma~\ref{lem:singular-val-upper} we have $C(|z|,k)\le\big(\frac{O(dk)}{\log(k+1)}+4|z|\big)^{2k}$. 

Let $\tilde{\nu}_{z,m}$ denote the uniform measure on the set $\bigcup_{1\le i\le m}\{\sigma_i(B_m-zI_m),-\sigma_i(B_m-zI_m)\}$. It suffices to prove convergence in distribution of $\tilde{\nu}_{z,m}$ to a measure $\tilde{\nu}_{|z|,\eps}$ by the continuous mapping theorem as $|\cdot|$ is continuous everywhere.

By a standard argument (see e.g.~\cite[Section~3.3.5,~p.~140]{Dur19}), it suffices to prove that there exists a unique distribution with zero odd moments and even moments $C(|z|,k)$. Note that 
$\sum_{k\ge 1}C(|z|,k)^{-1/(2k)}\gtrsim_{d,|z|} \sum_{k\ge 1}(\log k)/k = \infty$ and therefore Carleman's condition applies (see e.g.~\cite[Theorem~3.3.25,~Remark]{Dur19}). Also, evidently the distribution one converges to depends only on $|z|$.
\end{proof}

\section{Initializing the process }\label{sec:non-atomic-singular}
In this section we prove that we have control on the \emph{first} step in our random process for almost all $z \in \C$.  This amounts to showing that the singular value measure $B_m - zI_{m}$ does not have an atom at zero for almost all $z \in \C$. As we described in Section~\ref{sec:sketch-initializing}, the method we apply here is quite simple and completely different from the rest of the paper. Interestingly, it uses very little about our random matrix apart from the existence of the limiting singular value measures (Lemma~\ref{lem:converge-singular}) and the fact that $\|B_m\|_{\on{HS}}$ is controlled. Throughout this section we let $B_m$ be as we defined in Section~\ref{sec:setup} and 
define $U'_n$ be the logarithmic potential of $B_m$. We also recall the definition of the initial window in our process from \eqref{eq:window-def}
\[ W_{\g,\delta, m}(z) := - \frac{1}{n}\sum_{j= \g m}^{(\g+\delta)n} \log\big( \s_{m-j}(B_m-zI_m) \big),\]
where we have abused notation slightly and parametrized the window based on $\g,\delta,m$.

\begin{lemma}\label{lem:initial-step}
For all $d, \eps >0$ and $\g>0$, there exists $\delta : \C \rightarrow (0,1)$ such that 
\[ \lim_{m\rightarrow \infty } \mb{P}\big( W_{\g,\delta(z), m}(z) \leq \eps  \big) = 1,  \]
for all $z \in \C$ outside a set of measure $0$.
\end{lemma}
We remark that while we allow ourselves to exclude an arbitrary set of measure zero in the above lemma, we believe that the only value of $z$ that we are required to exclude is $z = 0$. 

We now turn to prove Lemma~\ref{lem:initial-step}. By Lemma~\ref{lem:converge-singular} we know that the limiting distribution of the singular values $\{\sigma_{j}(B_m-zI_m)\}$ exists and is $\nu'_z$. Thus we define the set of points in $\C$ for which $\nu'_z = \nu'_{d,z}$ is atomic at $0$
\begin{equation}\label{eq:def-Sd} S_{d} = \big\{ z \in \C : \nu'_z(\{0\}) > 0 \big\}. \end{equation}
The key step in proving Lemma~\ref{lem:initial-step} is to show that $S_d$ has measure $0$. For this, we first need the following simple deterministic result on the log potential. 

\begin{lemma}\label{lem:upperbounds-on-U(z)} Let $\rho>0$. Then (deterministically) we have
\begin{equation}\label{eq:lem-upperbounds-on-U(Z)} \frac{1}{2\pi}\int_0^{2\pi} U_n'(\rho e^{i\theta}) \, d\theta \leq \log \rho.\end{equation}
\end{lemma}
\begin{proof} We use the expression (from the definition of the log potential)
 \[ U'_n(z)  = -\frac{1}{m}\sum_{i=1}^m \log |z-\lambda_i| \]
 and thus our integral at \eqref{eq:lem-upperbounds-on-U(Z)} becomes 
 \[ \frac{1}{m}\sum_{i=1}^m \frac{1}{2\pi }\int_{0}^{2\pi} \log|\rho e^{i\theta}-\lambda_i|\, d\theta  = \frac{1}{m}\int_0^{\rho} \big| \{ i  : |\lambda_i| < s \big\}\big| s^{-1}\, ds \leq \log \rho, \]
 where the equality holds by Jensen's formula from complex analysis (or direct integration). 
\end{proof}

We now use that our matrices have controlled Hilbert-Schmidt norm. In particular, 
\[ \|B_m-zI_m\|_{\on{HS}}^2\leq C(d+|z|^2)m = O_{z,|d|}(m),\] for some $C>0$ with probability $1-o(1)$ as $m\rightarrow \infty$, by standard concentration estimates. Thus we define $\cB$ to be the event $\|B_m\|_{HS}^2\leq Cdm/2$, which then implies $\|B_m-zI_m\|_{\on{HS}}^2\leq C(d+|z|^2)m$.

\begin{lemma}\label{lem:log-am-gm} Let $z \in \C$. On the event $\cB$ we have 
\[ \frac{1}{m} \sum_{\sigma_i \geq 1} \log \sigma_i =  O_{|z|,d}(1),\]
where $\sigma_i = \sigma_{i}(B_m-zI_m)$. Thus $-U(z) = O_{|z|,d}(1)$ on $\cB$.
\end{lemma}
\begin{proof}
We first note that we always have (not just on $\cB$), 
\[ -U_m(z) \leq \frac{1}{2m}\sum_{\sigma_i \geq 1} \log \sigma_i^2 \leq \frac{1}{2}\log\bigg( \frac{1}{m}\sum_{i=1}^m \sigma_i^2\bigg) \leq \frac{1}{2}\log\big( m^{-1}\|B_m-zI_m\|_{\on{HS}}^2 \big). \]
We then finish by using the definition of $\cB$.
\end{proof}

\vspace{1em}

We now show that $S_d$ has measure zero.

\begin{lemma} We have $\mu(S_{d}) = 0 $.
\end{lemma}
\begin{proof} Assume that $\mu(S_{d}) >0$ and so there is a $\g >0$ so that the set 
\[ S_{d,\g} = \big\{ z \in \C : \nu'_z(\{0\}) > \g \big\}, \]
has positive measure. Now choose a radius $\rho>0$ so that 
the measure of 
\[ S_{\g,\rho} = S_{d,\g} \cap \{ z : |z| = \rho \} \]
is at least $\alpha > 0$ relative to the uniform measure on the circle $\{ z : |z| = \rho\}$.

We use this to lower bound the logarithmic potential $U_n'(z)$. Let $h_R(t) = \min\{ -\log t, R \}$ and note that 
\[ U_n'(z)  =  -\int \log t \, d\nu'_{n,z} \geq \int h_R(t)\, d\nu'_{n,z} = \int_0^1 h_R(t)\, d\nu'_{n,z}  + \int_1^{\infty}h_R(t)\, d\nu'_{n,z} . \, \]
Note the first term on the right hand side is $\geq 0$. Now since $h_R(t)$ is bounded on $[0,1]$, 
\[ \int_0^1 h_R(t)\, d\nu'_{n,z} \geq \int_0^1 h_R(t)\, d\nu'_{z} - 1,\]
with probability $1-o(1)$. Also note that on the event $\cB$, by Lemma~\ref{lem:log-am-gm},
\[ \int_1^{\infty}h_R(t)\, d\nu'_{n,z}  = \frac{1}{m} \sum_{\sigma_i(B_m-zI_m) \geq 1} \log \sigma_i(B_m-zI_m) = O_{|z|,d}(1).\]
So, when $z \in S_{d,\g}$, we have 
\[  U_n'(z) \geq \int h_R(t) \, d\nu'_{z} - O_{|z|,d}(1)  \geq R \g - O_{|z|,d}(1) ,\]
with probability $1-o(1)$ as $n\rightarrow \infty$. Here we are using that $\nu_z'(\{0\}) > \g$. Thus, 
\[ \EE\, U_n'(z)\1_{\cB} \geq R\g - O_{|z|,d}(1), \]
for $n$ sufficiently large. Since $\EE\, U_n'(z)\1_{\cB} + O_{|z|,d}(1)$ is point-wise bounded and non-negative we may apply Fatou's lemma to say
\[ \liminf_{n} \int_0^{2\pi} \EE\, U_n'(e^{i\theta})\1_{\cB}  \geq   \int_0^{2\pi} \liminf_{n} \EE\, U_n'(e^{i\theta})\1_{\cB} \geq (R\g)\alpha - O_{|z|,d}(1) . \]
However, here we see that we contradict Lemma~\ref{lem:upperbounds-on-U(z)}. Since for each fixed $n$ we have
\[ \int \EE\, U_n'(e^{i\theta}) \1_{\cB}\,  d\theta = \EE \1_{\cB} \int U_n'(e^{i\theta}) \, d\theta \leq 2\pi \log \rho . \] So if we chose $R \gg (\log \rho)/(\alpha \g)$, we obtain a contradiction. \end{proof}

\vspace{1em}

We now find ourselves in a position to derive Lemma~\ref{lem:initial-step}.

\begin{proof}[Proof of Lemma~\ref{lem:initial-step}] Let $z \not\in S_{d}$. Now let $\eta>0$ be such that $\nu'_z([0,\eta]) < \g$, which exists by continuity of measure. Thus, by Lemma~\ref{lem:converge-singular}, we have that $\nu'_{n,z}([0,\eta]) < \g $ with probability $1-o(1)$ as $n\rightarrow \infty$. This implies that 
\[\sigma_{(1-\g)m}(B_m-zI_m) \geq \eta,\] with probability $1-o(1)$. Thus if we choose $\delta = \eps(\log 1/\eta)^{-1}$, we have that 
\[ W_{\g, \delta, m}(z) = - \frac{1}{n}\sum_{j= \g m}^{(\g+\delta)n} \log\big( \s_{m-j}(B_m-zI_m) \big) \leq \delta \log 1/\eta \leq \eps, \]
with high probability, as desired. 
\end{proof}

\section{The random walk on windows of singular values}\label{sec:crucial}

In this section we put the ingredients of the proof together to prove the following key lemma on the random walk. Once we have proved this, it will be only a small step to the proof of our main theorem, Theorem~\ref{thm:main}. 

Throughout this section, we let $d>1$ be fixed and $B = B_n$ be the random $n\times n$ matrix defined in Section~\ref{sec:setup}. We also assume that $\eps>0$ is chosen to be sufficiently sufficiently small compared to $d>1$.

\begin{lemma}\label{lem:crucial}
Fix $d > 1$ and let $\eps>0$ be sufficiently small. Then, for almost all $z \in \C$, there exists $\delta = \delta(\eps,z) > 0$ such that
\[\lim_{n\to\infty}\mb{P}\bigg(-\frac{1}{n}\sum_{j=0}^{\delta n}\log \sigma_{n-j}(B-zI)> C\eps\bigg) = 0,\]
where $C = C(d,z)>0$ is a constant depending only on $d$ and $z$.
\end{lemma}

In the remainder of this section we prove this lemma. This essentially follows from piecing together the ingredients which have developed in previous sections.

\begin{proof}[Proof of Lemma~\ref{lem:crucial}]
We first appeal to Lemma~\ref{lem:initial-step} to initialize our process. Apply this lemma with $\eps^4/2$ and $\gamma = \eps^{4}$. Thus for $z$ outside a set of measure $0$, we obtain a $\delta = \delta(\eps,z)$ for which 
\[ - \frac{1}{n}\sum_{j= \g m}^{(\g+\delta)n} \log\big( \s_{m-j}(B_m-zI_m) \big) \leq \eps^4/2 \]
with high probability.

\noindent\textbf{Step 1: Definition of the random walk.}

\begin{comment}

 Consider $\Xi_j = 2^{j}$. By Lemma~\ref{lem:abstract} with $\gamma = \eps^{4}$, there exists $\tau_j>0$ such that for all but a measure $2^{-j}$ set of values $z$, we have that 
\[\mb{P}(\sigma_{\lceil(1-\eps^4)m\rceil}(B_m-zI_m)\le \tau_j) = o(1).\]
Considering all $j \in \Z_{>0}$, it follows that for all but a measure $0$ set of $z$ there is $\tau(z)>0$ such that 
\[\mb{P}(\sigma_{\lceil(1-\eps^4)m\rceil}(B_m-zI_m)\le \tau(z)) = o(1).\] 
Define $\delta m = \lfloor\eps^4\min((\log(1/\tau(z)))^{-1},1/25) m \rfloor$ and note that $\delta\le\eps^4/25$ which implies that $\delta \to 0$ as $\eps\to 0$. Furthermore with probability $1-o(1)$, we have 
\begin{align*}
\prod_{j = \lceil(1-\eps^4)m\rceil-\delta m}^{\lceil(1-\eps^4)m\rceil}\sigma_j(B_m-zI_m)&\ge (\tau(z))^{\delta m}\ge\min\bigg(1,(\tau(z))^{\eps^4 m/\log(1/\tau(z))}\bigg)= \exp(-\eps^4 m).
\end{align*}

\end{comment}

We now define the random walk. We set $X_m = m - \lceil(1-\eps^4)m\rceil\le \eps^4 m$ and therefore
\begin{align*}
\prod_{j = X_m}^{X_m + \delta m}\sigma_{m-j}(B_m-zI_m)\ge \exp(-\eps^4 m).
\end{align*}
Recall the definition of $\eps_r$ from Section~\ref{sec:anticoncentration}. We iteratively define the random variable $X_{t+1}$ when $m\le t\le n-\ell-1$ by:
\begin{itemize}
    \item If $X_t\le \lfloor n - \ell - t\rfloor/16$, $\snorm{B_{t+1}}_{HS}^2\ge 2dn$, or $t\ge n-\ell - (\log n)^{7/4}$ define $X_{t+1} = X_t + 1$. 
    \item Else if $\sigma_{t-X_t/2}(B_t-zI_t)\ge \eps_{X_t}$, 
    define $X_{t+1} = X_t - 1$.
    \item Else if
    \begin{align*}
    \text{either }&\min(\deg_{B_{n-\ell}}^+(v_{t+1},T_1),\deg_{B_{n-\ell}}^-(v_{t+1},T_1))\le \sqrt{\log(1/\eps)}\\
    \text{or }&\max(\deg_{B_{n-\ell}}^+(v_{t+1}),\deg_{B_{n-\ell}}^-(v_{t+1}))\ge (\log(1/\eps))^2,
    \end{align*}
    define $X_{t+1} = X_t + 1$. 
    \item Else if 
    \begin{align*}
    \|P_{X_t,B_{t}^{\dagger} -\ol{z}I_{t}}((B_{t+1}^{\ast}-zI_{t\times (t+1)})e_{t+1})\|_2 &\ge \eps_{X_t},\\
    \|P_{X_t,B_{t+1}^{\ast} - zI_{t\times (t+1)}}((B_{t+1}^{\dagger}-\ol{z}I_{t+1})e_{t+1})\|_2 &\ge \eps_{X_t}
    \end{align*}
    both hold, define $X_{t+1} = X_t - 1$.
    \item Else define $X_{t+1} = X_t + 1$.
\end{itemize}
When $n-\ell\le t\le n-1$ we define $X_{t+1}$ by:
\begin{itemize}
    \item If $X_t > 1$ and $\sigma_{t}(B_t)\ge \eps_{1}$, define $X_{t+1} = X_t -1$.
    \item Else if $X_t = 1$ and $\sigma_{t}(B_t)\ge \eps_{1}$, or $X_t = 0$ and
     \[\|P_{1,B_{t+1}^{\ast} - zI_{t\times (t+1)}}((B_{t+1}^{\dagger}-\ol{z}I_{t+1})e_{t+1})\|_2 \ge \eps_{1},\]
    define $X_{t+1} = 0$.
    \item Else if $X_t\ge 1$, $\sigma_{t}(B_t)\le \eps_{1}$, and 
    \begin{align*}
    \|P_{1,B_{t}^{\dagger} - \ol{z}I_{t}}((B_{t+1}^{\ast}-zI_{t\times (t+1)})e_{t+1})\|_2 &\ge \eps_{1},\\
     \|P_{1,B_{t+1}^{\ast} - zI_{t\times (t+1)}}((B_{t+1}^{\dagger}-\ol{z}I_{t+1})e_{t+1})\|_2 &\ge \eps_{1}
    \end{align*}
    both hold, define $X_{t+1} = X_t - 1$. 
    \item Else define $X_{t+1} = X_t + 1$.
\end{itemize}

\noindent\textbf{Step 2: Reducing to proving that $X_n = 0$ whp.}
Let 
\[\tau_1 = 8(d(\log(2/\eps))^4+|z|^2)\text{ and } \tau_2 = 8n^4(1+|z|^2).\]
We claim that
\begin{align}\label{eq:product-iterate}
\prod_{j = X_{t+1}}^{X_{t+1}+\delta m}\sigma_{t+1-j}(B_{t+1}-zI_{t+1})&\ge \tau_1^{-2}\eps_{X_t}^2\prod_{j = X_{t}}^{X_t+\delta m}\sigma_{t-j}(B_{t}-zI_t)
\end{align}
for $m\le t\le n-\ell - 1$ and
\begin{align}\label{eq:product-iterate-2}
\prod_{j = X_{t+1}}^{X_{t+1}+\delta m}\sigma_{t+1-j}(B_{t+1}-zI_{t+1})&\ge \tau_2^{-2}\eps_1^{2}\prod_{j = X_{t}}^{X_t+\delta m}\sigma_{t-j}(B_{t}-zI_t)
\end{align}
for $n-\ell\le t\le n-1$.

Given this claim, if $X_n=0$ then iterating shows that  
\begin{align*}
\prod_{j = 0}^{\delta m}\sigma_{n-j}(B_{n}-zI_n)&\ge \tau_1^{-2(n - m)}\tau_2^{-2\ell}\prod_{t=m}^{n-\ell}\eps_{X_t}^2 \cdot \eps_1^{2\ell}\cdot\prod_{j = X_{m}}^{X_m+\delta m}\sigma_{m-j}(B_{m}-zI_m)\\
&\ge \exp(-C'' \eps n)\cdot\prod_{j = X_{m}}^{X_m+\delta m}\sigma_{m-j}(B_{m}-zI_m)\ge \exp(-2C'' \eps n)\
\end{align*}
for a constant $C'' = C''(d,z)$. The only nontrivial term to estimate is $\prod_{t=m}^{n-\ell}\eps_{X_t}^2$. The desired estimate follows noting that the definition of the walk enforces $X_t\ge \lfloor n - \ell - t\rfloor/16 - 1$ for $m\le t\le n-\ell - 1$, and using $n-m \le 2\eps^{3}n$. We now show why \eqref{eq:product-iterate} and \eqref{eq:product-iterate-2} hold, which will allow us to focus on proving $X_n=0$ whp in the remainder of the proof.

When $X_{t+1} = X_t + 1$, we have that \eqref{eq:product-iterate} and \eqref{eq:product-iterate-2} hold trivially by Facts~\ref{fact:cauchy-interlacing} and~\ref{fact:lin-alg}. The second item defining $X_t$ in the first epoch and first item defining $X_t$ in the second epoch are similarly handled by Facts~\ref{fact:cauchy-interlacing} and~\ref{fact:lin-alg} and the lower bounds on singular values that we are given in these cases.

We now explain how to deduce the claim if the fourth item of the first epoch is used to define $X_{t+1}$; the remaining deductions (corresponding to the second and third items in the second epoch) are essentially identical and are omitted. Note that the norms of the additional row and column added are bounded by $(2\log(1/\eps)^4 + 2|z|^2)^{1/2}$ since the in- and out-degrees of $v_{t+1}$ are bounded by $(\log(1/\eps))^2$. Furthermore note that 
\[\sigma_{n/2}(B_t-zI_t)^2\le \frac{\snorm{B_t-zI_t}_{\mr{HS}}^2}{n/2}\le \frac{2\snorm{B_t}_{\mr{HS}}^2 + 2|z|^2n}{n/2}\le \tau_1\]
since $X_t + \delta m \le n/3$ almost surely and $\snorm{B_t}_{\mr{HS}}^2\le 2dn$. 

Therefore applying Proposition~\ref{prop:walk-row-modified} we have 
\begin{align*}
\prod_{j = X_t}^{X_t + \delta m}\sigma_{t+1 - j}(B_{t+1}^{\ast}-zI_{t\times (t+1)})&\ge \tau_1^{-1}\eps_{X_t}\prod_{j = X_t}^{X_t + \delta m}\sigma_{t - j}(B_t-zI_t),\\
\prod_{j = X_t+1}^{X_t + 1 + \delta m}\sigma_{t+1 - j}(B_{t+1}-zI_{t+1})&\ge \tau_1^{-1}\eps_{X_t}\prod_{j = X_t}^{X_t + \delta m}\sigma_{t+1 - j}(B_{t+1}^{\ast}-zI_{t\times (t+1)}).
\end{align*}
Note that in the first item we add a column to $B_t$ to get $B_{t+1}^{\ast}$; as we are considering right-singular values one needs Fact~\ref{fact:lin-alg} to relate the left- and right-singular values (and this explains the conjugate transpose and $\ol{z}$ in the first part of the fourth item of the first epoch). Multiplying these inequalities we derive exactly \eqref{eq:product-iterate} in this case.

\noindent\textbf{Step 3: Creating quasirandomness events and verifying they hold whp.} We now define a sequence of quasirandomness events; these will ultimately be required in order to import the results of Section~\ref{sec:anticoncentration}. For $m\le t\le n-\ell-1$, we define 
\[\mc{G}_{t+1} = \{B_t\in \mc{U}((\log n)^{3/2})\} \cap \{B_t^{\dagger}\in \mc{U}((\log n)^{3/2})\} \cap \{B_t\in \mc{D}\}\cap \{B_t^{\dagger}\in \mc{D}\}.\]
We define a vertex $v\in H$ to be \emph{degree-bad} if 
\begin{align*}
\min(\deg_{B_{n-\ell}}^+(v,T_1),\deg_{B_{n-\ell}}^-(v,T_1))&\le \sqrt{\log(1/\eps)}\text{ or }\\
\max(\deg_{B_{n-\ell}}^+(v),\deg_{B_{n-\ell}}^-(v))&\ge(\log(1/\eps))^2.
\end{align*}
Furthermore for $m\le t\le n-\ell-1$, let $\mc{H}_{t+1}$ be the event that $v_{t+1}$ is degree-bad and let
\[\mc{J}_{t+1} = \bigg\{\frac{\sum_{v\in S_{n-\ell}\setminus S_{t}}\mbm{1}[v\text{ is degree-bad}]}{(n-\ell)-t}\ge \eps\bigg\}.\]
For $n-\ell\le t\le n-1$, define 
\[\mc{G}_{t+1} = \{B_t^{\dagger}\in \mc{U}((\log\log n)^2)\} \cap \{B_t^{\dagger}\in \mc{D}\}\]
and for $n-\ell\le t\le n-1$, define
\[\mc{G}_{t+1}' = \{B_{t+1}^{\ast}\in \mc{U}((\log\log n)^2)\} \cap \{B_{t+1}^{\ast}\in \mc{D}\}\cap \{B_{t+1}^{\ast}\in\mc{U}^{\ast}\}.\]

We now stitch together various claims regarding these quasi-randomness events. 
\begin{claim}\label{clm:epoch-1-degree}
For all $t\le n-\ell - (\log n)^{7/4}$, we have 
\[\mb{P}(\mc{J}_{t+1}) \le n^{-\omega(1)}\quad\emph{and}\quad\mb{P}(\mc{H}_{t+1}|\mc{J}_{t+1}\cup \mc{F}_{t+1})\le \eps. \]
\end{claim}
\begin{proof}
By Lemma~\ref{lem:large-degree} at most a $2\eps^2$ fraction of vertices in $H$ are degree-bad. Since we add vertices in a random order back to form $S_j$ for $m\le j\le n-\ell$ the first result follows by Chernoff (see e.g.~Lemma~\ref{lem:chernoff}) for the hypergeometric distribution. The second follows by noting that $\mc{J}_{t+1}$ guarantees that at most an $\eps$ fraction of remaining vertices are degree-bad.
\end{proof}

We next have the following consequence of Lemma~\ref{lem:proj-anti-epoch1}.
\begin{claim}\label{clm:epoch-1-walk-down}
There is $C=C(d)>0$ such that for all $t\le n-\ell - (\log n)^{7/4}$, we have 
\[\mb{P}\big( X_{t+1} \le X_t - 1 \big| \mc{J}_{t+1},\mc{F}_{t+1},\mc{G}_{t+1},X_t> \lfloor n-\ell - t\rfloor/16\big) \ge 1- C(\log(1/\eps))^{-1/4}.\]
\end{claim}
\begin{proof}
By Claim~\ref{clm:epoch-1-degree}, we have 
\[\mb{P}\big( \mc{H}_{t+1} \big| \mc{J}_{t+1},\mc{F}_{t+1}\big)\le \eps.\]
Therefore it suffices to prove that 
\[\mb{P}(X_{t+1} \le X_t - 1|\mc{H}_{t+1}^c,\mc{F}_{t+1}',\mc{G}_{t+1},X_t> \lfloor n-\ell - t\rfloor/16) \le C(\log(1/\eps))^{-1/4}/2.\]
Note that $\mc{G}_{t+1}$ (and in particular $\mc{D}$) and $\mc{H}_{t+1}^c$ guarantees that $\snorm{B_{t+1}}_{\mr{HS}}^2\le 2dn$. If we have $\sigma_{t-X_t/2}(B_t)\ge \eps_{X_t}$, we instantaneously have $X_{t+1} = X_t -1$ as desired. Otherwise applying Lemma~\ref{lem:proj-anti-epoch1} implies that the fourth item defining $X_{t+1}$ in the first epoch holds with the desired probability bound. Note that $t\le n-\ell-(\log n)^{7/4}$ implies that $X_t\ge (\log n)^{7/4}/32$ due to the definition of the walk and therefore the necessary dimension lower bound to apply Lemma~\ref{lem:proj-anti-epoch1} holds. 
\end{proof}

In an analogous manner we have the following drift statement for the second epoch; the proof is analogous except we invoke Lemma~\ref{lem:proj-anti-epoch2} and keep track of events such as $\mc{U}^\ast$, hence the inclusion of $\mc{G}_{t+1}'$.
\begin{claim}\label{clm:epoch-2-walk-down}
There is $C=C(d)>0$ such that if $X_{n-\ell}\le 4(\log n)^{7/4}$ then for all $n-\ell\le t\le n-1$, we have 
\[\mb{P}\big(\{X_{t+1} \le X_t - 1 + \mbm{1}_{X_t = 0}\}\emph{ or } \mc{G}_{t+1}^{c} \emph{ or } \mc{G}_{t+1}'^{c} ~\big|~ \mc{F}_{t+1}\big) \ge 1- C(\log n)^{-1/4}.\]
\end{claim}

\noindent\textbf{Step 4: Proving that $X_n = 0$ whp.}
For $m\le t\le n-\ell$, we define
\[Y_t = X_t\cdot \mbm{1}\bigg[\bigcap_{j\le t} \mc{G}_{j} \cap \bigcap_{j\le t} \mc{J}_{j}\bigg].\]
By Claim~\ref{clm:epoch-1-walk-down}, we have 
\[\mb{P}\big( Y_{t+1} \le Y_t - 1 + \mbm{1}_{Y_t = 0}~\big| \mc{F}_{t+1}\big) \ge 1-C(\log(1/\eps))^{-1/4}\]
if $t\le n-\ell - (\log n)^{7/4}$ and $Y_t>\lfloor n-\ell - t\rfloor/16$. 

Let $Z_t = Y_t$ for $t\le n-\ell - (\log n)^{7/4}$ and $Z_{t+1} = Z_t - 1 + \mbm{1}_{Z_t = 0}$ for $n-\ell - (\log n)^{7/4}\le t\le n-\ell - 1$. By the random walk Lemma~\ref{lem:random-walk} we have 
\[\mb{P}\big(Z_{n-\ell}\ge (\log n)^{7/4}\big)\le n^{-\omega(1)}.\]
This implies that 
\[\mb{P}\big(Y_{n-\ell- (\log n)^{7/4}}\le 3(\log n)^{7/4}\big) \le n^{-\omega(1)}.\]
Note here that when we apply Lemma~\ref{lem:random-walk}, we require that $C(\log(1/\eps))^{-1/4}$ is smaller than an absolute constant which requires $\eps$ to be small as a function of $d$ only.

By Lemmas~\ref{lem:degree-event} and \ref{lem:expansion-epoch1}, and by Claim~\ref{clm:epoch-1-degree}, we have 
\[\mb{P}\big(Y_{n-\ell- (\log n)^{7/4}}\neq X_{n-\ell-(\log n)^{7/4}}\big) \le n^{-1+o(1)}.\]
Therefore
\[\mb{P}\big(X_{n-\ell}\le 4(\log n)^{7/4}\big)\ge 1-n^{-1+o(1)}.\] 

Next we define for $n-\ell\le t\le n$
\[Y_t' = X_t\cdot \mbm{1}\bigg[\bigcap_{n-\ell+1\le j\le t} \mc{G}_{j}\cap \bigcap_{n-\ell+1\le j\le t} \mc{G}_{j}'~\bigg].\]
By Claim~\ref{clm:epoch-2-walk-down} for $n-\ell\le t\le n-1$, we have 
\[\mb{P}\big(Y_{t+1}' \le Y_t' - 1 + \mbm{1}_{Y_t = 0}~\big|\mc{F}_t\big) \ge 1-C(\log n)^{-1/4}.\]
Given $Y_{n-\ell}'\le 4(\log n)^{7/4}$ at the start, which occurs with probability at least $1-n^{-1+o(1)}$, and recalling $\ell=\lfloor(\log n)^2\rfloor$, we have by Lemma~\ref{lem:random-walk} and Markov's inequality that
\[\mb{P}(Y_n'>0)\lesssim_d(\log n)^{-1/8}.\]
By Lemmas~\ref{lem:degree-event}, \ref{lem:expansion-epoch2}, and~\ref{lem:expansion-epoch2-spec} we find 
\[\mb{P}(Y_n'\neq X_n)\le (\log n)^{-\omega(1)}\]
and therefore 
\[\mb{P}(X_n \neq 0)\lesssim_d (\log n)^{-1/8}.\]
This (finally) completes the proof. 
\end{proof}
\begin{remark}\label{rem:quant}
The logarithmic probability guarantee is immediate from the proof given as well as noting that Lemma~\ref{lem:initial-step} can be made sufficiently quantitative. (This is essentially immediate since Lemma~\ref{lem:moment-convergence} holds with sufficiently strong probability.)

For an improved bound on the least singular value, notice that in the above argument we directly considered the product of last $\delta n$ singular values. Proving that $X_{n-\sqrt{n}}\le\sqrt{n}/4$ whp, one can see that the $(\sqrt{n}/4)$th smallest singular value of $B_{n-\sqrt{n}}$ is at least $\exp(-O(\sqrt{n}))$ with high probability. Iteratively applying Proposition~\ref{prop:walk-row-modified} with $\ell=k-1$ allows one to push the index of the unique singular value under consideration until it becomes the least singular value, at the cost of a factor $n^{-O(1)}$ each time. This gives the quality of bound mentioned in the remark following Theorem~\ref{thm:shifted-main}. It remains an interesting question whether the random walk approach used in this paper can prove a least singular value bound with probability quality $n^{-\Omega(1)}$ or with singular value size $\exp(-(\log n)^{O(1)})$. Such singular value estimates may prove useful when considering finer local laws.
\end{remark}

\section{The critical and subcritical case $d\leq 1$}\label{sec:graph-argument}
In this section we take care of the case of $d \leq 1$ which is substantially easier and completely different from the rest of this paper. In fact, we prove the following stronger result which tells us that in the case of $p\leq 1/n$ the limiting spectral distribution is simply a point mass at $0$. We remark that the case $d < 1 $ is also implicit in the work of Coste\cite{Cos23}.

\begin{lemma}\label{lem:subcritical-case}
For $d \leq 1$, let $p = d/n$ and let $A_n$ be a $n\times n$ matrix with iid entires distributed as $\Ber(p)$. Then $0$ is an eigenvalue of $A_n$ with multiplicity $n - o(n)$, with high probability. 
\end{lemma}

\subsection{A deterministic lemma} 

The motivation for the proof of Lemma~\ref{lem:subcritical-case} comes from the following simple deterministic lemma about matrices that are ``almost'' nilpotent.

\begin{lemma}\label{lem:matrix-fact}
For $k,\ell \in \N$, let $A$ be a matrix for which $A^k$ has $\ell$ non-zero rows. Then $A$ has at most $\ell$ non-zero eigenvalues when counted with multiplicity. 
\end{lemma}

To prove this we show that the Jordan blocks of $A$ must be supported on the coordinates corresponding to the non-zero rows of $A^k$.  

\begin{lemma}
For $k,\ell \in \N$, and $\l \not= 0$, let $A$ be a matrix for which all but the first $\ell$ rows of $A^k$ are $0$. Then all the generalized eigenvectors $v$ with eigenvalue $\l$ are supported on $[\ell]$.
\end{lemma}
\begin{proof} Let $v_1,\ldots, v_r$ be the generalized eigenvectors corresponding to a Jordan block with eigenvalue $\l$. This means that 
\begin{equation}\label{eq:gen-e-val}
 Av_1 = \l v_1 \quad \text{ and } \quad Av_s = \l v_s + v_{s-1}, \quad  \text{ for } 2 \leq s \leq r.\end{equation} We prove by induction on $r$ that $v_r$ is supported on $[\ell]$.  For the basis step simply note that $A^kv_1 = \lambda^k v_1$. Thus $v_1$ is supported on $[\ell]$. For general $v_s$, we iterate \eqref{eq:gen-e-val} to obtain 
\[ A^kv_s = \l^k v_s + c_{s-1}v_{s-1} + \cdots + c_1v_1, \]
for some $c_i \in \C$. By induction $v_{s-1}, v_{s-2}, \ldots, v_{1}$, are each supported on $[\ell]$. Also $A^kv_s$ is supported on $[\ell]$. Thus we must have that $v_s$ is also supported on $[\ell]$, since all but the first $\ell$ rows of $A^k$ are zero. This completes the induction. 
\end{proof}

\begin{proof}[Proof of Lemma~\ref{lem:matrix-fact}] Without loss, assume that the first $\ell$ rows are non-zero in $A^k$. Let $\l \not=0$. By the previous lemma we have that all of the generalized eigenvectors of $\l$ are supported on $[\ell]$. By linear independence of the Jordon blocks it follows that $\ell \geq \sum_{\l \not= 0} a_{\lambda}$, where $a_{\l}$ is the algebraic multiplicity of lambda. 
\end{proof}

\subsection{A probabilistic preliminary}

To study our matrix we need a basic fact about critical and subcritical branching processes. Of course, much more general results could be stated here, but we state only what we need. Here we let $X_k$ be a branching process with offspring distribution $\on{Bin}(n,p)$. It is easy to check that 
$\EE\, X_k = (pn)^k$ and thus if $pn \leq 1$ we have
\begin{equation}\label{eq:GW-size-bound} \PP\big( X_1 + \cdots + X_k  \geq k^2 \big) \leq 1/k , \end{equation} by Markov. We also need the following basic fact about the extinction time of this (sub)critical process. Recall the extinction time $\tau$ is the minimum $t$ for which $X_{t} = 0$. 
This follows from ``Kolmogorov's Estimate''; see Theorem 12.7 in \cite{LyonsPeresBook}. 

\begin{fact}\label{fact:branching-processes}
For $d \leq 1$, let $p= d/n$ and let $X_k$ be a branching process with with offspring distribution $\Ber(n,p)$ and extinction time $\tau$. 
If $n$ is sufficiently large, then for each $k \in \N$ we have $\PP( \tau > k ) = O(1/k)$.
\end{fact}

\subsection{Completing the proof of Lemma~\ref{lem:subcritical-case}}

We now turn to think of our matrix as a directed graph $G \sim D(n,p)$,
where $A_{i,j} = 1$ if and only if $(i,j)$ is an edge of $G$. To show that $G$ has the nilpotent property in Lemma~\ref{lem:matrix-fact}, we note that a row $i$ of $A^k$ is $0$ if and only if the vertex $i$, in the corresponding directed graph, has no walk of length $k$ starting at it.
To be precise, a \emph{walk of length} $k$, starting at a vertex $v$ in a graph $G$, is a sequence of (not necessarily distinct) vertices $v = v_0,\ldots,v_{k} $ where the edges $(v_0,v_1),\ldots, (v_{k-1},v_{k})$ are in $G$.

\begin{lemma}\label{lem:predecessors-vertex} For $n \in \N$, let $p \leq 1/n$, $k \leq n^{1/5} $ and let $v$ be a vertex in $D(n,p)$. Then,
\[\PP\big(\text{There is a directed walk of length $k$ starting at v}\hspace{0.5mm} \big) = O(1/k). \] 
\end{lemma}
\begin{proof}
We explore the vertices that can be reached from $v$ along a directed path, in a depth-first search, while \emph{only} exposing forward edges to unexplored vertices. In particular, we never expose edges between vertices that we have ``reached'' already. Let $S$ be the set of vertices reachable from $v$ in this manner and let $S_k$ be the vertices that were discovered at distance $k$ from $v$. 

Now let $X_k$ be a branching process with offspring distribution $\on{Bin}(n,p)$. Carefully note that $|S_k|$ is naturally coupled with $X_k$ so that  $|S_k| \leq X_k$, for all $k$. Let $\tau$ be the extinction time of this branching process.

Let $\prec$ be the linear ordering on $S$ defined by the ordering we discovered vertices in $S$. In the above exposure process, by definition, we never expose an edge $(u,w)$ where $w \prec u$ between two vertices that have already been reached. Define the random variable 
\[ R = |\{ (u,v) \in S^2 : v \prec u \}| \] to be the number of ``backward'' edges in $S$. Note this is well-defined since $S$ only depends on ``forward'' edges.  

Now, if there is a path of length $k$ starting at $v$, either this path is  \emph{increasing} or \emph{not} increasing, in the order $\prec$. In the former case we must have $\tau \geq k$. In the latter, we must have a backwards edge so that $R \geq 1$. Let $\cP_k$ be the event that $v$ has a walk of length $k$ starting at $v$. We have, 
\[\PP\big( \cP_k \big) \leq  \PP\big(  R \geq 1 \text{ and }  \tau < k   \big) + \PP( \tau \geq k ) \leq \PP\big(  R\geq 1 \text{ and } \tau < k  \big) + O(1/k), \] where we used Fact~\ref{fact:branching-processes}. We bound $\PP( \{R\geq 1\} \cap \{ \tau < k \} )$ by writing 
\begin{equation}\label{eq:split2}  \PP( R\geq 1 \text{ and } \tau < k  ) \leq \PP\big( \tau < k  \text{ and }  |S| > k^2  \big) + \PP\big( R \geq 1 \text{ and }  |S| \leq k^2  \big)\end{equation}
and then writing 
\[ \PP\big( \tau < k  \text{ and }  |S| > k^2  \big) \leq \PP\big( X_1 + \cdots + X_{k-1} > k^2 \big) = O(1/k), \]
where the last bound holds by \eqref{eq:GW-size-bound}. To take care of the second term in \eqref{eq:split2}, we first note that for given $S$ with $|S| \leq k^2$ then $S$ induces at most $k^4$ edges, and 
\[ \EE\big( R \, \big|\, S\big) \leq k^4p \leq 1/k, \]
since $k \leq n^{1/5}$, and $S$ only depends on the forward edges in the exposure. Thus
\[ \PP\big( R\geq 1 \text{ and } |S| \leq k^2  \big) \leq \max_{S : |S| \leq k^2, \tau < k}  
\PP\big( R \geq 1 \, \big|\,  S \big)  \leq  \max_{S : |S| \leq k^2, \tau < k}  \EE\big( R \, \big|\, S\big) \leq 1/k. \]\\
Putting the above together, gives $\PP(\cP_k) = O(1/k)$, as desired.
\end{proof}

\vspace{2mm}

We can now easily prove Lemma~\ref{lem:subcritical-case}, which is the goal of this section. 

\begin{proof}[Proof of Lemma~\ref{lem:subcritical-case}]
Let $W_k$ be the number of vertices in $G$ that have a directed walk of length $k$ starting at $v$. By Lemma~\ref{lem:predecessors-vertex}, we have $\EE\, W_k = O(n/k) $ and 
thus $\PP( W_k \geq n/\sqrt{k} ) = O(k^{-1/2})$, by Markov. Thus the matrix $A^k$ has at most $O(n/\sqrt{k})$ non-zero rows whp. Thus if we take $k = \lceil \log n \rceil$ and apply Lemma~\ref{lem:matrix-fact} to $A^k$, we learn that $A$ has at most $o(n)$ non-zero eigenvalues, counting multiplicity, whp. 
\end{proof}

\section{Proof of Theorem~\ref{thm:main}}\label{sec:proof}

In this short section we prove our main theorem, when $d >1$, thus completing the proof of Theorem~\ref{thm:main}. We do this by putting together Lemma~\ref{lem:crucial}, on the bottom window of the singular values, together with a convergence criterion for the spectral measure. 

We begin this section by stating the criterion we use from the work of Bordenave and Chafa\"{\i} \cite[Lemma~4.3, Remark~4.4]{BC12}. (We remark that related criteria appear in work of Tao and Vu \cite{TV10}.) Recall that if $M$ is an $n\times n$ matrix, then spectral and the singular value measures are defined to be
\[\mu_M = \frac{1}{n}\sum_{i=1}^n \delta_{\lambda_i} \qquad \text{ and } \qquad \nu_M = \frac{1}{n}\sum_{i=1}^n\delta_{\sigma_i},\]
where the $\l_i$ are the eigenvalues of $M$ and the s$\sigma_i$ are the singular values.  We also recall that if $\mu_n$ are random measures and $\mu$ is a deterministic measure we use the notation $\mu_{n}\rightsquigarrow\mu$ to mean that $\mu_n$ converges in probability to $\mu$.

\begin{prop}\label{prop:unicity}
For each $n$, let $M_n$ be a $n\times n$ random matrix and, for each $z \in \C$, let $\nu_z$ be a deterministic probability measure on $\mb{R}^{+}$. Assume that for almost all $z \in \C$ we have 
\[ \nu_{M_n-zI}\rightsquigarrow \nu_{z}  \]
and for almost all $z \in \C$ and all $\eps>0$ we have 
\[ \lim_{t\to+\infty}\limsup_{n\to \infty}\mb{P}\bigg(\int_{|\log u|\ge t}|\log u|~d\nu_{M_n-zI}(u) > \eps\bigg) = 0.\]
Then there exists a probability measure $\mu$ on $\mb{C}$ such that 
\[\mu_{M_n}\rightsquigarrow\mu.\]
\end{prop}

We now prove Theorem~\ref{thm:main} using our key lemma, Lemma~\ref{lem:crucial} along along with the above criteria. 

\begin{proof}[Proof of Theorem~\ref{thm:main}]
First note that we may assume that $d>1$, since the case $d\leq 1$ is taken care of by Lemma~\ref{lem:subcritical-case}.

Let $A$ and $B$ be as in Section~\ref{sec:setup}. For all $z\in \mb{C}$, by Lemma~\ref{lem:converge-singular}, Lemma~\ref{lem:TV-estimate}, and noting that permutation matrices are unitary it follows that $\nu_{A-zI}\rightsquigarrow \nu_{|z|}$. 

We now verify the crucial uniform integrability condition in Proposition~\ref{prop:unicity}. Fix $\eps = \mb{P}[\on{Pois}(d)\ge \Delta]> 0$. For almost all $z\in \mb{C}$, by Lemma~\ref{lem:crucial} there is $\delta = \delta(\eps,z) > 0$ such that 
\[\lim_{n\to\infty} \mb{P}\bigg(\frac{1}{n}\sum_{j=0}^{\delta n}\log(\sigma_{n-j}(B-zI)) \le -C\eps \bigg) = 0.\]
By the strong law of large numbers, whp $\snorm{B}_{\mr{HS}}^2 \le 2dn$. Thus $\sigma_{n/2}(B-zI)^2\le (\sigma_{n/2}(B)+|z|)^2\le 2 \sigma_{n/2}(B)^2 + 2|z|^2\le 2 \cdot (2/n) \snorm{B}_{\mr{HS}}^2 + 2|z|^2\le 8(dn + |z|^2)$. Therefore we have
\[\lim_{n\to\infty}\mb{P}\bigg(\frac{1}{n}\sum_{j=1}^{n}\log(\s_j(B-zI))\mbm{1}_{\sigma_j(B-zI)\le \exp(-2C\eps/\delta)}\le -C\eps -\delta \cdot \log(8(d +|z|^2))\bigg) = 0;\]
this is the crucial estimate controlling the lower tail. For the large values of the logarithm, note that $(\log x)\mbm{1}_{x\ge T}\le\frac{x^2}{T}$ for $T\ge 1$. Therefore if $\snorm{B}_{\mr{HS}}^2 \le 2dn$ then 
\begin{align*}
\frac{1}{n}\sum_{j=1}^{n}\log(\s_j(B-zI)) \mbm{1}_{\s_j(B-zI)\ge \eps^{-1}(d + |z|^2)} &\le \frac{1}{n(\eps^{-1}(d + |z|^2))}\sum_{j=1}^{n}\s_j(B-zI)^2\\
&\le \frac{2(\snorm{B}_{\mr{HS}}^2 + |z|^2n)}{n(\eps^{-1}(d + |z|^2))}\le 4\eps.
\end{align*}
Therefore for almost all $z\in \mb{C}$, there exists $T = T(z,\eps) > 0$ such that 
\[\lim_{n\to\infty}\mb{P}\bigg(\frac{1}{n}\sum_{j=1}^{n}|\log(\s_j(B-zI))|\mbm{1}_{|\log(\s_j(B-zI))|\ge T}\ge (C+4)\eps + \delta\log(4d + |z|^2)\bigg) = 0.\]
By Lemma~\ref{lem:TV-estimate}, we therefore have 
\[\lim_{n\to\infty}\mb{P}\bigg(\frac{1}{n}\sum_{j=1}^{n}|\log(\s_j(A-zI))|\mbm{1}_{|\log(\s_j(A-zI))|\ge T}\ge (C+4)\eps + \delta\log(4d + |z|^2)\bigg) = 0.\]
Taking the countable sequence of possible $\eps$ tending to $0$ given by taking $\Delta\to\infty$ and recalling $\delta(\eps,z)\to 0$ as $\eps\to 0$ verifies the second condition of Proposition~\ref{prop:unicity}. The result follows.
\end{proof}

\section*{Acknowledgements}
This project was carried out, in part, while MS was visiting Cambridge in the 2022-2023 academic year. MS would like to thank the Cambridge combinatorics community for creating a pleasant working environment during this time. We would also like to thank Marcus Michelen for some comments on a draft of this paper. 

\appendix

\section{Various preliminaries}\label{app:preliminaries}
\subsection{Directed graphs, bipartite graphs, and matrices}\label{sub:digraph-notation}
Throughout this proof, we will pass freely between the notions of matrices, bipartite graphs, and digraphs. Given an $m\times n$ matrix $M$ with entries in $\{0,1\}$, we may identify this bipartite graph with $n$ vertices on the left and $m$ vertices on the right and an edge between $j\in[n]$ and $i\in[m]$ if and only if $M_{ij} = 1$. If $m\le n$, we may identify $M$ with a digraph by adding $n-m$ empty right vertices, directing all edges from left to right, and gluing corresponding vertices in the obvious manner. We write $\deg_M^+(v,S)$ for the number of out-neighbors $v$ has in $S$ (including a self-loop), and similar for $\deg_M^-(v,S)$; we drop $S$ to refer to the total out- or in-degree. The \emph{degree sequence} of $M$ (or equivalently, of the corresponding bipartite graph or digraph) is $(\mbf d,\mbf d')$ where $\mbf d=(\deg_M^+(v))_{v\in[n]}$ and $\mbf d'=(\deg_M^-(v))_{v\in[m]}$.

For the entire paper we will concern ourselves with matrices such that $n\in\{m,m+1\}$. Let $I_{m\times n}$ denote the matrix such that $(I_{m\times n})_{ij} = 1$ if $i = j\le\min(m,n)$ and $0$ otherwise; this aligns with the standard definition of the identity matrix for square matrices. The matrix $M-zI_{m\times n}$ can be identified as a weighted bipartite graph with $n$ vertices on the left and $m$ vertices on the right (with possible weights of $0,1,-z,1-z$).

\subsection{Concentration inequalities}\label{sub:concentration}
We state a Chernoff bound for binomial and hypergeometric distributions (see for example \cite[Theorems~2.1,~2.10]{JLR00}).

\begin{lemma}[Chernoff bound]\label{lem:chernoff}
Let $X$ be either:
\begin{itemize}
    \item a sum of independent random variables, each of which take values in $\{0,1\}$, or
    \item hypergeometrically distributed (with any parameters).
\end{itemize}
Then for any $\delta>0$ we have
\[\Pr[X\le (1-\delta)\mb{E}X]\le\exp(-\delta^2\mb{E}X/2),\qquad\Pr[X\ge (1+\delta)\mb{E}X]\le\exp(-\delta^2\mb{E}X/(2+\delta)).\]
\end{lemma}

We will require a version of the classical Bernstein inequality; this appears as \cite[Theorem~2.8.1]{Ver18}.
\begin{theorem}\label{thm:bernstein}
For a random variable $X$ define the $\psi_1$-norm
\[\snorm{X}_{\psi_1}=\inf\{t>0\colon\mb{E}[\exp(|X|/t)]\le 2\}.\]
There is an absolute constant $c > 0$ such that the following holds. If $X_1,\ldots,X_N$ are independent random variables then 
\[\mb{P}\bigg(\bigg|\sum_{i=1}^NX_i\bigg|\ge t\bigg)\le 2\exp\bigg(-c\min\bigg(\frac{t^2}{\sum_{i=1}^N\snorm{X_i}_{\psi_1}^2},\frac{t}{\max_i\snorm{X_i}_{\psi_1}}\bigg)\bigg)\]
for all $t\ge 0$.
\end{theorem}

We will also require a standard concentration inequality for Lipschitz functions with respect to the symmetric group (and with respect to injections); this appears as \cite[Lemma~3.3]{FKSS23}.

\begin{lemma}\label{lem:injection-concentration}
Let $m\in\mb{N}$, let $S$ be a finite set with $|S| \ge m$, let $\mc{F}$ be the set of functions $\{1,\ldots,m\}\to S$ and let $\mc{I}\subseteq\mc{F}$ be the set of injections $\{1,\ldots,m\}\to S$. Consider a function $f\colon\mc{F}\to\mb{R}$ with the property $|f(\pi)-f(\pi')|\le \sum_{i=1}^mc_i\mbm{1}_{\pi(i)\neq\pi'(i)}$.
Let $\pi\in\mathcal{I}$ be a uniformly random injection. Then for $t\ge 0$,
\[\mb{P}(|f(\pi)-\mb{E}f(\pi)|\ge t)\le 2\exp\bigg(-\frac{t^2}{8\sum_{i=1}^mc_i^2}\bigg).
\]
\end{lemma}

Finally we will require the Azuma--Hoeffding inequality (see \cite[Theorem~2.25]{JLR00}).
\begin{lemma}[Azuma--Hoeffding inequality]\label{lem:azuma}
Let $X_0, \ldots, X_n$ form a martingale sequence such that $|X_k-X_{k-1}|\le c_k$ almost surely. Then 
\[\mb{P}(|X_0-X_n|\ge t)\le 2\exp\bigg(-\frac{t^2}{2\sum_{k=1}^nc_k^2}\bigg)\]
\end{lemma}

\subsection{Configuration model}\label{sub:config-model}
We will also require the definition of the configuration model for a bipartite graph. 
\begin{definition}\label{def:config-model}
Consider a pair of degree sequences $\mbf d=(d_1,\ldots,d_{n})$ and $\mbf d'=(d_1',\ldots,d_{m}')$ such that $\sum d_i = \sum d_i'$. Consider a set of $r=\sum d_i + \sum d_i'$ ``stubs'', $n$ left buckets, and $m$ right buckets. Assign $d_i$ stubs to the $i$th left bucket and $d_i'$ stubs to the $i$th right bucket. A \emph{configuration} is a perfect matching between the $r/2$ stubs assigned to the left buckets and $r/2$ stubs assigned to right buckets.  Given a configuration, contracting each of the buckets to a single vertex gives rise to a bipartite multigraph with degree sequence $d_1,\ldots,d_n$ in the left and $d_1',\ldots,d_m'$ in the right.

A random bipartite graph $G$ drawn from the \emph{configuration model} with degree sequences $\mb{G}(\mbf d, \mbf d')$ is the bipartite multigraph arising from choosing the perfect matching between the left and right stubs uniformly at random.
\end{definition}

Note that we may implicitly identify vertices on the left and right by identifying the $i$th vertex on the left and $i$th vertex on the right to obtain a digraph as in Appendix~\ref{sub:digraph-notation}. We have the following fact regarding the configuration model; the first is obvious by construction while the second is an immediate consequence of the results of Janson \cite{Jan14} (although many earlier results e.g. \cite{BC78,Bol80,MW91} would suffice). 

\begin{lemma}\label{lem:config}
Sample $G\sim \mb{G}(\mbf d, \mbf d')$. 
\begin{itemize}
    \item Conditioned on being simple, $G$ is a uniformly random bipartite graph with degree sequence $\mbf d$ on the left and $\mbf d'$ on the right.
    \item If $d_i,d_i'$ are positive integers and $\sum d_i^2 + \sum {d_i'}^2\le C(n + m)$ and $n/C\le m\le Cn$, we have that $G$ is simple with probability $\Omega_{C}(1)$.
\end{itemize}  
\end{lemma}

\bibliographystyle{amsplain0}
\bibliography{main}

\end{document}